\documentclass{amsart}

\usepackage[style=authoryear,
 backref=true,
 maxnames=4,
 maxbibnames=99,
 uniquename=false,
 firstinits=true
]{biblatex}
\let\cite\parencite
\DeclareNameAlias{sortname}{last-first}

\usepackage{amssymb,amsmath,stmaryrd,mathrsfs}
\usepackage[all,2cell]{xy}
\usepackage{enumitem}
\usepackage{color}
\definecolor{darkgreen}{rgb}{0,0.45,0} 
\usepackage[colorlinks,citecolor=darkgreen,linkcolor=darkgreen]{hyperref}
\usepackage{mathtools}          
\usepackage{braket}             
\let\setof\Set
\usepackage{url}                
\usepackage{xspace}             
\usepackage{mathpartir}         

\makeatletter
\let\ea\expandafter

\def\mdef#1#2{\ea\ea\ea\gdef\ea\ea\noexpand#1\ea{\ea\ensuremath\ea{#2}\xspace}}
\def\alwaysmath#1{\ea\ea\ea\global\ea\ea\ea\let\ea\ea\csname your@#1\endcsname\csname #1\endcsname
  \ea\def\csname #1\endcsname{\ensuremath{\csname your@#1\endcsname}\xspace}}

\newcount\foreachcount

\def\foreachletter#1#2#3{\foreachcount=#1
  \ea\loop\ea\ea\ea#3\@alph\foreachcount
  \advance\foreachcount by 1
  \ifnum\foreachcount<#2\repeat}

\def\foreachLetter#1#2#3{\foreachcount=#1
  \ea\loop\ea\ea\ea#3\@Alph\foreachcount
  \advance\foreachcount by 1
  \ifnum\foreachcount<#2\repeat}

\def\definescr#1{\ea\gdef\csname s#1\endcsname{\ensuremath{\mathscr{#1}}\xspace}}
\foreachLetter{1}{27}{\definescr}
\def\definecal#1{\ea\gdef\csname c#1\endcsname{\ensuremath{\mathcal{#1}}\xspace}}
\foreachLetter{1}{27}{\definecal}
\def\definetil#1{\ea\gdef\csname #1til\endcsname{\ensuremath{\widetilde{#1}}\xspace}}
\foreachLetter{1}{27}{\definetil}
\foreachletter{1}{27}{\definetil}
\def\definehat#1{\ea\gdef\csname #1hat\endcsname{\ensuremath{\widehat{#1}}\xspace}}
\foreachLetter{1}{27}{\definehat}
\foreachletter{1}{27}{\definehat}

\DeclareSymbolFont{bbold}{U}{bbold}{m}{n}
\DeclareSymbolFontAlphabet{\mathbbb}{bbold}
\newcommand{\bbtwo}{\ensuremath{\mathbbb{2}}\xspace}

\newcommand{\op}{^{\mathrm{op}}}
\let\dn\downarrow

\SelectTips{cm}{}
\newdir{ >}{{}*!/-10pt/@{>}}    

\newcommand{\pullbackcorner}[1][dr]{\save*!/#1-1.2pc/#1:(-1,1)@^{|-}\restore}

\alwaysmath{ell}
\alwaysmath{infty}
\alwaysmath{odot}
\mdef\idfunc{\mathrm{id}}

\newcommand{\too}[1][]{\ensuremath{\overset{#1}{\longrightarrow}}}
\newcommand{\ot}{\ensuremath{\leftarrow}}

\let\toto\rightrightarrows
\let\into\hookrightarrow

\mdef\we{\overset{\sim}{\longrightarrow}}
\mdef\leftwe{\overset{\sim}{\longleftarrow}}

\let\fib\twoheadrightarrow

\def\acof{\mathrel{\mathrlap{\hspace{3pt}\raisebox{4pt}{$\scriptscriptstyle\sim$}}\mathord{\rightarrowtail}}}

\let\xto\xrightarrow

\def\rightarrowtailfill@{\arrowfill@{\Yright\joinrel\relbar}\relbar\rightarrow}
\newcommand\xrightarrowtail[2][]{\ext@arrow 0055{\rightarrowtailfill@}{#1}{#2}}

\def\twoheadrightarrowfill@{\arrowfill@{\relbar\joinrel\relbar}\relbar\twoheadrightarrow}
\newcommand\xtwoheadrightarrow[2][]{\ext@arrow 0055{\twoheadrightarrowfill@}{#1}{#2}}

\let\xfib\xtwoheadrightarrow

\def\toiso{\xto{\smash{\raisebox{-.5mm}{$\scriptstyle\sim$}}}}

\def\defthm#1#2{%
  \newtheorem{#1}{#2}[section]%
  \expandafter\def\csname #1autorefname\endcsname{#2}%
  \expandafter\let\csname c@#1\endcsname\c@thm}

\newtheorem{thm}{Theorem}[section]

\defthm{cor}{Corollary}
\defthm{prop}{Proposition}
\defthm{lem}{Lemma}
\theoremstyle{definition}
\defthm{defn}{Definition}
\defthm{notn}{Notation}
\theoremstyle{remark}
\defthm{rmk}{Remark}
\defthm{eg}{Example}
\defthm{egs}{Examples}

\let\c@equation\c@thm
\numberwithin{equation}{section}

\@ifpackageloaded{mathtools}{\mathtoolsset{showonlyrefs,showmanualtags}}{}

\alwaysmath{alpha}
\alwaysmath{beta}
\alwaysmath{gamma}
\alwaysmath{Gamma}
\alwaysmath{delta}
\alwaysmath{Delta}
\alwaysmath{epsilon}
\mdef\ep{\varepsilon}
\alwaysmath{zeta}
\alwaysmath{eta}
\alwaysmath{theta}
\alwaysmath{Theta}
\alwaysmath{iota}
\alwaysmath{kappa}
\alwaysmath{lambda}
\alwaysmath{Lambda}
\alwaysmath{mu}
\alwaysmath{nu}
\alwaysmath{xi}
\alwaysmath{pi}
\alwaysmath{rho}
\alwaysmath{sigma}
\alwaysmath{Sigma}
\alwaysmath{tau}
\alwaysmath{upsilon}
\alwaysmath{Upsilon}
\alwaysmath{phi}
\alwaysmath{Pi}
\alwaysmath{Phi}
\mdef\ph{\varphi}
\alwaysmath{chi}
\alwaysmath{psi}
\alwaysmath{Psi}
\alwaysmath{omega}
\alwaysmath{Omega}
\let\al\alpha

\let\Gm\Gamma

\let\De\Delta

\let\Si\Sigma
\let\om\omega
\let\ka\kappa
\let\la\lambda

\let\Th\Theta

\makeatother

\mdef\iscontr{\mathsf{isContr}}
\mdef\isprop{\mathsf{isProp}}
\newcommand{\istrunc}[1]{\mathsf{is}\text-#1\text-\mathsf{Trunc}}
\mdef\isequiv{\mathsf{isEquiv}}
\mdef\ishiso{\mathsf{ishIso}}
\let\eq\equiv
\mdef\equiv{\mathsf{Equiv}}
\mdef\idpath{\mathsf{idpath}}
\mdef\hequiv{\mathsf{hEquiv}}
\mdef\hfiber{\mathsf{hFiber}}
\mdef\type{\mathsf{Type}}
\mdef\etoi{\mathsf{etoi}}
\mdef\itoe{\mathsf{itoe}}
\mdef\fibeqv{\mathsf{fibeqv}}
\mdef\map{\mathsf{ap}}
\let\ap\map
\mdef\happly{\mathsf{happly}}
\mdef\funext{\mathsf{funext}}
\mdef\idequiv{\mathsf{idequiv}}
\mdef\ptoe{\mathsf{pathToEquiv}}
\mdef\fst{\mathsf{fst}}
\mdef\snd{\mathsf{snd}}
\mdef\io{(\infty,1)}
\def\sprod#1{{\textstyle\prod_{#1}}}
\def\ssum#1{{\textstyle\sum_{#1}}}
\def\sSet{\ensuremath{\mathrm{sSet}}\xspace}
\def\nSet{\ensuremath{\mathrm{Set}}\xspace}
\def\nTTFC{\ensuremath{\mathrm{TTFC}}\xspace}
\def\nCAT{\ensuremath{\mathrm{CAT}}\xspace}
\def\nGpd{\ensuremath{\mathrm{Gpd}}\xspace}
\def\pr{\;\vdash\;}
\def\f{_{\mathbf{f}}}
\def\m#1{\llbracket#1\rrbracket}
\mdef\ty{\;\mathsf{type}}
\mdef\r{\mathsf{r}}
\mdef\el{\mathsf{El}}
\mdef\up{\mathsf{up}}
\mdef\Id{\mathsf{Id}}
\newcommand{\oplim}[2]{\llbracket #1,#2 \rrbracket}
\newcommand{\shnno}{\textsc{shnno}\xspace}
\mdef\lntil{\widetilde{\lN}}
\mdef\ctf{(\sC^\bbtwo)\f}
\def\id{\leadsto}

\setitemize[1]{leftmargin=1.5em}
\setenumerate[1]{leftmargin=*}

\begin{document}

\title[Univalence for inverse diagrams]{Univalence for inverse diagrams and homotopy canonicity}
\author[Michael Shulman]{Michael Shulman}
\thanks{This material is based upon work supported by the National Science Foundation under a postdoctoral fellowship and agreement No. DMS-1128155.  Any opinions, findings, and conclusions or recommendations expressed in this material are those of the author and do not necessarily reflect the views of the National Science Foundation.}

\maketitle
\begin{abstract}
  We describe a homotopical version of the relational and gluing models of type theory, and generalize it to inverse diagrams and oplax limits.
  Our method uses the Reedy homotopy theory on inverse diagrams, and relies on the fact that Reedy fibrant diagrams correspond to contexts of a certain shape in type theory.
  This has two main applications.
  First, by considering inverse diagrams in Voevodsky's univalent model in simplicial sets, we obtain new models of univalence in a number of \io-toposes; this answers a question raised at the Oberwolfach workshop on homotopical type theory.
  Second, by gluing the syntactic category of univalent type theory along its global sections functor to groupoids, we obtain a partial answer to Voevodsky's homotopy-canonicity conjecture: in 1-truncated type theory with one univalent universe of sets, any closed term of natural number type is homotopic to a numeral.
\end{abstract}

\setcounter{tocdepth}{1}
\tableofcontents

\section{Introduction}
\label{sec:introduction}

Recently it has become apparent that Martin-L\"of's \emph{intensional type theory} admits semantics in \emph{homotopy theory}~\parencite{hs:gpd-typethy,warren:thesis,aw:htpy-idtype,gb:topsimpid,voevodsky:typesystems,klv:ssetmodel,lw:localuniv}.
The basic idea is that intensional \emph{identity types} are interpreted by \emph{path spaces}.
Since there can be nontrivial paths even from a point to itself, these models make a virtue out of the failure of ``uniqueness of identity proofs''. One may conclude that intensional type theory is naturally a theory of ``homotopy types'', and many of its traditionally uncomfortable attributes come from trying to force it to be a theory only of sets.
This raises the possibility of using intensional type theory as a ``natively homotopical'' foundation for mathematics.

One of the innovations of homotopical type theory, due to Voevodsky, is the identification of the correct identity types for universes.
It is natural to consider two types ``equal'', as terms belonging to a universe, if there is an isomorphism between them.
However, this is hard to square with uniqueness of identity proofs, since two types can be isomorphic in more than one way, and if the equality between them doesn't remember which isomorphism it came from, how can we meaningfully substitute along that equality?
But homotopically, taking isomorphisms (or, more precisely, equivalences) to form the identity type of the universe makes perfect sense; the resulting rule is called the \emph{univalence axiom}.

Since its introduction, much research has centered around this axiom, and it has proven quite valuable for formalizing mathematics and homotopy theory in type theory.
However, important meta-theoretical questions remain, such as:
\begin{enumerate}[label=(\alph*)]
\item What are its categorical semantics?\label{item:catsem}
\item What are its logical consequences?\label{item:logcsq}
\item How does it impact the computational behavior of type theory?\label{item:cmpbvr}
\end{enumerate}
Until now, essentially the only known model of univalence (aside from syntactic ones) has been the one constructed by Voevodsky~\parencite{klv:ssetmodel} in simplicial sets, and the question was raised at the Oberwolfach mini-workshop~\parencite{oberwolfach:hott} of whether such models exist.
In this paper we will describe a general class of constructions on models of type theory, and show that they preserve univalence.
Besides answering this question, these models have further important implications for all three questions above.

The simplest example of the constructions we will describe is that if \sC is a categorical model of univalence, then so is the category $\sC^\bbtwo$ of arrows in \sC.
This already has nontrivial consequences.
For instance, Voevodsky's model in the category \sSet of simplicial sets takes place in a classical metatheory, and hence satisfies the law of excluded middle (appropriately formulated), while our model in $\sSet^\bbtwo$ does not.
Thus, univalence does not imply excluded middle, which seems not to have been known previously.

We can say more than this, however: if $I$ is any \emph{inverse category}, then the functor category $\sC^I$ inherits a model of type theory with univalence from \sC.
An inverse category is one containing no infinite composable strings
\[ \to\;\to\;\to\;\to\;\cdots \]
of nonidentity morphisms.
For instance, a finite category is inverse just when it is skeletal and has no nonidentity endomorphisms.
This property enables us to construct diagrams by well-founded induction, which we exploit to build a univalent universe.
Collectively, the internal logics of the categories $\sSet^I$ suffice to violate any propositional statement that is not an intuitionistic tautology; thus univalence ``has no non-constructive implications for propositional logic''.

From a higher categorical point of view, the model category $\sSet^I$ is a presentation of the \io-topos $\infty \mathrm{Gpd}^I$.
Thus, we may say that this \io-topos admits univalent type theory as an ``internal language'', analogously to how ordinary 1-toposes admit extensional type theory as an internal language.
Since the univalence axiom is closely analogous to Lurie-Rezk \emph{object classifiers} (see~\textcite[\S6.1.6]{lurie:higher-topoi} and~\textcite{gk:univlcc}), it is natural to conjecture that \emph{all} \io-toposes admit univalent type theory as an internal language.
This is morally true, but coherence questions remain to be resolved, since the type theories in common use are stricter than \io-category theory.
From this perspective, the contribution of this paper is to resolve the coherence problem in this special case.
(An alternative would be to weaken type theory so as to match \io-category theory better.)

Our construction has an additional advantage, however: it generalizes further to the case of \emph{oplax limits} of diagrams of models indexed by an inverse category.
The simplest case of this which goes beyond functor categories is the \emph{gluing construction} along a functor between two models.
The gluing construction of the ``global sections'' functor is called the \emph{scone} (Sierpinski cone).
It is well-known that scones can be used to prove canonicity and parametricity results about type theories, by an argument of Peter Freyd; the same is true here.

Specifically, by gluing along a groupoid-valued global sections functor of a syntactic category, we can give a partial answer to the \emph{homotopy canonicity} conjecture of Voevodsky.
(Essentially the same gluing construction was considered by~\textcite{hw:crmtt}, but without univalence.)
We show that in type theory with a 1-truncation axiom (so every type is homotopically at most a 1-type) and one univalent universe of sets (0-truncated types), every closed term of natural number type is provably homotopic to a numeral.
Thus, although the univalence axiom (like any axiom) destroys the direct computational content of type theory, it preserves it ``up to homotopy''.%
\footnote{This is not the case for most axioms that might be added to type theory.
  For instance, the axiom of excluded middle yields terms like ``$0$ if the G\"odel sentence is true and $1$ if it is false'' that are not provably equal to any specific numeral.
  This is not so important when using type theory as a basis for classical mathematics, but canonicity is an essential property when using type theory as a programming language.}

Our partial answer to this conjecture is very similar to that of~\textcite{lh:canonicity}, who also study a 1-truncated type theory with one univalent universe of sets.
They describe instead a modified version of this type theory with stricter equality rules, under which univalence is true by definition rather than being an axiom, and show that in this theory every closed term of natural number type is judgmentally (i.e.\ strictly) \emph{equal} to a numeral.
Thus, their answer gives a stronger result, but only in a stronger theory.
Both methods should in principle extend to multiple univalent universes with no truncation hypotheses; the problem in both cases relates to constructing a sufficiently computational ``higher groupoid model'' of type theory.

Scones and more general gluing constructions can also be used to prove \emph{parametricity} theorems, which say that any definable term having a given type must automatically satisfy some theorem derived from that type.
(This is a category-theoretic formulation of the method of ``logical relations''.)
For instance, any term with the type $\prod_{X:\type} (X\to X)$ must be indistinguishable from the polymorphic identity function.
We will not pursue this here, however.

Finally, our constructions can also be interpreted as a ``stability'' result for categories that model univalence.
Probably they can even be performed internally inside of type theory.
This has implications for a hypothetical definition of ``elementary \io-topos''.

\subsection*{Organization}

We begin in \S\ref{sec:ttfc} by defining the basic categorical structures which corresponds to the type-theoretic operations we will consider: dependent sums, dependent products, identity types, and (sometimes) the natural numbers.
We call categories with all of this structure \emph{type-theoretic fibration categories}, since they are a special sort of the ``fibration categories'' and ``categories of fibrant objects'' that are used in homotopy theory.
They include both syntactic categories of type theory and an important class of Quillen model categories which we call \emph{type-theoretic model categories} (closely related to those of~\textcite{ak:htmtt,gk:univlcc}).

In \S\ref{sec:hothy-fibcat} we do some basic categorical homotopy theory in type-theoretic fibration categories.
In particular, we prove that type-theoretic fibration categories are automatically ``categories of fibrant objects''~\parencite{brown:ahtgsc}, and have some of the same good properties as model categories.
(This has recently also been proven using internal type-theoretic arguments by~\textcite{akl:faht}.)

In \S\ref{sec:catsem}, we recall how a type-theoretic fibration category interprets intensional type theory.
This implies that we can prove things about type-theoretic fibration categories using their internal type theory.
In \S\ref{sec:homotopy-type-theory} we explore this further, giving some basic definitions and results of homotopical type theory, and explaining their meaning in the categorical semantics.
Then in \S\ref{sec:tyuniverses} we recall how type-theoretic universes arise from categorical ones, and in \S\ref{sec:univalence-axiom} we state Voevodsky's univalence axiom and interpret it categorically.

The heart of the paper is in \S\S\ref{sec:sierpinski}--\ref{sec:univalence}, although inverse categories in general do not appear until \S\ref{sec:invcat}.
Sections~\ref{sec:sierpinski}--\ref{sec:univalence} treat in detail the first nontrivial example of an inverse category, which was already mentioned above: the arrow category $\bbtwo = (1\to 0)$.
Assuming \sC to be a type-theoretic fibration category with one or more universe objects, in \S\ref{sec:sierpinski}--\ref{sec:universes} we build the same structure in $\sC^\bbtwo$, and then in \S\ref{sec:univalence} we show that the universes in $\sC^\bbtwo$ inherit univalence from those in \sC.

Then in \S\ref{sec:invcat} we consider general inverse categories.
It turns out that once the arguments of \S\S\ref{sec:sierpinski}--\ref{sec:univalence} are understood, little work is required to generalize to the case of arbitrary inverse categories.
The work of \S\S\ref{sec:sierpinski}--\ref{sec:univalence} is almost exactly the same as the induction step in the corresponding proof for a general inverse category.
The main new ingredient is that certain limits need to exist and be well-behaved in \sC in order for the Reedy homotopy theory on $\sC^I$ to define a type-theoretic fibration category when $I$ is a general inverse category.
If \sC is a type-theoretic model category, then this is automatic.
For general \sC, it is true as long as all the co-slice categories $x/I$ are finite; the proof follows~\textcite{rb:cofibrations} and involves proving that acyclic cofibrations are stable under homotopy pullbacks.
With this in place, it suffices to merely sketch the necessary modifications to the proofs of \S\S\ref{sec:sierpinski}--\ref{sec:univalence}.

Finally, in \S\ref{sec:oll} we extend the arguments further to the general case of oplax limits (which is again completely straightforward); and in \S\ref{sec:scones} we consider applications to gluing constructions and canonicity. 

\subsection*{Acknowledgments}

I would like to thank Steve Awodey and Peter LeFanu Lumsdaine for the many things they have taught me about homotopical type theory, and for providing helpful feedback on drafts of this paper.
I am especially grateful to Peter for pointing out some holes in the treatment of nested universes, and also for reminding me repeatedly of the advantages of the definition~\eqref{eq:hiso} (which was proposed in this context by Andr\'e Joyal) until the point finally sunk in.
I would also like to thank Karol Szumi\l{}o for telling me about fibration categories, Bob Harper for teaching me about logical relations (and many other things), and Thierry Coquand for making the connection for me between logical relations and inverse diagrams.
I am also grateful to Richard Garner for helping to point out an error in an overly optimistic draft, and Dan Licata for many useful conversations about canonicity.
The referees also made many very helpful corrections and suggestions.
Finally, I would like to thank Vladimir Voevodsky, Steve Awodey, and Thierry Coquand for organizing the special year at the IAS on homotopy type theory, where this paper was finished.

\section{Type-theoretic fibration categories}
\label{sec:ttfc}

The following definition, written in the style of homotopy theory, nevertheless also encapsulates the category-theoretic structure necessary for modeling dependent type theory with dependent sums, dependent products, and identity types.

\begin{defn}\label{def:ttfc}
  A \textbf{type-theoretic fibration category} is a category \sC with the following structure.
  \begin{enumerate}[leftmargin=*,label=(\arabic*)]
  \item A terminal object $1$.\label{item:cat1}
  \item A subcategory $\cF\subset\sC$ containing all the objects, all the isomorphisms, and all the morphisms with codomain $1$.\label{item:cat2a}
    \begin{itemize}[leftmargin=*]
    \item A morphism in \cF is called a \textbf{fibration}; we write fibrations as $A\fib B$.
    \item A morphism $i$ is called an \textbf{acyclic cofibration} if it has the left lifting property with respect to all fibrations.
      This means that if $p$ is a fibration and $p f = g i$, then there is an $h$ (not generally unique) with $f = h i$ and $g = p h$.
      We write acyclic cofibrations as $A \acof B$.
    \end{itemize}
  \item All pullbacks of fibrations 
    exist and are fibrations.\label{item:cat3}
  \item For every fibration $g: A\fib B$, 
    the pullback functor $g^*: \sC/B \to \sC/A$ has a partial right adjoint $\Pi_g$, defined at all fibrations over $A$, and whose values are fibrations over $B$.
    This implies that acyclic cofibrations are stable under pullback along $g$.\label{item:cat4}
  \item Every morphism factors as an acyclic cofibration followed by a fibration.\label{item:cat7}
  \item In the following commutative diagram:
    \[\vcenter{\xymatrix@R=0.7pc{
        X\ar[r]\ar[dd] \pullbackcorner &
        Y\ar[r]\ar[dd] \pullbackcorner &
        Z\ar[dd]\\\\
        A\ar@{>->}[r]^{\sim} \ar@{->>}@(dr,dl)[rr] &
        B\ar@{->>}[r] &
        C\\
        & &
      }}\]
    if $B\fib C$ and $A\fib C$ are fibrations, $A\acof B$ is an acyclic cofibration, and both squares are pullbacks (hence $Y\to Z$ and $X\to Z$ are fibrations by~\ref{item:cat3}), then $X\to Y$ is also an acyclic cofibration.\label{item:cat8}
  \end{enumerate}
\end{defn}

Typically one says that an object $A$ is \textbf{fibrant} if the map $A\to 1$ is a fibration; thus we are assuming all objects to be fibrant.
Frequently, we obtain this by restricting to the subcategory of fibrant objects in some larger category; we generally denote this by $\sC\f$.
For instance, if \sC is a type-theoretic fibration category, then $(\sC/A)\f$ denotes the full subcategory of $\sC/A$ consisting of the fibrations $B\fib A$.
It is easy to verify that $(\sC/A)\f$ is again a type-theoretic fibration category.

\begin{rmk}
  In type theory the terms \emph{display map} and \emph{dependent projection} are usually used instead of \emph{fibration}.
  Under this translation, conditions~\ref{item:cat1}, \ref{item:cat2a}, \ref{item:cat3}, and~\ref{item:cat4} make \sC into a \emph{display map category} (see e.g.~\textcite[\S10.4]{jacobs:cltt}) or a \emph{\sD-category} (see e.g.~\textcite{streicher:semtt}) with the well-known additional structure required for interpreting a unit type, strong dependent sums, and dependent products.
  As we will now explain, conditions~\ref{item:cat7} and~\ref{item:cat8} are a rephrasing of the analogous structure required for identity types.
\end{rmk}

Condition~\ref{item:cat7} implies, in particular, that we have the following structure from homotopy theory.

\begin{defn}
  A \textbf{weak factorization system} $(\cL,\cR)$ on a category consists of two classes of maps \cL and \cR such that
  \begin{itemize}
  \item \cL is precisely the class of maps having the left lifting property with respect to \cR, and dually.
  \item Every morphism factors as $p \circ i$ for some $i\in \cL$ and $p\in \cR$.
  \end{itemize}
\end{defn}

In a type-theoretic fibration category, the acyclic cofibrations and fibrations satisfy this definition, except that a map having the right lifting property with respect to the acyclic cofibrations need not be a fibration.
Thus we must take \cR instead to be the class of all such maps, which includes the fibrations but may be strictly larger.
The ``retract argument'' from homotopy theory (e.g.~\textcite[{}1.1.9]{hovey:modelcats}) then implies that \cR is precisely the class of \emph{retracts} of fibrations (in the arrow category).

We have chosen~\ref{item:cat7} and~\ref{item:cat8} as better-motivated axioms from a category-theoretic perspective.
However, they are equivalent to a pair of axioms which are more directly related to type theory.

\begin{lem}\label{thm:path-fact}
  Suppose \sC satisfies~\ref{item:cat1}--\ref{item:cat4} of \autoref{def:ttfc}.
  Then it satisfies~\ref{item:cat7} and~\ref{item:cat8} (hence is a type-theoretic fibration category) if and only if it satisfies the following.
  \begin{enumerate}[start=5,label=(\arabic*\/$'$)]
  \item For any fibration $A\fib B$, the diagonal morphism $A\to A\times_B A$ factors as $A\acof P_B A \fib A\times_B A$, where $P_BA \fib A\times A$ is a fibration and $A\acof P_B A$ is an acyclic cofibration.\label{item:cat7p}
  \item There exists a factorization as in~\ref{item:cat7p} such that~\ref{item:cat8} holds whenever the bottom row is $A\acof P_B A \fib B$.\label{item:cat8p}
  \end{enumerate}
\end{lem}

In homotopy theory, a factorization as in~\ref{item:cat7p} is called a \emph{path object} for $A$ over $B$.
(Sometimes these are said to be ``very good'' path objects, but they will be the only path objects we consider.)
We will usually denote the acyclic cofibration $A\acof P_B A$ by $r$ (for reflexivity).
Conditions~\ref{item:cat7p} and~\ref{item:cat8p} are similar to the \emph{stable path objects} of~\textcite{warren:thesis,aw:htpy-idtype}, but weaker because we don't (yet) require a functorial global \emph{choice} of path objects.
We will return to this question in \S\ref{sec:catsem}.

\begin{proof}[Proof of \autoref{thm:path-fact}]
  Clearly~\ref{item:cat7p} and~\ref{item:cat8p} are special cases of~\ref{item:cat7} and~\ref{item:cat8}, respectively.
  Conversely, assuming~\ref{item:cat7p} and~\ref{item:cat8p}, suppose given morphisms $A \xto{f} B \fib C$ such that $B\fib C$ and the composite $A\fib C$ are fibrations.
  Define $P_C f$ as the pullback in the following diagram:
  \[\vcenter{\xymatrix{
      A \ar[r]^f \ar@{.>}[dr] \ar[ddr] \ar@(d,ul)[dddr]_{f} & B \ar[dr]\\
      &P_C f\ar[r]^q\ar@{->>}[d] \pullbackcorner &
      P_C B\ar@{->>}[d]\\
      &A\times_C B\ar[r]_{f\times 1_B} \ar@{->>}[d] &
      B\times_C B.\\
      &B
    }}\]
  For~\ref{item:cat7}, it suffices to show that the induced map $i: A\to P_C f$ is an acyclic cofibration.
  This is a simple translation of the proof of~\textcite[{}4.2.1]{gg:idtypewfs} from type theory into category theory, which we now sketch.

  First, we need some basic operations on paths.
  Let $B \xto{r} P_C B \fib B\times_C B$ be a factorization satisfying~\ref{item:cat7p} and~\ref{item:cat8p}.
  Consider the following square of solid arrows:
  \begin{equation}
  \vcenter{\xymatrix@C=3pc{
      P_C B\ar@{=}[r]\ar[d]_{(1,r)} &
      P_C B\ar@{->>}[d]\\
      P_C B \times_B P_C B\ar[r]_-{\pi_1 \times\pi_3} \ar@{.>}[ur]^c &
      B\times_C B
    }}\label{eq:cdef}
  \end{equation}
  where the pullback $P_C B \times_B P_C B$ is over the ``middle'' copies of $B$.
  Then the left-hand map is the pullback of $r$ along the fibration $P_C B\fib B$, hence is an acyclic cofibration.
  Thus there exists a lift which we have called $c$; we think of it as a ``concatenation'' operation on paths.
  The commutativity of the upper triangle in~\eqref{eq:cdef} means that $c(1,r) = 1_{P_C B}$, i.e.\ post-concatenating with a constant path is the identity.
  Pulling~\eqref{eq:cdef} back along $f\times 1: A\times_C B \to B\times_C B$, we obtain
  \[\vcenter{\xymatrix@C=3pc{
      P_C f\ar@{=}[r]\ar[d]_{(1,r)} &
      P_C f\ar@{->>}[d]\\
      P_C f \times_B P_C B\ar[r]_-{\pi_1 \times\pi_3} \ar@{.>}[ur]^c &
      A\times_C B.
    }}\]
  Now the following square of solid arrows commutes:
  \begin{equation}
  \vcenter{\xymatrix{
      B \ar[r]^-r \ar[d]_r &
      P_C B\ar[r] &
      P_{(B\times_C B)} (P_C B)\ar@{->>}[d]\\
      P_C B\ar[rr]_-{(c(r,1),1)} \ar@{.>}[urr]^(.4)\psi &&
      P_C B \times_{(B\times_C B)} P_C B
    }}\label{eq:psidef}
  \end{equation}
  (where we have chosen a particular path object for $P_C B$ over $B\times_C B$).
  Of course, $r$ is an acyclic cofibration, so there exists a lift which we have called $\psi$.
  Commutativity of the lower triangle in~\eqref{eq:psidef} means that $\psi$ is a path from $c(r,1)$ to $1_{P_C B}$, i.e.\ pre-concatenating with a constant path is homotopic to the identity.
  Pulling~\eqref{eq:psidef} back along $f\times 1$, we obtain
  \begin{equation}
  \vcenter{\xymatrix{
      A \ar[r]^-i \ar[d]_i &
      P_C f\ar[r] &
      P_{(A\times_C B)} (P_C f)\ar@{->>}[d]\\
      P_C f\ar[rr]_-{(c(r,1),1)} \ar@{.>}[urr]^(.4)\psi &&
      P_C f \times_{(A\times_C B)} P_C f.
    }}\label{eq:psipb}
  \end{equation}
  Note that by~\ref{item:cat8p}, the pullback of $P_{(B\times_C B)} (P_C B)$ along $f\times 1$ is a valid path object for $P_C f$, so we have denoted it $P_{(A\times_C B)} (P_C f)$.

  To show that $i$ is an acyclic cofibration, we must show that it has the left lifting property with respect to fibrations.
  However, since fibrations are stable under pullback, in fact it suffices to find a lift in any commutative square
  \begin{equation}
  \vcenter{\xymatrix{
      A\ar[r]^d\ar[d]_i &
      D\ar@{->>}[d]\\
      P_C f\ar@{=}[r] &
      P_C f
    }}\label{eq:desired}
  \end{equation}
  with an identity arrow on the bottom.
  In this case, the composite $D \fib P_C f \fib B$ is a fibration, so in the following commutative square of solid arrows:
  \[\vcenter{\xymatrix{
      D\ar@{=}[rr]\ar[d]_{(1,r)} &&
      D\ar@{->>}[d]\\
      D\times_B P_C B\ar[r] \ar@{.>}[urr]^m &
      P_C f \times_B P_C B \ar[r]_-c &
      P_C f
    }}\]
  the left-hand map $(1,r)$ is an acyclic cofibration; thus there exists a lift which we have called $m$.
  Therefore, the following diagram commutes:
  \[\vcenter{\xymatrix@C=3pc{
      A \ar[r]^d\ar[d]_i &
      D\ar@{->>}[d]\\
      P_C f\ar[r]_-{c(r,1)} \ar[ur]|{m(d,q)} &
      P_C f
    }}\]
  (recall that $q: P_C f \to P_C B$ is the map induced by $f$).
  Finally, since $D\fib P_C f$ is a fibration, the left-hand map in the following square is an acyclic cofibration:
  \[\vcenter{\xymatrix@C=3pc{
      D\ar@{=}[r]\ar[d] &
      D\ar@{->>}[d]\\
      D\times_{P_C f} P_{(A\times_C B)}(P_C f)\ar[r]_-{\pi_2} \ar@{.>}[ur]^\tau &
      P_C f
    }}\]
  so we have a lift $\tau$.
  Now the composite
  \[ P_C f \xto{(m(d,q),\psi)} D\times_{P_C f} P_{(A\times_C B)}(P_C f) \too[\tau] D \]
  is a lift in~\eqref{eq:desired}.
  This proves~\ref{item:cat7}.

  Of course, to prove~\ref{item:cat7} it would have sufficed to take $C=1$, but the extra generality is convenient for proving~\ref{item:cat8}.
  Namely, if $f$ is the acyclic cofibration $A\acof B$ in the situation of~\ref{item:cat8}, we construct $P_C f$ as above.
  The entire construction is then preserved by pullback along any map into $C$, using~\ref{item:cat8p} for the factorization $B \acof P_C B \fib B\times_C B$.
  However, since $f$ is an acyclic cofibration, by the ``retract argument'', it is a retract of $A\to P_C f$ in $\sC/C$.
  It follows that any pullback of $f$ along a map $Z\to C$ will also be a retract of the corresponding pullback of $A\acof P_C f$, and hence also an acyclic cofibration.
  This gives~\ref{item:cat8}.
\end{proof}

\begin{rmk}
  This construction of factorizations from path objects is, of course, motivated by the classical \emph{mapping path space} construction in homotopy theory.
  However, in classical homotopy theory this construction does \emph{not} always produce the desired factorizations.
  In fact, even for the classical ``Hurewicz'' model structure on topological spaces~\parencite{strom:the-htpy-cat}, the inclusion of a space $A$ into the mapping path space of $f: A\to B$ need not be a Hurewicz cofibration.

  That particular example can be remedied by using Moore paths~\parencite{br:ffact}, but the point is that even in a model category where all objects are fibrant, the general construction may fail.
  It only works for type-theoretic fibration categories because acyclic cofibrations are stable under pullback along fibrations.
\end{rmk}

We can now describe the two main classes of examples we have in mind.

\begin{eg}\label{thm:syncat}
  Consider a dependent type theory with a unit type, dependent sums, dependent products, and intensional identity types.
  We require the unit type, sums, and products to satisfy judgmental $\eta$-conversion rules, e.g.\ we have $f \eq (\lambda x. f(x))$ for $f:\sprod{x:A} B(x)$, and $w \eq (\fst(w),\snd(w))$ for $w:\ssum{x:A} B(x)$.
  (The symbol $\eq$ denotes judgmental equality.)
  These $\eta$-conversion rules are not really necessary, but they simplify our definitions and proofs.

  Let \sC be the category of contexts (or ``syntactic category'') of such a type theory.
  We define the \emph{fibrations} in \sC to be the closure under isomorphisms of the ``dependent projections'' from any context to an initial segment thereof.
  The $\eta$-conversions imply that every context is isomorphic to one consisting of a single type (namely, the iterated dependent sum of the context, if it is nonempty, and the unit type otherwise), and similarly that every fibration is isomorphic to the projection from a single dependent sum $\sum_{x: A} B(x)$ to the base type $A$.
  The right adjoints $\Pi_g$ come from dependent product types; $\eta$-conversion for dependent products makes them actual adjoints (rather than weak adjoints).

  We obtain~\ref{item:cat7p} and~\ref{item:cat8p} from dependent identity types, following~\textcite{gg:idtypewfs}.
  We can avoid the more complicated ``identity contexts'' of \emph{ibid.}\ by using dependent sums with $\eta$.

  A word about notation: when working internally to type theory, it is natural to write the identity type as simply $(x=y)$.
  However, that can be confusing when also discussing categorical semantics, since we also need to consider ordinary set-theoretic equality of morphisms in general type-theoretic fibration categories, which in the syntactic category is \emph{judgmental} equality.
  Thus, motivated by the path-object interpretation, we write the identity type of a dependent type $(x: A) \pr B(x) \ty$ as
  \begin{equation}
    (x: A),\, (b_1: B(x)),\, (b_2: B(x)) \pr (b_1\id b_2)  \ty.\label{eq:idtype}
  \end{equation}
  (Note that we write ``$A\ty$'' for the judgment that $A$ is a type, which we distinguish from a judgment $A:\type$ that $A$ is a term belonging to some universe type; see \S\ref{sec:tyuniverses}.)
  In the syntactic category,~\eqref{eq:idtype} represents the path object fibration $P_A B \fib B\times_A B$ of a fibration $B\fib A$.
  The map $B\to P_A B$ is the reflexivity term
  \[(x:A),(b:B(x)) \pr (\r_b : b\id b),\]
  which is an acyclic cofibration: this is essentially the content of the elimination rule for the identity type.
  Finally, the whole construction is stable under pullback, because dependent identity types are preserved by substitution.
\end{eg}

The second class of examples comes from homotopy theory.
In these examples, the fibrations are generally closed under retracts, and so~\ref{item:cat7} simply asserts that (acyclic cofibrations, fibrations) is a factorization system.
In this case, the remaining axioms simplify:
\begin{itemize}
\item The right class of a weak factorization system is automatically preserved by pullback, so~\ref{item:cat3} need only assert that such pullbacks exist.
\item If $\Pi_g$ is defined at fibrations, then by adjunction, it takes fibrations as values if and only if $g^*$ preserves acyclic cofibrations.
\end{itemize}
Most examples from homotopy theory have the following additional structure~\parencite{quillen:htpical-alg,hovey:modelcats,hirschhorn:modelcats}.

\begin{defn}
  A \textbf{model category} is a complete and cocomplete category with three classes of maps \cC (cofibrations), \cF (fibrations), and \cW (weak equivalences) such that
  \begin{itemize}
  \item $(\cC\cap\cW, \cF)$ and $(\cC, \cF\cap\cW)$ are weak factorization systems.
  \item If two of $f$, $g$, and $g f$ are in \cW, so is the third.
  \end{itemize}
\end{defn}

In a model category, the maps in $\cC\cap\cW$ are called \emph{acyclic cofibrations}, and similarly the maps in $\cF\cap\cW$ are \emph{acyclic fibrations} (some authors say \emph{trivial} instead of \emph{acyclic}).
We will mostly work only with one weak factorization system, as we have in a type-theoretic fibration category.
But since that weak factorization system behaves like $(\cC\cap\cW, \cF)$ in a model category, we use the names ``acyclic cofibration'' and ``fibration'' for it.

Now we can define our second main class of examples.

\begin{defn}\label{def:ttmc}
  A \textbf{type-theoretic model category} is a model category \sM with the following additional properties.
  \begin{enumerate}
  \item Limits preserve cofibrations.\label{item:m1}
  \item \sM is \emph{right proper}, i.e.\ weak equivalences are preserved by pullback along fibrations.\label{item:m2}
  \item Pullback $g^*$ along any fibration 
    has a right adjoint $\Pi_g$.\label{item:m3}
  \end{enumerate}
\end{defn}

\noindent
Some comments on the various parts of this definition are in order.
\begin{enumerate}
\item Limits preserving cofibrations means that any natural transformation that is a levelwise cofibration induces a cofibration between the limits.
  This is automatic if the cofibrations are exactly the monomorphisms.
  It implies easily that cofibrations are stable under pullback.
\item Right properness is automatic if all objects of \sM are fibrant.
  Moreover, since cofibrations are stable under pullback, if \sM is right proper, then acyclic cofibrations are stable under pullback along fibrations.
  On the other hand, if this latter condition holds, then since any weak equivalence in a model category factors as an acyclic cofibration followed by an acyclic fibration, and acyclic fibrations are always stable under pullback, it follows that \sM is right proper.
\item Of course, if \sM is locally cartesian closed, then all pullback functors have right adjoints.
\end{enumerate}

A \emph{Cisinski model category}~\parencite{cisinski:topos,cisinski:presheaves} is a model structure on a Grothendieck topos whose cofibrations are the monomorphisms.
Therefore, any right proper Cisinski model category is a type-theoretic model category.

\begin{prop}
  If \sM is a type-theoretic model category, then its full subcategory $\sM\f$ of fibrant objects is a type-theoretic fibration category.
\end{prop}
\begin{proof}
  By the remarks above, conditions~\ref{item:cat1}--\ref{item:cat3} and~\ref{item:cat7} hold for any model category.
  For~\ref{item:cat4}, it remains to show that $\Pi_g$ preserves fibrations.
  As remarked above, for this it suffices for $g^*$ to preserve acyclic cofibrations, but we have seen that this follows from \autoref{def:ttmc}\ref{item:m1} and~\ref{item:m2}.
  Finally,~\ref{item:cat8} follows since cofibrations are stable under pullback, while weak equivalences between fibrations are always stable under pullback.
\end{proof}

\begin{rmk}\label{thm:ttmc-piquillen}
  In a type-theoretic model category, any fibration $g$ 
  yields a Quillen adjunction $g^* \dashv \Pi_g$.
\end{rmk}

\begin{rmk}
  In~\textcite{ak:htmtt} a \emph{logical model category} was defined to be one where pullback along fibrations preserves acyclic cofibrations and has a right adjoint.
  This suffices to interpret type theory with dependent sums and products, but for identity types we need at least pullback-stability of cofibrations to ensure~\ref{item:cat8}.
  The additional assumption that all limits preserve cofibrations will be useful in \S\ref{sec:invcat}.
\end{rmk}

\begin{egs}
  Here are our basic examples of type-theoretic model categories.
  \begin{itemize}
  \item Any locally cartesian closed category, equipped with the trivial model structure in which the weak equivalences are the isomorphisms and every morphism is a cofibration and a fibration.
    Of course, this sort of category will only interpret \emph{extensional} type theory.
  \item The category of groupoids, with its canonical model structure in which the weak equivalences are the equivalences of categories, the fibrations are the functors with isomorphism-lifting (``isofibrations''), and the cofibrations are the injective-on-objects functors.
    All objects are fibrant, cofibrations are clearly preserved by limits, and isofibrations are exponentiable (although the category of groupoids is not locally cartesian closed).
    A closely related construction gave the first non-extensional set-theoretic model of type theory~\parencite{hs:gpd-typethy}.
    The desire to include this example is the main reason not to assume in the definition of type-theoretic model category that the cofibrations are the monomorphisms (as was done by~\textcite{gk:univlcc}).
  \item The category \sSet of simplicial sets, with its traditional (Quillen) model structure.
    This is a right proper Cisinski model category.
  \item The ``injective model structure'' on any category of simplicial presheaves is also a right proper Cisinski model category.
    In fact, \textcite{cisinski:lccc-rpcmc} has shown that any locally cartesian closed, locally presentable \io-category admits a presentation by a right proper Cisinski model category; an alternative proof can be found in \textcite{gk:univlcc}.
  \end{itemize}
\end{egs}

Of course, we can add additional structure to a type-theoretic fibration category that corresponds to additional type-forming operations.
In this paper we will mostly restrict ourselves to the above structure, which is the minimum necessary to state the univalence axiom and prove that it lifts to inverse diagrams.
However, for applications to canonicity, we will also need a natural numbers type.

\begin{defn}\label{def:shnno}
  A \textbf{strong homotopy natural numbers object} (\shnno) in a type-theoretic fibration category \sC is an object $N$ together with morphisms $o:1\to N$ and $s:N\to N$ such that:
  \begin{itemize}
  \item for any fibration $p:B\fib N$ and morphisms $o':1\to B$ and $s':B\to B$ such that $po'=o$ and $ps'=sp$, there exists a section $f:N\to B$ (meaning $p f = 1_N$) such that $f o = o'$ and $f s = s' f$.
  \end{itemize}
\end{defn}

The adjective ``strong homotopy'' indicates that this is a weakening of the usual category-theoretic notion of natural numbers object, but only up to \emph{coherent} homotopy.

\begin{eg}
  If a type theory contains a natural numbers type, then its syntactic category contains a \shnno.
  The universal property is exactly the dependent eliminator (proof by induction).
\end{eg}

\begin{eg}
  Suppose \sC is a type-theoretic model category in which the countable coproduct $\sum_{n\in\mathbb{N}} 1$ of copies of the terminal object is fibrant (such as groupoids or simplicial sets).
  Then we can define $N$ to be this coproduct, with $o$ the inclusion of the $0^{\mathrm{th}}$ summand and $s$ taking the $n^{\mathrm{th}}$ summand to the $(n+1)^{\mathrm{st}}$.
  And given $p:B\fib N$ with $o'$ and $s'$, we can simply define $f:\sum_{n\in\mathbb{N}} 1 \to B$ to act on the $n^{\mathrm{th}}$ summand by $(s')^n\circ o'$.
  Thus, $\sC\f$ contains a \shnno.

  If $\sum_{n\in\mathbb{N}} 1$ is not fibrant, then we need to fibrantly replace it in a controlled way.
  We will explain how to do this for more general inductive types and higher inductive types in \textcite{ls:hits}.
\end{eg}

\section{Homotopy theory in type-theoretic fibration categories}
\label{sec:hothy-fibcat}

In this section we show that type-theoretic fibration categories enjoy many of the same nice properties as type-theoretic model categories.
It is well-known that path objects suffice to define many notions of homotopy theory, but they are not always well-behaved without cofibrancy assumptions, which are unavailable in a fibration category.
However, the stability properties of acyclic cofibrations in a type-theoretic fibration category can frequently serve as a substitute.

We define a \emph{(right) homotopy} between two maps $f,g: A\toto B$ to be a lifting of $A\to B\times B$ to a path object $P B$ for $B$.
We denote a homotopy by $H: f\sim g$.
Strictly speaking, this depends on a choice of path object for $B$.
However, since $B\to P B$ is always an acyclic cofibration, every path object factors through every other, so the homotopy \emph{relation} is independent of this choice.

The morphism $c$ defined in the proof of \autoref{thm:path-fact} ``concatenates'' homotopies, so that if $H: f\sim g$ and $K: g\sim h$, then $c(H,K): f\sim h$.
Similarly, for any $f$ we have $r f : f \sim f$, while by lifting in the square
\begin{equation}
\vcenter{\xymatrix@C=3pc{
    B\ar[r]^r\ar[d]_r &
    P B\ar@{->>}[d]^{(\pi_1,\pi_2)}\\
    P B\ar[r]_-{(\pi_2,\pi_1)} \ar@{.>}[ur]^v&
    B\times B
  }}\label{eq:inversion}
\end{equation}
we obtain an inversion morphism on paths, so that if $H: f\sim g$ then $v H: g\sim f$.
Moreover, $v H$ is actually a homotopy inverse of $H$ for concatenation, in the sense that $c(v(H),H) \sim r f$.
We can see this by lifting in the following square:
\begin{equation}
  \label{eq:oppconcat}
  \vcenter{\xymatrix{
      B\ar[r]\ar@{>->}[d]_r &
      P B\ar[r] &
      P_{B\times B}(P B)\ar@{->>}[d]\\
      P B\ar[rr]_-{(c(v,1),r)} \ar@{.>}[urr] & &
      P B\times_{B\times B} P B.
      }}
\end{equation}
These operations are the lowest levels of an ``algebraic'' weak $\infty$-groupoidal structure on any object in a type-theoretic fibration category~\parencite{bg:type-wkom,pll:wkom-type}.

Now if $f,g: A\toto B$ are two morphisms and $H: f\sim g$, then for any $k: C\to A$ the composite $H k$ is a homotopy $f k\sim g k$.
On the other side, for any morphism $k: B\to C$ we can lift in the square
\begin{equation}
\vcenter{\xymatrix{
    B\ar[r]^k \ar[d]_r &
    C\ar[r] &
    P C\ar@{->>}[d]\\
    P B\ar[r] \ar@{.>}[urr]^(.4){\map_k} &
    B\times B\ar[r]_{k\times k} &
    C\times C.
  }}\label{eq:mapdef}
\end{equation}
Then any homotopy $H: f\sim g$ yields a homotopy $\map_k H : k f \sim k g$.%
\footnote{The notation $\ap_k$ can be read either as the \underline{a}ction on \underline{p}aths of $k$ or as the \underline{ap}plication of $k$ to a path.}
The ``operation'' \map respects concatenation up to homotopy, in the sense that
\begin{equation}
  \map_k(c(H,K)) \sim c(\map_k(H),\map_k(K)).\label{eq:mapconcat}
\end{equation}
We can see this by lifting in the square
\begin{equation}
  \label{eq:mapconcatsq}
  \vcenter{\xymatrix@C=3pc{
      P B\ar[r]^-{\map_k}\ar[d]_-{(1,r)} &
      P C\ar[r]^-r &
      P_{C\times C} (P C)\ar@{->>}[d]\\
      P B \times_{B} P B \ar[rr]_-{(\map_k c, c (\map_k,\map_k))} \ar@{.>}[urr] & &
      P C \times_C P C.
    }}
\end{equation}
This is the first level of another hierarchy of coherences here making $k$ a weak $\infty$-groupoid functor.
Finally, we note that \map is functorial with respect to composition of morphisms as well, in that
\begin{equation}
  \label{eq:mapcompose}
  \map_{k_2} \map_{k_1} (H) \sim \map_{k_2 k_1}(H).
\end{equation}
We can see this by lifting in the square
\begin{equation}
  \label{eq:mapcomposesq}
  \vcenter{\xymatrix@C=2pc{
      B\ar[r]^-{k_2 k_1}\ar[d]_r &
      D\ar[r] & P D \ar[r] &
      P_{D\times D}(P D)\ar@{->>}[d]\\
      P B \ar[rrr]_-{(\map_{k_2} \map_{k_1}, \map_{k_2 k_1})} \ar@{.>}[urrr] & &&
      P D \times_{D\times D} P D.
      }}
\end{equation}

We define a map $f: A\to B$ to be a \emph{homotopy equivalence} if there is a map $g: B\to A$ and homotopies $g f\sim 1_A$ and $f g \sim 1_B$.
As observed by \textcite{gg:idtypewfs}, the factorizations constructed in \autoref{thm:path-fact} yield a characterization of acyclic cofibrations as certain special homotopy equivalences.

\begin{lem}\label{thm:acof}
  A morphism $f: A\to B$ in a type-theoretic fibration category is an acyclic cofibration if and only if there exists a morphism $g: B\to A$ such that $g f = 1_A$, and a homotopy $H: f g \sim 1_B$ such that $H f$ is a constant homotopy (i.e.\ factors through $B\acof P B$).
\end{lem}
\begin{proof}
  Such a $g$ and $H$ together precisely form a lift in the following square:
  \[\vcenter{\xymatrix{
      A\ar[r]\ar[d]_f &
      P f\ar@{->>}[d] \\
      B\ar@{=}[r] \ar@{.>}[ur]&
      B.
    }}\]
  Since $P f \fib B$ is a fibration, if $f$ is an acyclic cofibration then such a lift certainly exists.
  Conversely, if such a lift exists, then by the retract argument, $f$ is a retract of the acyclic cofibration $A \acof P f$, hence is also an acyclic cofibration.
\end{proof}

Note that the statement of \autoref{thm:acof} is true in all model categories, but only under cofibrancy assumptions on $A$ and $B$.

We would now like to show that type-theoretic fibration categories, while not model categories, do fit into a well-known abstract framework for homotopy theory: the \emph{categories of fibrant objects} of \textcite{brown:ahtgsc}.
By definition, this is a category satisfying \autoref{def:ttfc}\ref{item:cat1}--\ref{item:cat3} and equipped with a further subcategory \cW of ``weak equivalences'' such that:
\begin{itemize}
\item \cW contains all isomorphisms.
\item \cW satisfies ``2-out-of-3'': if two of $f$, $g$, and $g f$ are in \cW, so is the third.
\item Any diagonal $B\to B\times B$ factors as a map in \cW followed by a fibration.
\item Any pullback of a fibration in \cW (an ``acyclic fibration'') is also in \cW.
\end{itemize}
In the absence of any other data in a type-theoretic fibration category, it is natural to choose \cW to be the homotopy equivalences.
It is easy to verify that these contain all isomorphisms and satisfy 2-out-of-3.
The factorization axiom follows from~\ref{item:cat7p} and the observation that acyclic cofibrations are homotopy equivalences.

However, the final axiom is somewhat more difficult to prove.
We begin with the following ``cancellation'' property of acyclic cofibrations.

\begin{lem}\label{thm:acof-cancel}
  If $g f$ and $g$ are acyclic cofibrations in a type-theoretic fibration category, then so is $f$.
\end{lem}

Note that this holds in any model category whose cofibrations are the monomorphisms, since monomorphisms have this cancellation property and weak equivalences have the 2-out-of-3 property.

\begin{proof}
  Suppose $f: A\to B$ and $g: B\to C$ with $g f$ and $g$ acyclic cofibrations.
  By \autoref{thm:acof}, we have $h: C \to B$ and $k: C\to A$ such that $h g = 1_B$ and $k g f = 1_A$, and homotopies $H: g h \sim 1_C$ and $K: g f k \sim 1_C$ such that $H g$ and $K g f$ are constant.

  Define $\ell = k g : B \to A$; then $\ell f = k g f = 1_A$.
  Let
  \[L = \map_h K g \;:\; f \ell = f k g = h g f k g \sim h g = 1_B.\]
  Then $L f = \map_h K g f$, but $K g f$ is constant and $\map_h$ preserves constancy of homotopies.
  Thus $L f$ is constant, so by the other direction of \autoref{thm:acof}, $f$ is an acyclic cofibration.
\end{proof}

\begin{rmk}
  One of the referees pointed out that in the presence of axioms~\ref{item:cat1}--\ref{item:cat7} of a type-theoretic fibration category, the statement of \autoref{thm:acof-cancel} is equivalent to axiom~\ref{item:cat8}.
  Recall that axiom~\ref{item:cat8} says that given the following diagram
  \[\vcenter{\xymatrix@R=0.7pc{
      X\ar[r]\ar[dd] \pullbackcorner &
      Y\ar[r]\ar[dd] \pullbackcorner &
      Z\ar[dd]\\\\
      A\ar@{>->}[r]^{\sim} \ar@{->>}@(dr,dl)[rr] &
      B\ar@{->>}[r] &
      C,\\
      & &
    }}\]
  the map $X\to Y$ is also an acyclic cofibration.
  To deduce this from \autoref{thm:acof-cancel}, by factorization we may consider separately the cases when $Z\to C$ is a fibration and when it is an acyclic cofibration.
  When $Z\to C$ is a fibration, then so is $Y\to B$, and hence by~\ref{item:cat4} $X\to Y$ is an acyclic cofibration.
  When $Z\to C$ is an acyclic cofibration, then by~\ref{item:cat4} so are $Y\to B$ and $X\to A$, hence also $X\to B$.
  Thus, the statement of \autoref{thm:acof-cancel} implies that $X\to Y$ is also an acyclic cofibration.
\end{rmk}

The following lemma says that the two possible meanings of ``fiberwise homotopy'' are the same.

\begin{lem}\label{thm:fibhtpy}
  Suppose $p: A\fib C$ and $q: B\fib C$ are fibrations, and $f,g: A\toto B$ are morphisms in $(\sC/C)\f$.
  Then $f\sim g$ in $(\sC/C)\f$ if and only if there is a homotopy $H: f\sim g$ in \sC such that $\map_q H = r p$.
\end{lem}
\begin{proof}
  Suppose first that $f\sim g$ in $(\sC/C)\f$, via some homotopy $H: A\to P_C B$ for some path object $P_C B$.
  Note that $P_C B$ is not a path object for $B$ in $\sC$, since the composite $P_C B \fib B\times_C B \to B\times B$ will not generally be a fibration.
  However, we can still make this homotopy into a homotopy in \sC as follows.
  Choose some path object $P C$ for $C$.
  Using~\ref{item:cat7}, factor the induced map $B\to (P C) \times_{(C\times C)} (B\times B)$ as an acyclic cofibration followed by a fibration, and call the middle object $P' B$.
  Then the composite
  \[P' B \fib  P C \times_{C\times C} (B\times B) \fib B\times B\]
  is a fibration, so $P' B$ is a path object for $B$.
  Moreover, the composite fibration $P' B \to P C$ is a lift in the square~\eqref{eq:mapdef}, so we can call it $\map_q$.

  Now the composites $P_C B \fib B\times_C B \to B\times B$ and $P_C B \to C \to P C$ agree in $C\times C$, and hence induce the bottom map in the following commutative square:
  \begin{equation}
  \vcenter{\xymatrix{
      B\ar[r]\ar@{>->}[d]_r &
      P' B\ar@{->>}[d]\\
      P_C B\ar[r] \ar@{.>}[ur] &
      P C \times_{C\times C} (B\times B).
    }}\label{eq:fhtobase}
  \end{equation}
  Since $B\to P_C B$ is an acyclic fibration, we have a lift as shown.
  Composing with this lift makes a homotopy $f\sim g$ in $(\sC/C)\f$ into a homotopy $H: f\sim g$ in \sC, and as the bottom map in~\eqref{eq:fhtobase} factors through $C$, we have $\map_q H = r p$.

  Conversely, suppose given $H: f\sim g$ in \sC with $\map_q H = r p$, for particular chosen path objects $P B$ and $P C$ and a morphism $\map_q$ as in~\eqref{eq:mapdef}.
  Starting from this path object $P C$, define $P'B$ as above.
  Now by lifting in the square
  \[\vcenter{\xymatrix{
      B\ar[r]\ar[d] &
      P' B\ar[d]\\
      P B\ar[r]_{\map_q} \ar@{.>}[ur] &
      P C
    }}\]
  we see that $P B$ factors through $P' B$ by a map over $P C$.
  Thus we may assume $H$ to be a homotopy $A \to P' B$ which becomes constant in $P C$.

  Now define $Q$ to be the pullback
  \[\vcenter{\xymatrix{
      Q\ar@{>->}[r]^\sim\ar[d] \pullbackcorner &
      P' B\ar@{->>}[d]\\
      C\ar@{>->}[r]_\sim &
      P C
    }}\]
  so that $H$ induces a map $H': A \to Q$.
  Then $Q \acof P' B$ is an acyclic cofibration, as it is the pullback of $C\acof P C$ along the fibration $P' B \fib P C$.
  Since $B\acof P' B$ is also an acyclic cofibration, by \autoref{thm:acof-cancel} the induced map $B\to Q$ is also an acyclic cofibration.

  Let $P_C B$ be a path object for $B$ in $(\sC/C)\f$.
  Then by lifting in the square
  \[\vcenter{\xymatrix{
      B\ar[r]\ar@{>->}[d]_\sim &
      P_C B\ar@{->>}[d]\\
      Q \ar[r] \ar@{.>}[ur] &
      B\times_C B
    }}\]
  we obtain a map $Q \to P_C B$ over $B\times_C B$.
  Therefore, $H'$ induces a homotopy $f\sim g$ in $(\sC/C)\f$, using the path object $P_C B$.
\end{proof}

Finally, we can characterize the acyclic fibrations, dually to \autoref{thm:acof}.

\begin{lem}[The Acyclic Fibration Lemma]\label{thm:afib}
  A fibration $f: B\fib A$ in a type-theoretic fibration category is a homotopy equivalence if and only if there is a morphism $g: A\to B$ such that $f g = 1_A$ and $g f \sim 1_B$ in $(\sC/A)\f$.
\end{lem}
\begin{proof}
  The ``if'' direction follows from the ``only if'' direction of \autoref{thm:fibhtpy}.
  Conversely, suppose $f: B\fib A$ is a fibration and a homotopy equivalence, with a map $h: A\to B$ and homotopies $f h \sim 1_A$ and $h f \sim 1_B$.
  Choose a path object $P A$ for $A$.
  As $f$ is a fibration, the left-hand map in the square
  \[\vcenter{\xymatrix{
      B \ar@{=}[rrr]\ar@{>->}[d]_\sim &&&
      B\ar@{->>}[d]^f\\
      B\times_A P A\ar[r]_-{\pi_2} \ar@{.>}[urrr]^t &
      P A \ar@{->>}[r] &
      A \times A \ar[r]_-{\pi_2} &
      A}}
  \]
  is an acyclic cofibration, so we have a lift $t$.
  The homotopy $f h \sim 1_A$ gives a map $A\to B\times_A P A$, and composing this with $t$ we obtain a map $g: A\to B$ such that $f g = 1_A$.
  We then have a concatenated homotopy $g f \sim h f g f = h f \sim 1_B$, so it remains only to modify this homotopy to live in $(\sC/A)\f$.
  
  Let $P B$ be a path object for $B$ such that $P B \fib (P A) \times_{(A\times A)} (B\times B)$ is a fibration, constructed as in the proof of \autoref{thm:fibhtpy}.
  Then in particular, we have a fibration $\map_f : P B \fib P A$.
  We may assume our homotopy $H : g f \sim 1_B$ to be defined using this path object, and by \autoref{thm:fibhtpy}, it suffices to modify it to a homotopy which becomes constant after applying $\map_f$.

  Let $K$ denote the concatenated homotopy $g f \sim g f g f = g f \sim 1_B$, where the first homotopy is $\map_g \map_f v(H)$ and the second is $H$.
  Here $v$ is the inversion morphism defined in~\eqref{eq:inversion}, $\map_f$ is the above fibration, and $\map_g$ is defined as in~\eqref{eq:mapdef}.
  Upon applying $\map_f$ to $K$, we have a sequence of secondary homotopies (that is, homotopies of maps into $P A$ over $A\times A$):
  \begin{align*}
    \map_f ( c (\map_g \map_f v(H), H) )
    &\sim c(\map_f \map_g \map_f v(H), \map_f H)\\
    &\sim c(\map_f v(H), \map_f H)\\
    &\sim \map_f c(v(H),H)\\
    &\sim \map_f r g f\\
    &= r f.
  \end{align*}
  The first and third of these homotopies are instances of~\eqref{eq:mapconcat}.
  The second is an instance of~\eqref{eq:mapcompose} (using the fact that we may take $\map_{1_A} = \map_{f g}$ to be the identity), while the fourth is $\map_{\map_f}$ applied to an instance of~\eqref{eq:oppconcat}.
  Putting these together, we have a homotopy $\map_f K \sim r f$.

  Finally, as $\map_f : P B \fib P A$ is a fibration, the left-hand morphism in the following square is an acyclic cofibration, while the right-hand morphism was defined to be a fibration:
  \[\vcenter{\xymatrix{
      P B\ar@{=}[r]\ar[d] &
      P B\ar@{->>}[d]\\
      P B \times_{P A} P_{A\times A} (P A)\ar[r]_-{\pi_2} \ar@{.>}[ur]^s &
      P A \times_{A\times A} (B\times B).
    }}\]
  Composing the lifting $s$ with the map
  \[\big(K,\,\map_f K \sim r f\big) \;:\; B \to P B \times_{P A} P_{A\times A} (P A),\]
  we obtain a homotopy $g f \sim 1_B$ which becomes constant in $A$, as desired.
\end{proof}

Like \autoref{thm:acof}, the statement of \autoref{thm:afib} is true in any model category, but only when $B$ and $C$ are cofibrant; the usual proof (e.g.~\textcite[{}7.8.2]{hirschhorn:modelcats}) uses also \emph{left} homotopies (i.e.\ homotopies defined using cylinders rather than path objects).

\begin{cor}\label{thm:afib-stable}
  Acyclic fibrations are stable under pullback.
\end{cor}
\begin{proof}
  The characterization in \autoref{thm:afib} uses only structure in $(\sC/A)\f$ that is preserved by pullback $(\sC/A)\f \to (\sC/A')\f$ along any $g:A'\to A$.
\end{proof}

This completes the proof of the following theorem.

\begin{thm}\label{thm:fibcat}
  Any type-theoretic fibration category is a category of fibrant objects where the weak equivalences are the homotopy equivalences.\qed
\end{thm}

We also have:

\begin{cor}\label{thm:fibhe}
  A morphism $f: A\to B$ in $(\sC/C)\f$ between fibrations $A\fib C$ and $B\fib C$ is a homotopy equivalence in $(\sC/C)\f$ if and only if it is a homotopy equivalence in \sC.
\end{cor}
\begin{proof}
  Since the acyclic cofibrations in \sC and $(\sC/C)\f$ are the same, by 2-out-of-3 it suffices to assume that $f$ is a fibration.
  But in that case, the characterization of \autoref{thm:afib} refers only to $(\sC/B)\f$, and we have $((\sC/C)\f/B)\f \cong (\sC/B)\f$.
\end{proof}

\begin{cor}\label{thm:he-stable}
  Homotopy equivalences are stable under pullback along fibrations.
\end{cor}
\begin{proof}
  This follows from \autoref{thm:fibcat}, but is also a direct consequence of \autoref{thm:afib-stable} and \autoref{def:ttfc}\ref{item:cat4}, since (using the 2-out-of-3 property) a morphism is a homotopy equivalence if and only if it factors as an acyclic cofibration followed by an acyclic fibration.
\end{proof}

\begin{rmk}
  Nowhere in this section did we use \autoref{def:ttfc}\ref{item:cat4} itself, only its consequence that pullback along fibrations preserves acyclic cofibrations.
  Replacing \autoref{def:ttfc}\ref{item:cat4} by this weaker statement would yield a notion of type-theoretic fibration category that seems appropriate to a type theory with dependent sums and identity types, but without dependent products.
\end{rmk}

\begin{rmk}\label{rmk:ttmc-whe-he}
  If \sC is a type-theoretic model category, then homotopies in the type-theoretic fibration category $\sC\f$ are right homotopies in the model-categorical sense.
  Thus, any homotopy equivalence in $\sC\f$, in the sense considered here, is also a weak equivalence in \sC, and similarly for acyclic fibrations.
  \emph{A priori}, there is no reason for the converse to hold: a weak equivalence between fibrant objects in a model category need not be a homotopy equivalence unless its domain and codomain are also cofibrant.
  However, in all the examples of type-theoretic model categories that I know, all objects are cofibrant; in which case the notions defined above do agree with the model-categorical ones.
\end{rmk}

\section{Categorical semantics of type theory}
\label{sec:catsem}

We would like to say that any type-theoretic fibration category has an ``internal language'' which is an intensional dependent type theory.
As is well-known, however, there is a coherence issue, because substitution in type theory is strictly functorial and preserves all operations strictly, while in categorical semantics it corresponds to taking pullbacks, which only has these properties up to isomorphism.
Fortunately, general coherence theorems have recently been found which essentially solve this problem~\parencite{klv:ssetmodel,lw:localuniv}.

Since the goal of this paper is to construct new models of type theory from old ones, we could deal with this in two ways.
We could assume that the models we start with are already strictly coherent (perhaps by the application of a coherence theorem), and show that our constructions preserve strict coherence.
Alternatively, we could simply apply coherence theorems \emph{after} the construction is finished.
The latter choice is easier, but the former gives more precise information.

For the special case considered in \S\S\ref{sec:sierpinski}--\ref{sec:univalence}, we will first perform the constructions without regard to coherence, and then verify that coherence is preserved; this additional information will be important in \S\ref{sec:scones}.
However, for the generalizations considered in \S\S\ref{sec:invcat}--\ref{sec:oll}, we will fall back to invoking coherence theorems, which is sufficient if all we want is to use type theory as an ``internal language'' for homotopy theory.
There seems no obstacle in principle to carrying through coherence in the general case as well, but it would be more tedious.

\subsection{Cloven and split fibration categories}
\label{sec:cloven-split}

Since we will need to treat coherence carefully in some places, at least, we begin by recalling the definitions.
The reader uninterested in the details can skip ahead to \S\ref{sec:ittfc} on page~\pageref{sec:ittfc}.

\begin{defn}\label{def:cttfc}
  A type-theoretic fibration category \sC is \textbf{cloven} if it is equipped with the following additional structure.
  \begin{enumerate}
  \item For each fibration $p:B\fib A$, a set of \emph{fibration structures} on $p$.\label{item:cloven1}
  \item For each morphism $f:C\to A$ and each structured fibration $p:B\fib A$ (that is, each fibration $p$ equipped with a specified fibration structure), a specified pullback square
    \begin{equation}
      \vcenter{\xymatrix{
          f^*B\ar[r]^{f_p} \ar@{->>}[d]_{f^*p} \pullbackcorner &
          B\ar@{->>}[d]^p\\
          C\ar[r]_f &
          A
        }}
    \end{equation}
    together with a specified fibration structure on $f^*p$.\label{item:cloven2}
  \item A specified object $\mathsf{u}$ such that the unique map $\mathsf{u}\to 1$ is an isomorphism and is given a specified fibration structure.\label{item:cloven6}
  \item For every composable pair of structured fibrations $C \xfib{p} B \xfib{q} A$, a specified structured fibration $\Si_q p:\Si_q C \fib A$ and an isomorphism $C\toiso \Si_q C$ over $A$.\label{item:cloven3}
  \item For every structured pair $C \xfib{p} B \xfib{q} A$ as above, a value of the dependent product $\Pi_q p: \Pi_q C \fib A$ equipped with a specified fibration structure.\label{item:cloven4}
  \item For every structured fibration $A\fib B$, a factorization of its diagonal as an acyclic cofibration $A\acof P_B A$ followed by a structured fibration $P_B A \fib A\times_B A$.\label{item:cloven5}
  \item In the situation of~\ref{item:cloven5}, given a morphism $f:D\to B$, let $f^*(P_B A)$ be determined by the following diagram, in which each square is a specified pullback as in~\ref{item:cloven2}:
    \begin{equation}
      \vcenter{\xymatrix{
          f^*(P_B A)\ar[r]\ar@{->>}[d] \pullbackcorner &
          P_B A\ar@{->>}[d]\\
          f^*A \times_D f^*A \ar[r]\ar@{->>}[d]\pullbackcorner &
          A\times_B A\ar@{->>}[d] \ar[r]\pullbackcorner & A \ar@{->>}[d]\\
          f^*A\ar[r]\ar@{->>}[d]\pullbackcorner &
          A\ar@{->>}[d] \ar[r] & B\\
          D\ar[r]_f &
          B.
        }}
    \end{equation}
    Then for every structured fibration $C\fib f^*(P_B A)$, we require a specified diagonal filler in every commutative square\label{item:cloven5a}
    \begin{equation}
      \vcenter{\xymatrix{
          f^* A\ar[r]\ar[d] &
          C\ar@{->>}[d]\\
          f^*(P_B A)\ar@{=}[r] \ar@{.>}[ur] &
          f^*(P_B A).
        }}
    \end{equation}
  \item If \sC has a \shnno $N$, we require a specified fibration structure on $N\fib 1$, and for every structured fibration $B\fib N$ equipped with morphisms $s'$ and $o'$ as in \autoref{def:shnno}, a specified section $f$.\label{item:cloven7}
  \end{enumerate}
\end{defn}

In the usual terminology of type-theoretic semantics, conditions~\ref{item:cloven1} and~\ref{item:cloven2} make \sC into a \emph{full comprehension category}.

\begin{eg}\label{eg:cloven-term}
  Every type-theoretic fibration category can be cloven by giving each fibration a unique fibration structure, taking $\mathsf{u}=1$, $\Si_q C = C$, and choosing particular values of the dependent products, path objects, and liftings.
\end{eg}

\begin{eg}
  If \sC is cloven, then $(\sC/A)\f$ is canonically cloven.
  For most of the structure, this is obvious, while for~\ref{item:cloven6} and~\ref{item:cloven7} we pull back (using~\ref{item:cloven2} in \sC) the corresponding fibration structures of \sC to $\sC/A$.
\end{eg}

\begin{eg}
  In the syntactic category of a type theory as in \autoref{thm:syncat}, whose objects are contexts $\Gm$, we may take a fibration structure on a map $\De\to\Gm$ to be the assertion that $\De$ is the extension of $\Gm$ by a single additional variable declaration $x:A$, where $A$ is a type in context $\Gm$.
  Thus, a map admits at most one fibration structure.
  Then~\ref{item:cloven2} comes from substitution into types in context.
  The singleton context consisting of the unit type is isomorphic to the empty context, giving~\ref{item:cloven6}, while a double context extension $\Gm,(x:A),(y:B)$ is isomorphic to $\Gm,(z:\sum_{x:A} B)$, giving~\ref{item:cloven3}.
  Similarly, dependent product types give~\ref{item:cloven4}, identity types give~\ref{item:cloven5} and~\ref{item:cloven5a}, and a natural numbers type gives~\ref{item:cloven7}.
\end{eg}

\begin{eg}\label{eg:set-split}
  In the category of sets, where all morphisms are fibrations, we can take a fibration structure on a map $p:B\to A$ to be an $A$-indexed family of sets $\{B_a\}_{a\in A}$ such that $B = \coprod_{a:A} B(a)$ with $p$ the canonical projection.
  For~\ref{item:cloven2}, we assign to $\{B_a\}_{a\in A}$ and $f:C\to A$ the family $\{B_{f(c)}\}_{c\in C}$, with the resulting pullback square.
  The rest of the structure is similarly easy to define.
\end{eg}

The last two examples satisfy the following stronger definition.

\begin{defn}\label{def:split}
  A cloven type-theoretic fibration category is \textbf{split} if it satisfies the following.
  \begin{enumerate}
  \item For a structured fibration $p:B\fib A$ and any $f,g$, the specified pullback squares from \autoref{def:cttfc}\ref{item:cloven2}
    \label{item:split1}
    \begin{equation}
      \vcenter{\xymatrix{
          (1_A)^* B\ar[r]\ar@{->>}[d]_{(1_A)^*p} &
          B\ar@{->>}[d]^p\\
          A\ar[r]_{1_A} &
          A
        }}\qquad\text{and}\qquad
      \vcenter{\xymatrix{
          (fg)^* B\ar[r]\ar@{->>}[d]_{(fg)^*p} &
          B\ar@{->>}[d]^p\\
          C\ar[r]_{fg} &
          A
        }}
    \end{equation}
    are equal, respectively, to the pullback squares
    \begin{equation}
      \vcenter{\xymatrix{
          B\ar[r]^{1_B}\ar@{->>}[d]_{p} &
          B\ar@{->>}[d]^p\\
          A\ar[r]_{1_A} &
          A
        }}\qquad\text{and}\qquad
      \vcenter{\xymatrix{
          g^* f^* B\ar[r]\ar@{->>}[d]_{g^*f^*p} &
          f^* B\ar[r]\ar@{->>}[d]^{f^*p} &
          B\ar@{->>}[d]^p\\
          C\ar[r]_{g} &
          D\ar[r]_f &
          A.
        }}
    \end{equation}
  \item For structured fibrations $C \xfib{p} B \xfib{q} A$ and a morphism $f:D\to A$, the canonical isomorphism $\Si_{f^*q} (f_q)^*C \toiso f^* \Si_q C$ is an identity, and the two induced fibration structures on $\Si_{f^* q} (f_q)^* p = f^* \Si_q p$ are equal.\label{item:split2}
  \item In the same situation, the canonical isomorphism $f^* \Pi_q C \toiso\Pi_{f^*q} (f_q)^*C$ is an identity, and the two induced fibration structures on $f^* \Pi_q p = \Pi_{f^* q} (f_q)^* p$ are equal.\label{item:split3}
  \item Similarly, the structure in \autoref{def:cttfc}\ref{item:cloven5} and~\ref{item:cloven5a} is preserved strictly by the specified pullbacks along any morphism into $B$ or $D$.\label{item:split4}
  \end{enumerate}
\end{defn}

Split full comprehension categories can be presented in many equivalent ways; two other commonly used ones are called \emph{categories with families} and \emph{categories with attributes}.

\begin{eg}\label{thm:gpd-split}
  An additional useful example is the model category of groupoids.
  Recall that for any groupoid $A$, there is an equivalence of 2-categories between pseudofunctors $A \to \nGpd$ and fibrations over $A$, with the ``Grothendieck construction'' producing a fibration from a pseudofunctor.
  In this case, we can take a fibration structure on a fibration $p:B\fib A$ to be a pseudofunctor $A\to \nGpd$ whose Grothendieck construction is (literally) $p$ --- this is a categorified version of \autoref{eg:set-split}.

  Pullback of fibration structures along $f:C\to A$ is defined by composing $f$ with pseudofunctors $A\to\nGpd$, and all the rest of the structure can be given explicitly.
  With care, the resulting cloven structure is split; see~\textcite{hs:gpd-typethy,hw:crmtt} for details.
  (The description in \textcite{hs:gpd-typethy} refers only to \emph{strict} functors $A\to\nGpd$, corresponding to \emph{split} fibrations, but the same constructions work in the more general case, as described by \textcite{hw:crmtt}.)
\end{eg}

Note that a split type-theoretic fibration category is an essentially algebraic structure: it consists of some sets (objects, morphisms, fibration structures) and partially defined operations (composition, specified pullbacks, specified factorizations, etc.)\ satisfying some axioms.
Thus, we have a category of split type-theoretic fibration categories, whose morphisms are \textbf{strict functors}, which preserve all the cloven structure on the nose.
More generally, we have such a category for any additional axioms or type-forming operations we might add.

Now the standard way to obtain the categorical semantics of type theory is by way of the following theorem.

\begin{thm}\label{thm:syn-init}
  The syntactic category of any type theory is the initial object of the corresponding category of split type-theoretic fibration categories.
\end{thm}

This theorem is type-theoretic folklore, but precise references can be hard to find.
(It is sometimes stated in terms of \emph{contextual categories}~\parencite{cartmell:gatcc} instead, but these form a coreflective subcategory of split type-theoretic fibration categories, and hence have the same initial object.)
In \textcite[Chapter 3]{streicher:semtt}, a proof is written out in full in the case of the Calculus of Constructions, which contains only dependent products; the general case is essentially no different.

For purposes of categorical semantics, \autoref{thm:syn-init} means that any split type-theoretic fibration category \sC admits a strict functor from the syntactic category of an appropriate type theory.
This functor supplies the \emph{semantics} in \sC of any type-theoretic construction.

The coherence problem can now be precisely stated: how can we replace a general type-theoretic fibration category, such as that arising from a type-theoretic model category, by an equivalent split one?
Here we can appeal to general theorems.
The first such theorem was due to \textcite{hofmann:ttinlccc}, but only worked for extensional identity types.
More recently, general theorems have been found~\parencite{klv:ssetmodel,lw:localuniv} which apply to the intensional case as well.
The basic idea of these theorems is that a fibration structure on $B\fib A$ is given by a pullback square
\begin{equation}
  \vcenter{\xymatrix@-.5pc{
      B\ar[r]\ar@{->>}[d] \pullbackcorner &
      \Vtil\ar@{->>}[d]\\
      A\ar[r]&
      V
      }}
\end{equation}
for some ``universe'' fibration $\Vtil\fib V$; see the cited papers for details.

\subsection{The internal type theory of a fibration category}
\label{sec:ittfc}

The upshot is that if we are given a non-split type-theoretic fibration category \sC, we can interpret type theory in it by replacing it with an equivalent split one and then applying the universal morphism from the syntactic category, which we denote $\m-$.
This yields a collection of inductive rules for interpreting contexts, types, and terms as objects, fibrations, and morphisms of \sC, respectively, which we summarize briefly as follows.
\begin{itemize}
\item Each context $\Gamma$ is interpreted by an object $\m\Gm$.
\item Each substitution between contexts $\Gm \pr (\vec d :\Delta)$ is interpreted by a morphism $\m {\vec d}: \m\Gm \to \m\De$.
\item The empty context is interpreted by the terminal object, $\m{\cdot} = 1$.
\item Each dependent type $\Gamma \pr A\ty$ is interpreted by a fibration $p_A:\m{\Gm,A} \fib \m\Gm$.
  The object $\m{\Gm,A}$ interprets the context extension of $\Gamma$ by a variable of type $A$.
\item The substitution of $\Gm \pr (\vec d :\Delta)$ into a dependent type $\De \pr A\ty$, yielding a dependent type $\Gm \pr (\vec d^* A)\ty$, is interpreted by the pullback of $\m{\De,A} \fib \m\De$ along $\m {\vec d}: \m\Gm \to \m\De$.
\item Each term $\Gamma \pr (a:A)$ is interpreted by a section of $p_A$.
  Note that if $A$ does not depend on \Gm, then $\m{\Gm,A} = \m\Gm \times \m A$, so that such sections correspond bijectively with morphisms $\m\Gm \to \m A$.
\item The unit type (in the empty context) is interpreted by a terminal object.
\item For a dependent type $\Gm, (x:A) \pr B\ty$, the dependent sum $\Gm \pr \sum_{x:A} B \ty$ is interpreted by the composite fibration $\m{\Gm,A,B}\fib \m{\Gm,A} \fib \m\Gm$.
\item In the same situation, the dependent product $\Gm \pr \prod_{x:A} B \ty$ is interpreted by the fibration $\Pi_{p_A} \m{\Gm,A,B} \fib\m\Gm$, where $\Pi_{p_A}$ denotes the right adjoint to pullback along $p_A:\m{\Gm,A}\fib \m\Gm$.
\item For $\Gm\vdash A\ty$, the identity type
  \[ \Gm,(x:A),(y:A) \pr (x\id y) \ty\]
  is interpreted by a path object $P_{\m\Gm}\m A \fib \m A \times_{\m\Gm} \m A$, with the reflexivity constructor $\Gm,(x:A) \pr (\r_x:x\id x)$ being interpreted by the acyclic cofibration $\m A \acof P_{\m\Gm}\m A$.
\item If \sC has a \shnno, then it interprets the natural numbers type.
\end{itemize}
In particular, we have the elimination rule for identity types:
\begin{equation}\label{eq:idelim}
  \inferrule{\Gm,(x:A),(y:A),(p:x\id y),\Theta \pr B \ty \\
  \Gm,(x:A),\Theta[x/y,\r_x/p] \pr (d:B[x/y,\r_x/p])}{\Gm,(x:A),(y:A),(p:x\id y),\Theta \vdash (J_d(x,y,p) : B)}
\end{equation}
The interpretation of this rule (together with its computation rule, $J_d(x,x,\r_x) \eq d$) must be a lift in the following square:
\begin{equation}\label{eq:idelimlift}
  \vcenter{\xymatrix{
      \m{\Gm,A,\Th[x/y,\r_x/p]}\ar[r]^-d \ar@{>->}[d]_{\m\r} &
      \m{\Gm,A,A,P_\Gm A,\Th,B}\ar@{->>}[d]^{p_B}\\
      \m{\Gm,A,A,P_\Gm A,\Th}\ar@{=}[r] \ar@{.>}[ur]|{\m {J_d}} &
      \m{\Gm,A,A,P_\Gm A,\Th}.
    }}
\end{equation}
Here $\Gm,A,A,P_\Gm A$ is shorthand for the context
\[ \Gm, (x:A),(y:A),(p:x\id y). \]
The left-hand map in~\eqref{eq:idelimlift} is the pullback of the acyclic cofibration
\[ \xymatrix{\m r : \m{\Gm,A} \ar@{>->}[r]^-{\sim} & \m{\Gm,A,A,P_\Gm A} = P_{\m\Gm} \m{\Gm,A}} \]
along the fibration $p_{\Th} : \m{\Gm,A,A,P_\Gm A,\Th} \fib \m{\Gm,A,A,P_\Gm A}$, and hence is an acyclic cofibration.
Since $p_B$ is a fibration, some lift in~\eqref{eq:idelimlift} exists; splitness gives a specified lift which is stable under pullback.

\begin{rmk}
  Since our type theory has dependent products, the additional context $\Th$ in~\eqref{eq:idelim} is unnecessary: it can be shifted into the type $B$.
  However, making it explicit shows why, even in the absence of dependent products, we need acyclic cofibrations to be stable under pullback along fibrations, as observed by \textcite{gg:idtypewfs}.
\end{rmk}


From now on we will use freely the internal type theory of a type-theoretic fibration category.
If it is split, then this can be obtained directly from \autoref{thm:syn-init}; otherwise it involves a coherence theorem.
From the category-theoretic point of view, all that matters is that the semantics satisfies the bullet points listed above, which are independent of how the splitting is performed.

We will generally abuse notation by omitting the brackets $\m-$, identifying an object of $\sC$ with the type that represents it and a type (or context) with the object that interprets it.
Moreover, since the presence of an unchanged context of parameters $\Gm$ in type theory corresponds to working in the slice category $(\sC/\Gm)\f$, which is itself a perfectly good type-theoretic fibration category, we will also generally leave ambient contexts implicit.

\section{Homotopy type theory}
\label{sec:homotopy-type-theory}

We will now give some definitions and results for doing homotopy theory inside of type theory, many originally due to \textcite{voevodsky:github} but developed further by the author and others~\parencite{hott,hottbook}.
As we give each definition, we will explain its categorical meaning under the above semantics.

First of all, we can use the eliminator $J$ to define operations of concatenation, inversion, and so on for paths in type theory, which categorically interpret to the morphisms $c$, $v$, and so on considered in \S\S\ref{sec:ttfc}--\ref{sec:hothy-fibcat}.
For instance, the concatenation operation, which we denote
\[ (x:B),(y:B),(z:B),(p:x\id y),(q:y\id z) \vdash (p \cdot q:x\id z), \]
can be defined by $p\cdot q \eq J_p(y,z,q)$.
Comparing with~\eqref{eq:cdef}, we see that this produces exactly the concatenation morphism $c$ defined there, since we have
\begin{align*}
  \m{(x:B),(y:B),(p:x\id y)} &= P \m B\\
  \m{(x:B),(y:B),(z:B),(p:x\id y),(q:y\id z)} &= P \m B \times_{\m B} P \m B
\end{align*}
and so on.
The computation rule of identity types implies that $p\cdot \r_y \eq p$, which is the commutativity of the upper-left triangle in~\eqref{eq:cdef}.
Similarly, the path $\psi:(\r_x \cdot p \id  p)$ from the proof of \autoref{thm:path-fact} can be defined type-theoretically as $J_{\r_{\r_x}}(x,y,p)$.

Another important operation which we will need later is \emph{transport}: given any dependent type $(x:A)\vdash B(x)\ty$, we have a term
\[ (x:A),(y:A),(p:x\id y),(b:B(x)) \vdash (p_*b : B(y)) \]
defined by $p_* b \eq J_b (x,y,p)$.
The morphism $t$ in the proof of the Acyclic Fibration Lemma (\ref{thm:afib}) is an instance of transport.

We can then rephrase many of the proofs in \S\S\ref{sec:ttfc}--\ref{sec:hothy-fibcat} in terms of the internal type theory.
In particular, as remarked there, the proof of \autoref{thm:path-fact} is a direct translation of the corresponding type-theoretic proof by \textcite{gg:idtypewfs}.
Thus, when the latter is interpreted in the internal language of \sC, it becomes precisely the proof given in \S\ref{sec:ttfc}.
Working through the correspondence between the two is a good exercise in understanding how the internal type theory translates into category theory.

Now for any type $A$, consider the type
\[ \iscontr(A) \eq \sum_{x:A}\; \prod_{y:A} (x\id y). \]
Categorically, ${\iscontr(A)}$ is the dependent product of the path object $P  A$ along one projection $A \times A \fib A$.
By adjunction, to give a global element of ${\iscontr(A)}$ (that is, a morphism $1 \to {\iscontr(A)}$) is to give a global element $1\to A$ together with a homotopy relating the composite $A\to 1\to A$ to the identity.
In other words, it is a witness exhibiting $A$ as homotopy equivalent to the terminal object; we say that such an $A$ is \emph{contractible}.

We also consider the type
\[ \isprop(A) \eq \prod_{x:A}\; \prod_{y:A} (x\id y). \]
Categorically, ${\isprop(A)}$ is the dependent product of $P A$ along the projection $A \times A \fib 1$.
By adjunction, to give a global element $1\to \isprop(A)$ is to give a section of the fibration $P A \fib A\times A$.
This implies that any two maps $f,g:X\toto A$ are homotopic, and is also implied by it (take $f$ and $g$ to be the two projections $A\times A \toto A$).
We call such an $A$ an \emph{h-proposition}, since to construct a term in such an $A$ gives no more information than that a certain property is true.

\begin{lem}\label{thm:isprop-inhabited-contr}
  We have
  \begin{align*}
    \isprop(A) &\to (A\to \iscontr(A))
    \qquad\text{and}\\
    (A\to \iscontr(A))&\to \isprop(A).
  \end{align*}
\end{lem}
\begin{proof}
  Given $p:\isprop(A)$ and $a:A$, we have $(a, y\mapsto p(a,y)) : \iscontr(A)$.
  Conversely, given $f:A\to \iscontr(A)$, we have $p:\isprop(A)$ where $p(x,y):x\id y$ is defined to be the composite of
  ${\snd(f(x))(x)} : x \id  \fst(f(x))$ with the inverse of ${\snd(f(x))(y)} : y\id \fst(f(x))$.
\end{proof}

We now interpret contractibility in the category $(\sC/B)\f$, which corresponds to working in the context $(b:B)$ in type theory.
Thus, for any fibration $A\fib B$, the fibration represented by the dependent type
\[ (b: B) \pr \iscontr(A(b)) \ty \]
has a section (which is equivalent to $\prod_{b: B} \iscontr(A)$ having a global element) precisely when $A\fib B$ is a homotopy equivalence in $(\sC/B)\f$.
By the Acyclic Fibration Lemma (\ref{thm:afib}), this is equivalent to saying that $A\fib B$ is an acyclic fibration, i.e.\ a fibration and a homotopy equivalence in \sC.

In fact, one can also prove \autoref{thm:afib} directly in the type theory, although we will not do so.
(A complete proof of \autoref{thm:fibcat} using type-theoretic methods, rather than the purely categorical ones of \S\ref{sec:hothy-fibcat}, has been given by \textcite{akl:faht}.)
This proof results in two terms with types
\begin{equation}\label{eq:ttafib}
  \begin{array}{rclr}
    \hequiv(\fst) &\too& \Big(\sprod{x:A} \iscontr (P(x))\Big) & \quad\text{and}\\
    \Big(\sprod{x:A} \iscontr (P(x))\Big) &\too& \hequiv(\fst).
  \end{array}
\end{equation}
Here $\fst: \sum_{x:A} P(x) \to A$ is the first projection, and for any $f:A\to B$,
\begin{equation}
  \hequiv(f) \eq \sum_{g:B\to A} \left(\left( \prod_{y:B} (f(g(y)) \id  y)\right) \times \left(\prod_{x:A} (g(f(x)) \id  x) \right)\right)
\end{equation}
is the type of ``homotopy equivalence data'' for $f$.
Of course, $\hequiv(f)$ has a global element precisely when $f$ is a homotopy equivalence.
Thus, since the existence of the morphisms~\eqref{eq:ttafib} imply that $\hequiv(\fst)$ has a global element if and only if $\sprod{x:A} \iscontr (P(x))$ does, \autoref{thm:afib} follows---but the type-theoretic proof actually says rather more than this.

However, $\hequiv(f)$ is not especially well-behaved as a type.
Specifically, because a given map can admit multiple inequivalent choices of ``homotopy equivalence data'', it is problematic to regard $\hequiv (f)$ as the mere assertion ``$f$ is a homotopy equivalence''.
One possible replacement is obtained by noting that by the 2-out-of-3 property, $f:A\to B$ is a homotopy equivalence just when the fibration half of its ``mapping path space'' constructed in \autoref{thm:path-fact} is an acyclic fibration.
In type theory, this fibration is $\m{\ssum{a:A} (f(a)\id b)} \to \m{B}$.
Thus, by the type-theoretic proof of \autoref{thm:afib}, $f$ is an equivalence just when the type
\begin{equation}
  \prod_{b:B} \iscontr\left( \sum_{a:A} (f(a)\id b) \right)\label{eq:contrmap}
\end{equation}
is inhabited.
The type~\eqref{eq:contrmap} is better-behaved than $\hequiv (f)$, but we will use instead the following definition, which is also well-behaved and easier to work with.
(It was first suggested in this context by Andr\'e Joyal.)
\begin{equation}
  \isequiv(f) \eq \left(\sum_{s: B\to A} \,\prod_{b: B} (f(s(b)) \id  b)\right) \,\times\,
  \left(\sum_{r: B\to A}\, \prod_{a: A} (r(f(a)) \id  a) \right).\label{eq:hiso}
\end{equation}
To give a global element of $\isequiv(f)$ is to give a homotopy section and a homotopy retraction of $f$.
It is easy to define a term of type $\hequiv (f) \to \isequiv (f)$ by taking $s$ and $r$ to be the same.
In the other direction, given $((s,p),(r,q)):\isequiv (f)$, we can first construct a term $u:\sprod{b} (s(b)\id r(b))$, by concatenating $\ap_r(p_b)$ with the inverse of $q_{s(b)}$:
\[ s(b) \id  r(f(s(b))) \id  r(b). \]
From this we obtain a term $v:\sprod{a} (s(f(a))\id a)$ by concatenating $u_{f(a)}$ with $q_a$, so that $(s,u,v):\hequiv (f)$.
Thus we have $\isequiv(f) \to \hequiv(f)$ also; in particular, one has a global element if and only if the other does.

The most important advantage of~\eqref{eq:contrmap} and~\eqref{eq:hiso} is that, at least under an additional natural assumption, they are h-propositions.
The necessary assumption is called \emph{function extensionality}: it specifies the path objects of function spaces, including dependent products (up to equivalence), which plain type theory leaves undetermined.
Function extensionality has several forms, which we now explain.
Note first that there is always a canonical term
\begin{equation*}
  \big(f,g : \sprod{a:A} B(a)\big)
  \pr \big( \happly : (f\id g) \too \sprod{a:A} (f(a) \id  g(a))\big)
\end{equation*}
defined by $\happly(p) \eq J_{\lambda a.\r_{f(a)}}(f,g,p)$.
The traditional meaning of ``function extensionality'' is simply the existence of a function in the opposite direction to \happly.
However, in homotopy type theory, where the types $(f\id g)$ and $\sprod{a:A} (f(a) \id  g(a))$ may contain higher information, we need to know furthermore that such a function is actually an \emph{inverse} to \happly.

Voevodsky has shown that this strong form of function extensionality is in fact equivalent to the following even weaker-looking form:
\begin{equation}
  \funext \,:\, \sprod{a:A}\iscontr(B(a)) \to \iscontr\big(\sprod{a:A} B(a)\big).\label{eq:funext}
\end{equation}
It is easy to construct~\eqref{eq:funext} under the ``naive'' assumption that there exists a function in the opposite direction to \happly, while conversely we have:

\begin{thm}[Voevodsky]\label{thm:strong-funext}
  Assuming~\eqref{eq:funext}, the function \happly is an equivalence.
\end{thm}
\begin{proof}
  We sketch the proof informally; a Coq formalization can be found in \textcite{voevodsky:github,hott} (see also \textcite{pll:funext-blog}).
  First, for any $f:\sprod{x:A} B(x)$, we have an equivalence
  \begin{equation}\label{eq:fppf-eqv}
    \sum_{g:\prod_{x:A} B(x)} \; \prod_{x:A} \, (f(x) \id  g(x))
    \quad\simeq\quad
    \prod_{x:A}\, \sum_{y:B(x)} (f(x) \id  y).
  \end{equation}
  From left to right, we send $(g,h)$ to $\lambda x. (g(x),h(x))$, while from right to left we send $k$ to $(\lambda x.\fst(k(x)), \lambda x. \snd(k(x)))$.
  With definitional $\eta$-conversion for both dependent sums and products, these functions are in fact a judgmental isomorphism (i.e.\ the composites in either direction are judgmentally equal to identities).\footnote{Without definitional $\eta$-conversion for dependent sums, and without knowing already the conclusion of this theorem, we only have that the left-hand side is a homotopy retract of the right-hand side, but this is sufficient for the argument.
    However, we do need at least propositional $\eta$-conversion for dependent products.}

  Second, the type $\sum_{y:B(x)} (f(x) \id  y)$ on the right-hand side of~\eqref{eq:fppf-eqv} is always contractible; this is essentially an expression of the induction principle for identity types.
  Thus, by~\eqref{eq:funext}, so is the entire right-hand side $\prod_{x:A}\, \sum_{y:B(x)} (f(x) \id  y)$, and hence so must be the left-hand side.
  However, by the same argument in reverse, this implies an ``induction principle'' for pointwise paths: given a dependent type
  \[ \big(f,g : \sprod{a:A} B(a)\big),\, \big(h : \sprod{a:A}(f(x)\id g(x))\big)
  \pr Q(f,g,h) \ty
  \]
  along with a term $d:Q(f,f, \lambda x.\r_{f(x)})$, we have a ``$J$-term'' inhabiting $Q(f,g,h)$, which computes (at least modulo a path) to $d$ when applied to $(f,f, \lambda x.\r_{f(x)})$.
  But now the types $(f\id g)$ and $\sprod{a:A}(f(x)\id g(x))$ have the same induction principle, hence must be equivalent.
\end{proof}

In the internal language of a type-theoretic fibration category \sC, function extensionality~\eqref{eq:funext} means that for fibrations $P\xfib{f} X\xfib{g} A$, there is a map
\begin{equation}
  \Pi_g (\iscontr_X(P)) \to \iscontr_A(\Pi_g P).\label{eq:catfunext}
\end{equation}
By the Yoneda lemma and the definition of $\Pi_g$, this means that for any $h: B\to A$, if there exists a map from $h^* X$ to $\iscontr_X(P)$ over $X$, then there exists a map from $B$ to $\iscontr_A(\Pi_g P)$ over $A$.
And by the above characterization of $\iscontr$, slicing, and preservation of all structure by pullback, this means that if the pullback $h^* P \to h^* X$ is an acyclic fibration, then so is $h^* (\Pi_g P) \to B$.
In particular, this means that whenever $f: P\to X$ is an acyclic fibration, then so is $\Pi_g(f)$.
However, this special case implies the general one, by the Beck-Chevalley condition for dependent products.
Thus we have:

\begin{lem}\label{thm:cat-funext}
  Function extensionality holds in the internal type theory of a type-theoretic fibration category if and only if dependent products along fibrations preserve acyclicity of fibrations.
\end{lem}

More precisely, if the latter condition holds in \sC, we can find morphisms~\eqref{eq:catfunext}.
Regarding \sC (or a split replacement of it) as equipped with such morphisms, it lives in the category whose initial object is the syntactic category of a type theory with function extensionality.
Thus, we can interpret the latter type theory into \sC.

\begin{rmk}\label{thm:ttmc-funext}
  If the acyclic fibrations are the right class in a weak factorization system, then this condition is equivalent to requiring pullback along fibrations to preserve the corresponding left class (the ``cofibrations'').
  Thus, by \autoref{thm:ttmc-piquillen}, it holds in any type-theoretic model category satisfying the condition of \autoref{rmk:ttmc-whe-he} that weak equivalences between fibrant objects are homotopy equivalences.
\end{rmk}

Note that by Ken Brown's lemma (see e.g.~\textcite[{}1.1.12]{hovey:modelcats}), dependent product along a fibration in a type-theoretic model category preserves weak equivalences.
The following lemma says that the same is true in any type-theoretic fibration category satisfying function extensionality.

\begin{lem}\label{thm:funext-forallequiv}
  Given $(a:A) \pr (f_a : B(a) \to C(a))$ such that $(a:A) \pr \isequiv(f_a)$ holds, then $\Pi f:\sprod{a} B(a) \to \sprod{a} C(a)$ defined by $\Pi f(h)(a) \eq f_a(h(a))$ is also an equivalence.
\end{lem}
\begin{proof}
  Since $\isequiv(f_a) \to \hequiv (f_a)$, we have
  \[(a:A)\pr (g_a : C(a) \to B(a))\]
  such that
  \begin{align*}
    (a:A),(b:B(a)) &\pr (p_{a,b}:g_a(f_a(b))\id b) \qquad\text{and}\\
    (a:A),(c:C(a)) &\pr (q_{a,c}:f_a(g_a(c))\id c).
  \end{align*}
  Define $\Pi g:\sprod{a} C(a) \to \sprod{a} B(a)$ by $\Pi g(k)(a) \eq g_a(k(a))$.
  Then for any $h:\sprod{a} B(a)$ and $a:A$ we have
  \[\Pi g(\Pi f(h))(a) \eq g_a(\Pi f(h)(a)) \eq g_a(f_a(h(a))),\]
  so that $p_{a,h(a)}:\Pi g(\Pi f(h))(a) \id  h(a)$.
  By function extensionality, therefore, $\Pi g(\Pi f(h))\id h$.
  The other side is analogous, so $\Pi f$ is a homotopy equivalence and hence an equivalence.
\end{proof}

Now we sketch the proof of the following fact referred to above.

\begin{lem}\label{thm:isprop-isequiv}
  Assuming function extensionality, for any $f:A\to B$ we have $\isprop(\isequiv(f))$.
\end{lem}
\begin{proof}
  By \autoref{thm:isprop-inhabited-contr}, we may extend the context by $e:\isequiv(f)$ and seek a term inhabiting $\iscontr(\isequiv(f))$.
  It is easy to show that a cartesian product of contractible types is contractible, so we may deal separately with the two factors in $\isequiv(f)$.
  For the first, function extensionality implies that $\sprod{b:B} (f(s(b))\id b)$ is equivalent to $f\circ s \id  1_B$, so that the first factor in~\eqref{eq:hiso} is equivalent to
  \begin{equation}
    \ssum{s:B\to A} (f\circ s \id  1_B).\label{eq:postcompfiber}
  \end{equation}
  But this is just $\ssum{s:B\to A} (F(s) \id  1_B)$, where $F:(B\to A) \to (B\to B)$ is post-composition with $f$.
  It is easy to show that $F$ is an equivalence if $f$ is.
  Thus, $F$ satisfies~\eqref{eq:contrmap}, so that~\eqref{eq:postcompfiber} is contractible as desired.
  Contractibility of the other half of $\isequiv (f)$ is nearly identical.
\end{proof}

We also observe the following.

\begin{lem}\label{thm:isprop-isprop}
  Assuming function extensionality, for any $A$ we have $\isprop(\isprop(A))$.
\end{lem}
\begin{proof}
  Suppose given $h,k:\isprop(A)$; we must construct a term inhabiting $(h\id k)$.
  By function extensionality, it suffices to extend the context by $a,b:A$ and construct a term inhabiting $h(a,b)\id k(a,b)$.

  To start with, we claim that in the context of $a,b:A$ and $p:a\id b$ we have a term inhabiting $p \id  h(a,a)^{-1} \cdot h(a,b)$, where $\cdot$ denotes path concatenation and $(-)^{-1}$ denotes path inversion.
  This follows from the eliminator $J$, for when $a\eq b$ and $p$ is the reflexivity path, then the type desired reduces to $\r_a \id  h(a,a)^{-1} \cdot h(a,a)$, which is inhabited by the easy proof that inversion is an inverse for concatenation.

  Finally, letting $p$ be $h(a,b)$ and $k(a,b)$ successively, we have
  \[ h(a,b) \id  h(a,a)^{-1} \cdot h(a,b) \id  k(a,b) \]
  as desired
\end{proof}

\begin{lem}\label{thm:prop-forall}
  If $\sprod{x:A} \isprop(B(x))$, then $\isprop(\sprod{x:A} B(x))$.
\end{lem}
\begin{proof}
  Given $f,g:\sprod{x:A} B(x)$, to show $f\id g$, by function extensionality it suffices to show $f(x)\id g(x)$ for any $x$, but this follows from the assumption.
\end{proof}

Finally, for types $A$ and $B$ we define the ``space of equivalences'' from $A$ to $B$ to be the dependent sum type
\[\equiv(A,B) \eq \sum_{f: A\to B} \isequiv(f). \]
This will play an essential role in the univalence axiom (\S\ref{sec:univalence-axiom}).

\section{Universes}
\label{sec:tyuniverses}

Voevodsky's univalence axiom for homotopical type theory depends on a \emph{universe} or ``type of (small) types''.
We denote such a type by ``\type'', and assume that it is equipped with an ``\`a la Tarski'' coercion from terms of type \type to types:
\begin{equation}
  (A:\type) \pr \el(A) \ty.\label{eq:el}
\end{equation}
Moreover, the type-forming operations should be reflected by operations on \type, as shown in \autoref{fig:univop},
which coerce, definitionally, to the actual type-forming operations, as shown in \autoref{fig:univopbeta}.
(One can then make the coercion \el implicit, so that terms of type \type appear to be literally identified with types.
We will do this in later sections.)

\begin{figure}
  \centering
  \begin{align}
    &\pr (\mathsf{u} : \type)\label{eq:unitop}\\
    (A:\type),\, (B:\el(A) \to \type) &\pr (\Sigma(A,B) : \type)\label{eq:Sigmaop}\\
    (A:\type),\, (B:\el(A) \to \type) &\pr (\Pi(A,B) : \type)\label{eq:Piop}\\
    (A:\type),\, (x:\el(A)),\, (y:\el(A)) &\pr (\mathsf{Id}(A,x,y) : \type)\label{eq:Idop}\\
    &\pr (\mathsf{N} : \type).\label{eq:natop}
  \end{align}
\caption{Operations on the universe type}
\label{fig:univop}
\end{figure}

\begin{figure}
  \centering
  \begin{align}
    \el(\mathsf{u}) &\eq \mathsf{unit} \label{eq:unitopbeta}\\
    \el(\Sigma(A,B)) &\eq \sum_{x:\el(A)} \el(B(x)) \label{eq:Sigmaopbeta}\\
    \el(\Pi(A,B)) &\eq \prod_{x:\el(A)} \el(B(x)) \label{eq:Piopbeta}\\
    \el(\Id(A,x,y)) &\eq (x\id y) \label{eq:Idopbeta}\\
    \el(\mathsf{N}) &\eq \mathsf{nat}. \label{eq:natopbeta}
  \end{align}
\caption{Coercion identities for the universe type}
\label{fig:univopbeta}
\end{figure}
\noeqref{eq:Sigmaopbeta,eq:Piopbeta,eq:Idopbeta}

In a type-theoretic fibration category, the dependent type~\eqref{eq:el} must be represented by a fibration $p:\Util \fib U$.
After fixing such a fibration $p$, we refer to the class of all pullbacks of $p$ as \textbf{small fibrations}.
Of course, an object $A$ is \emph{small} if the fibration $A\fib 1$ is small.
Note that in a small fibration $B\fib A$, the \emph{object} $A$ may not be small.

The (non-split) category-theoretic version of~\eqref{eq:unitop}--\eqref{eq:natop} is the following.

\begin{defn}\label{def:univ}
  A fibration $p: \Util\fib U$ in a type-theoretic fibration category \sC is a \textbf{universe} if the following hold, where ``small fibration'' means ``a pullback of $p$''.
  \begin{enumerate}
  \item Small fibrations are closed under composition and contain the identities.\label{item:u1}
  \item If $f: B\fib A$ and $g: A\fib C$ are small fibrations, so is $\Pi_g f \fib C$.\label{item:u2}
  \item If $A\fib C$ and $B\fib C$ are small fibrations, then any morphism $f: A\to B$ over $C$ factors as an acyclic cofibration followed by a small fibration.\label{item:u3}
  \end{enumerate}
\end{defn}

In the presence of an \shnno, we may of course want to assume that it is also small.

\begin{rmk}
  \autoref{def:univ}\ref{item:u3} clearly implies that any small fibration $A\fib C$ has a small path fibration $P_C A \fib A\times_C A$.
  The converse holds in the presence of~\ref{item:u1}, using the construction of \autoref{thm:path-fact}.
\end{rmk}

The assumptions in \autoref{def:univ} enable us to choose particular morphisms representing the operations~\eqref{eq:unitop}--\eqref{eq:natop}, as follows.
\begin{itemize}
\item The identity $1\to 1$ is small, hence is the pullback of $p$ along some map $1\to U$.
  Any such map can represent~\eqref{eq:unitop}.
  The case of~\eqref{eq:natop} is similar.
\item For~\eqref{eq:Sigmaop}, let $U^{(1)}$ interpret the context $(A:\type),(B:\el(A) \to \type)$.
Categorically, it is the local exponential
\begin{equation}
  U^{(1)} = (U\times U \to U)^{(\Util\to U)}.\label{eq:udnt}
\end{equation}
Its universal property is that morphisms $A\to U^{(1)}$ correspond to pairs $(a,b)$ where $a:A\to U$ and $b: a^*\Util \to U$.
In particular, it comes with a universal such pair $a_0:U^{(1)} \to U$ and $b_0 : (a_0)^*\Util \to U$, inducing a pair of composable small fibrations $(b_0)^*\Util\fib (a_0)^*\Util\fib U^{(1)}$.
By \autoref{def:univ}\ref{item:u1}, the composite $(b_0)^*\Util\fib U^{(1)}$ is also small; hence there exists a morphism $\Sigma : U^{(1)} \to U$ and a pullback square
\begin{equation}
  \vcenter{\xymatrix@R=1pc@C=1.5pc{
      (b_0)^*\Util\ar[r]\ar@{->>}[d] \pullbackcorner &
      \Util\ar[dd]\\
      (a_0)^*\Util\ar@{->>}[d] &
      \\
      U^{(1)}\ar[r]_{\Sigma} &
      U.
    }}
\end{equation}
Any such map $\Sigma$ can represent~\eqref{eq:Sigmaop}.
\item Similarly, for~\eqref{eq:Piop} we note that by \autoref{def:univ}\ref{item:u2}, the dependent product of $(b_0)^*\Util\fib (a_0)^*\Util$ along $(a_0)^*\Util\fib U^{(1)}$ is a small fibration over $U^{(1)}$, hence is classified by some map $\Pi: U^{(1)} \to U$.
\item Finally, for identity types, we consider the object $\Util\times_U \Util$, which represents the context $(A:\type), (x:\el(A)), (y:\el(A))$.
This has the universal property that morphisms $A \to \Util\times_U \Util$ correspond to triples $(a,x,y)$ where $a:A\to U$ and $x$ and $y$ are both sections of $a^*\Util$.
Since the fibration $\Util\times_U \Util \fib U$ is small, as is $p:\Util \to U$, and $U$ is a universe, we can factor the diagonal $\Util \to \Util\times_U \Util$ to yield a path object $P_U \Util$ for which the projection $P_U \Util \fib \Util \times_U \Util$ is a small fibration.
Thus, it has some classifying morphism $\Id:\Util \times_U \Util \to U$, which can represent~\eqref{eq:Idop}.
\end{itemize}

This suggests the following definition.

\begin{defn}\label{def:cloven-universe}
  A universe $p: \Util\fib U$ in a cloven type-theoretic fibration category \sC is \textbf{cloven} if $p$ and $U\fib 1$ are equipped with fibration structures, and we have specified morphisms $1\to U$, $U^{(1)}\to U$, $U^{(1)}\to U$, and $\Util\times_U \Util \to U$ implementing the unit type, dependent sums, dependent products, and identity types as above.
  If \sC has a small \shnno, we require an additional morphism $1\to U$ classifying it.
\end{defn}

Thus, any universe admits some cloven structure.
However, the definitional equalities~\eqref{eq:unitopbeta}--\eqref{eq:natopbeta} may not hold in general; thus we introduce a name for the case when they do.

\begin{defn}
  A cloven universe $p:\Util \fib U$ in a split type-theoretic fibration category is \textbf{split} if the specified pullbacks of $p$ along the universe structure morphisms are equal, as structured fibrations, to the specified structured fibrations over $1$, $U^{(1)}$, or $\Util\times_U \Util$ coming from the ambient split structure.
\end{defn}

As usual, a universe type in type theory yields a split universe in the syntactic category, while the coherence theorems imply that any universe can be made split in an equivalent category.
Thus, type theory containing a type universe \type can be interpreted into any type-theoretic fibration category containing a categorical universe $p:\Util\to U$.

\begin{rmk}
  It is possible to make a universe $U$ into an internal category in \sC, and the universe structure into internal operations on this category, reflecting the type-theoretic structure of \sC itself.
  This is analogous to how the subobject classifier in a topos automatically becomes an internal complete Heyting algebra, reflecting the logical operations on subobjects in the topos.
  (However, this structure does not capture splitness.)
\end{rmk}

Now, note that not every type is of the form $\el(A)$ for some $A:\type$.
In particular, \type itself cannot be of that form without leading to inconsistency; thus \type is only a universe of ``small types''.
Thus, it is natural to introduce a hierarchy of universes with $\type_n : \type_{n+1}$ for all $n$, each with their own coercion $\el_n$, and ``level-raising'' operations
\begin{equation}
  \up:\type_n \to \type_{n+1}.\label{eq:up}
\end{equation}
We generally require \up to respect the coercions to types:
\begin{equation}
  \el_{n+1}(\up(A)) \eq \el_n(A)\label{eq:elup}
\end{equation}
and also all the type-forming operations, in the sense that, for instance,
\begin{equation}
  \Sigma(\up(A),\lambda x. \up(B(x))) \eq \up(\Sigma(A,B))\label{eq:coerceSigma}
\end{equation}
and so on for all the others.
(We do not need to worry about the eliminators, even in the case of identity types, because they never come into play until \emph{after} the coercions $\el_n$ are applied.)

On the categorical side, consider for simplicity the case of two universe objects, say $U$ and $U'$, with coercions $\el$ and $\el'$.
To have $\type : \type'$ (or, more precisely, a term $\mathsf{U}:\type'$ such that $\el'(\mathsf{U})\eq \type$), we must assume that $U$ is a $U'$-small object, i.e.\ the fibration $U\fib 1$ is a pullback of $\Util' \fib U'$.
And for~\eqref{eq:up}, we need a map $i:U \to U'$ which fits into a pullback square
\begin{equation}
  \vcenter{\xymatrix{
      \Util\ar[r]^-{\itil}\ar[d]_p \pullbackcorner &
      \Util'\ar[d]^{p'}\\
      U\ar[r]_-i &
      U'.
    }}\label{eq:univemb}
\end{equation}
Such a pullback square exists precisely when every $U$-small fibration is also $U'$-small.
If $U$ and $U'$ are split, then to obtain~\eqref{eq:elup} we need~\eqref{eq:univemb} to exhibit $\Util$ as the specified pullback $i^*\Util'$ from the split structure of \sC.

Finally, for~\eqref{eq:coerceSigma} to hold, the square
\begin{equation}
  \vcenter{\xymatrix{
      U^{(1)}\ar[r]^{i^{(1)}}\ar[d]_{\Sigma} &
      (U')^{(1)}\ar[d]^{\Sigma'}\\
      U\ar[r]_i &
      U'
    }}\label{eq:uembstr}
\end{equation}
must commute. 
Here the top morphism $i^{(1)}: U^{(1)} \to (U')^{(1)}$ is most easily described representably: given a pair $(X \xto{a} U,\, a^*\Util\xto{b} U)$ corresponding to a morphism $X\to U^{(1)}$, the pullback square~\eqref{eq:univemb} tells us that $a^* \Util \cong (i a)^* \Util'$, so the pair
\[(X\xto{a} U \xto{i} U', \, (i a)^* \Util' \cong a^*\Util \xto{b} U \xto{i} U' )
\]
corresponds to a morphism $X \to (U')^{(1)}$.
The equations analogous to~\eqref{eq:coerceSigma} are similar; this leads to the following definition.

\begin{defn}
  If $U$ and $U'$ are cloven (or split) universes and we are given a morphism $1\to U'$ classifying $U'$, and a specified pullback square~\eqref{eq:univemb}, such that~\eqref{eq:uembstr} commutes, as well as the analogous squares for the unit type, dependent products, and identity types (and the natural numbers type, if present), we say that $i: U\to U'$ is an \textbf{embedding of (cloven) universes}.
\end{defn}

\begin{rmk}
  The case of the unit type just means that the composite $1\xto{e} U \xto{i} U'$ is $1\xto{e'} U'$.
  This is easy to obtain, if it doesn't hold already, by simply defining $e'$ to be $i \circ e$.
  The same holds for a natural numbers type, if present.
\end{rmk}

The same principle applies to arbitrarily many nested universes: we require all morphisms between them to be universe embeddings for some fixed cloven structure on each.
We also require that for any pair of such embeddings $U \xto{i} U' \xto{i'} U''$, if $1 \xto{u} U'$ is the specified morphism with $u^*\Util' \cong U$ (witnessing $\type : \type'$), then the composite $1\xto{u} U' \xto{i'} U''$ must be the specified morphism witnessing $\type : \type''$.
But like the unit type, this is easy to obtain by choosing the latter morphism appropriately.

\begin{rmk}\label{thm:no-new-names}
  Suppose that $i: U\into U'$ is monic, and also that it \emph{adds no new names} in the sense that if $f:X\to U'$ is such that $f^*\Util' \fib X$ is $U$-small, then $f$ factors through $U$.
  Then any morphism implementing a type-forming operation for $U'$ must preserve $U$-smallness, and hence induce a unique corresponding such morphism for $U$ which commutes with $U\into U'$.
  Thus, if $U'$ is cloven (or split), there is a unique way to make $U$ cloven (or split) such that $i$ becomes a universe embedding.

  More generally, this technique can be applied to any collection of universes having a largest element, but it does not work if there are countably many universes not all contained in an ``$\omega^{\mathrm{th}}$'' one.
  However, this is rarely a problem in practice, since any \emph{particular} construction requires only finitely many universes.
\end{rmk}

\section{The univalence axiom}
\label{sec:univalence-axiom}

Now, just as type theory without function extensionality does not determine the identity types of function types (including dependent products), ordinary type theory with a universe does not determine the identity types of the universe.
We now describe Voevodsky's univalence axiom, which remedies this.

Suppose \type is a particular fixed universe.
First of all, since identity maps are equivalences, we have a canonical term
\[ (A: \type) \pr (\idequiv_A : \equiv(A,A)). \]
Using the elimination rule $J$, we obtain a canonical term
\[ (A: \type),\, (B: \type) \pr \big(\ptoe_{A,B} : (A\id B) \to \equiv(A,B)\big). \]
Of course, $(A\id B)$ denotes the identity type of the universe \type.
We say the \emph{univalence axiom holds} for the universe \type, or that \type \emph{is univalent}, if there is a term
\[ \mathsf{univalence} : \prod_{A,B} \isequiv(\ptoe_{A,B}).\]

In categorical terms, this states that the canonically defined map $P U \to E$ over $U\times U$ is an equivalence, where $E\to U\times U$ is the fibration representing the dependent type
\[ (A: \type),\, (B: \type) \pr (\equiv(A,B) : \type). \]
Since this map $P U \to E$ is defined by the lifting property of $P U$ (i.e.\ path induction), by the 2-out-of-3 property this is equivalent to saying that the map $U \to E$, which sends a type $A$ to its identity equivalence, is itself an equivalence.

\begin{rmk}
  Like function extensionality, univalence is an \emph{axiom} in type theory, i.e.\ a constant term belonging to some type.
  \autoref{thm:syn-init} with axioms implies that if the univalence axiom holds in a type-theoretic fibration category \sC, in the sense that $P U \to E$ is an equivalence, then its internal type theory may be taken to satisfy the univalence axiom (for that universe).
\end{rmk}

We now consider several examples.

\begin{eg}
  Let \sC be an elementary topos with the trivial model structure.
  Thus all morphisms are fibrations, all homotopies are identities, and the equivalences are the isomorphisms.
  Let $U=\Omega$ be the subobject classifier, with $\Util = 1 \to \Omega$ the universal subobject.
  Then $U$ is a universe whose small fibrations are exactly the monomorphisms.
  A natural numbers object in \sC, in the usual topos-theoretic sense, is in particular an \shnno, but it is not of course small for this universe.

  Since this universe classifies only monomorphisms, the types which belong to this universe $U$ in the internal logic are all h-propositions.
  This implies that $\equiv(A,B)$ is equivalent to the type of bi-implications, $(A\to B) \times (B\to A)$.
  It is well-known that bi-implication on the subobject classifier is the same as equality, so this universe is univalent.

  In particular, we can take $\sC=\nSet$, in which case $\Omega = \{\top,\bot\}$.
  As remarked in \S\ref{sec:catsem}, the category \nSet has a canonical splitting (although surprisingly the universe $\Omega$ is not canonically split unless we make some unnatural choices).
  Then the small fibrations are the monomorphisms, the only small objects are $\emptyset$ and $1$, and the universe is univalent.
\end{eg}

\begin{eg}
  In the model category of groupoids, we can take $U$ to be the groupoid of sets of rank $<\ka$, for some inaccessible cardinal \ka, with $\Util$ the corresponding groupoid of pointed sets.
  Then the $U$-small fibrations are precisely the discrete fibrations with fibers of cardinality $<\ka$, which are closed under all the relevant category-theoretic operations.
  Moreover, functors $A\to U$ are precisely pseudofunctors $A\to\nGpd$ which happen to take values in sets of rank $<\ka$, so the canonical splitting described in \autoref{thm:gpd-split} restricts to a split universe structure on $U$.

  Tracing through the construction of the universal space of equivalences, we find that the fiber of $E\fib U\times U$ over a pair of sets $(a,b)$ is the set of isomorphisms from $a$ to $b$.
  Since this is also the hom-set $U(a,b)$, with the obvious constructions, the map $P U \to E$ is in fact an isomorphism.
  Thus, this universe is univalent.

  This universe is called $\mathrm{Gpd}_\triangle(V_\ka)$ in~\textcite{hs:gpd-typethy}.
  Since it is not discrete, it is not an element of any larger univalent universe.
  But it does contain a smaller univalent universe, namely the universe $\Omega = \{\top,\bot\}$ which classifies monic fibrations.

  There are universes in the groupoid model which contain non-discrete groupoids, such as the groupoid of all \emph{groupoids} of rank $<\ka$, but these universes are not univalent.
  Note that even this universe classifies only \emph{split} fibrations with $\ka$-small fibers, whereas we have allowed arbitrary isofibrations to represent dependent types.
  (The original groupoid model of~\textcite{hs:gpd-typethy} involved only split fibrations.)
\end{eg}

Finally and most importantly, Voevodsky has shown that in simplicial sets, there is a \emph{universal Kan fibration} $p: \Util\to U$ such that $U$ is a Kan complex, and every Kan fibration with fibers of cardinality $<\ka$ (for some chosen cardinal $\ka$) is $U$-small.
This universe object is moreover univalent; see~\textcite{klv:ssetmodel} for a detailed exposition and~\textcite{moerdijk:univalence} for an alternative proof.
If \ka is inaccessible, such fibrations are closed under category-theoretic operations, and if $\la<\ka$ is also inaccessible, we have a universe embedding $U_\la \into U_\ka$ (either from \autoref{thm:no-new-names} or by choosing the structure carefully).
Thus, invoking the coherence theorems, one has:

\begin{thm}[Voevodsky]
  The model category \sSet supports a model of intensional type theory with a unit type, dependent sums and products, identity types, and with as many univalent universes as there are inaccessible cardinals.
\end{thm}

(The construction of~\textcite{klv:ssetmodel} requires an extra inaccessible to be the ``external universe'' needed for the coherence theorem.
The improved coherence theorem of~\textcite{lw:localuniv} eliminates this requirement.)

Since the homotopy theory of simplicial sets is a model for the \io-topos $\infty\mathrm{Gpd}$, we can say informally that the above model lives in that \io-topos.

We can obtain a few other examples easily from this one.
For instance, for any $n$, Voevodsky's universe has a subuniverse consisting of the $n$-truncated Kan fibrations (those whose fibers are homotopy $n$-types).
This universe is itself univalent and $(n+1)$-truncated, so we can obtain (for instance) a nested sequence of univalent universes of increasing truncation level as well as size.
(The universe of 0-truncated Kan fibrations is, of course, closely related to the groupoid of sets.)

Finally, we can pull back any univalent universe to any slice category.
However, it seems that until now, no other set-theoretic models of univalence have been known.

\begin{rmk}
  Voevodsky has also shown that the univalence axiom implies function extensionality (see~\textcite{voevodsky:github,hott}).
  Specifically, if there are two nested univalent universes, then function extensionality holds for all types belonging to the smaller universe.
  In what follows, we will need to apply function extensionality even for \type-valued functions (that is, dependent types).
  This can be deduced from a third nested univalent universe---or from the observation (\autoref{thm:ttmc-funext}) that any type-theoretic model category satisfies function extensionality.
\end{rmk}

\section{The Sierpinski \io-topos}
\label{sec:sierpinski}

We now move on to the main goal of the paper: constructing a new model of type theory with the univalence axiom in a category of inverse diagrams.
Before considering the general case, we treat a particular one in detail, which contains essentially all the ideas.
Let $\sC$ be a type-theoretic fibration category, and let $\sC^\bbtwo$ denote the category of arrows $(\alpha: A_1 \to A_0)$ of \sC.
We will construct a model of type theory in a subcategory of fibrant objects in $\sC^\bbtwo$.

\begin{defn}
  A morphism
  \[\vcenter{\xymatrix{
      A_1\ar[r]^\alpha\ar[d]_{f_1} &
      A_0\ar[d]^{f_0}\\
      B_1\ar[r]_\beta &
      B_0
    }}
  \]
  in $\sC^\bbtwo$ is a \textbf{Reedy fibration} if
  \begin{enumerate}
  \item $f_0$ is a fibration, and
  \item The induced map $A_1 \to A_0 \times_{B_0} B_1$ is a fibration.
  \end{enumerate}
  On the other hand, $f$ is a \textbf{Reedy acyclic cofibration} if $f_0$ and $f_1$ are acyclic cofibrations in \sC.
\end{defn}

\begin{rmk}
  Of course, an object $(\alpha: A_1 \to A_0)$ of $\sC^\bbtwo$ is \emph{Reedy fibrant} if $A\to 1$ is a Reedy fibration, which means that $A_0$ is fibrant (as is always the case) and $\alpha$ is a fibration.
  Thus, in the type theory of \sC, the Reedy fibrant objects of $\sC^\bbtwo$ can be regarded as \emph{two-type contexts} of the form
  \[ (a_0 : A_0), \; (a_1 : A_0(a_0)). \]
  This point of view will be crucial in what follows.
\end{rmk}

We write $(\sC^\bbtwo)\f$ for the full subcategory of $\sC^\bbtwo$ on the Reedy fibrant objects.

The following is easy and standard~\parencite{hovey:modelcats,hirschhorn:modelcats}.

\begin{lem}\label{thm:reedy-fact}
  A morphism is a Reedy acyclic cofibration if and only if it has the left lifting property with respect to Reedy fibrations.
  Every morphism in $\sC^\bbtwo$ factors as a Reedy acyclic cofibration followed by a Reedy fibration.
\end{lem}
\begin{proof}
  Given a square
  \begin{equation}
  \vcenter{\xymatrix@-.5pc{
      A\ar[r]\ar[d] &
      C\ar@{->>}[d]\\
      B\ar[r] &
      D
    }}\label{eq:rl}
  \end{equation}
  in which $A\to B$ is a Reedy acyclic cofibration and $C\to D$ is a Reedy fibration, we first define a lift
  \begin{equation}
  \vcenter{\xymatrix@-.5pc{
      A_0\ar[r]\ar@{>->}[d]_\sim &
      C_0\ar@{->>}[d]\\
      B_0\ar[r] \ar@{.>}[ur] &
      D_0
    }}\label{eq:rl0}
  \end{equation}
  and then a lift
  \begin{equation}
  \vcenter{\xymatrix@-.5pc{
      A_1\ar[r]\ar@{>->}[d]_\sim &
      C_1\ar@{->>}[d]\\
      B_1\ar[r] \ar@{.>}[ur] &
      C_0 \times_{D_0} D_1
    }}\label{eq:rl1}
  \end{equation}
  where the bottom map in~\eqref{eq:rl1} is defined using the diagonal lift in~\eqref{eq:rl0}.
  Together these form a lift in~\eqref{eq:rl}; thus Reedy acyclic cofibrations have the left lifting property with respect to Reedy fibrations.

  To factor $f: A\to B$, we first factor $A_0 \to B_0$ as
  \begin{equation}
    \label{eq:rf0}
    A_0 \acof C_0 \fib B_0
  \end{equation}
  and then factor the induced map $A_1 \to C_0 \times_{B_0} B_1$ as
  \begin{equation}
    \label{eq:rf1}
    A_1 \acof C_1 \fib C_0 \times_{B_0} B_1.
  \end{equation}
  This shows the second statement.
  By the retract argument, it follows that any map with the left lifting property against Reedy fibrations must be a retract of a Reedy acyclic cofibration, and hence itself a Reedy acyclic cofibration.
\end{proof}

Note that we have already used a Reedy factorization in the proof of \autoref{thm:fibhtpy}.

\begin{rmk}
  If \sC is a model category, then the Reedy fibrations are the fibrations in a model structure on $\sC^\bbtwo$ whose cofibrations and weak equivalences are both defined levelwise.
  If \sC is a type-theoretic model category, then so is $\sC^\bbtwo$.
  And if \sC is simplicial sets, then the Reedy model structure on $\sSet^\bbtwo$ presents the \io-category $\infty \mathrm{Gpd}^\bbtwo$.
\end{rmk}

\begin{thm}\label{thm:reedy-ttfc}
  If \sC is a type-theoretic fibration category, then so is $(\sC^\bbtwo)\f$.
\end{thm}
\begin{proof}
  We consider the axioms of \autoref{def:ttfc} in order.
  The terminal object is of course $1\fib 1$, giving~\ref{item:cat1}.
  The fibrations are the Reedy fibrations, while \autoref{thm:reedy-fact} identifies the acyclic cofibrations; thus for~\ref{item:cat2a} it suffices to verify that Reedy fibrations are stable under composition.
  If $B \xto{f} A \xto{g} C$ are Reedy fibrations, then $(gf)_0 = g_0 f_0$ is a fibration.
  Moreover, the induced map $B_1 \to B_0 \times_{C_0} C_1$ is the composite
  \begin{equation}
    B_1 \too B_0 \times_{A_0} A_1 \too
    B_0 \times_{A_0} (A_0 \times_{C_0} C_1) \xto{\cong}
    B_0 \times_{C_0} C_1\label{eq:rfib-compose}
  \end{equation}
  where the first map is a fibration since $f$ is a Reedy fibration, and the second is a fibration since it is a pullback of $A_1 \to A_0 \times_{C_0} C_1$, which is a fibration since $g$ is is a Reedy fibration.
  Thus, $g f$ is a Reedy fibration.
  
  Now since fibrations in \sC are closed under pullback and composition, if $f: A\fib B$ is a Reedy fibration, then $f_1$, being the composite
  \[ A_1 \fib A_0 \times_{B_0} B_1 \fib B_1, \]
  is also a fibration.
  Thus, Reedy fibrations are in particular levelwise fibrations.
  Since limits are also levelwise in $\sC^\bbtwo$, it follows that all pullbacks of Reedy fibrations between Reedy fibrant objects exist.
  This gives the first half of~\ref{item:cat3}; the rest is that a pullback of a Reedy fibration is again a Reedy fibration.
  Thus, suppose
  \[\vcenter{\xymatrix{
      P \ar[r]\ar[d] &
      B\ar@{->>}[d]\\
      A\ar[r] &
      C
    }}\]
  is a pullback diagram in $\sC^\bbtwo$, with $B\fib C$ a Reedy fibration.
  Then $P_0 \to A_0$ is a pullback of the fibration $B_0 \fib C_0$, hence a fibration.
  Now both squares below are pullbacks:
  \[\vcenter{\xymatrix{
      P_0 \times_{A_0} A_1 \ar[r]\ar[d] &
      P_0\ar[r]\ar[d] &
      B_0\ar[d]\\
      A_1\ar[r] &
      A_0\ar[r] &
      C_0
    }}\]
  hence so is the outer rectangle.
  But this is the same as the lower rectangle below:
  \[\vcenter{\xymatrix{
      P_1 \ar[d] \ar[r] & B_1 \ar[d] \\
      P_0 \times_{A_0} A_1\ar[r]\ar[d] &
      B_0 \times_{C_0} C_1\ar[r]\ar[d] &
      B_0 \ar[d]\\
      A_1 \ar[r] &
      C_1 \ar[r] &
      C_0
    }}\]
  and the lower-right square here is a pullback, hence so is the lower-left square.
  But the left-hand rectangle is also a pullback, hence so is the upper-left square.
  Thus $P_1 \to P_0 \times_{A_0} A_1$ is a pullback of the fibration $B_1 \fib B_0 \times_{C_0} C_1$, hence is also a fibration.
  
  For axiom~\ref{item:cat4},
  let $f: A\fib C$ and $g: B\fib A$ be Reedy fibrations between Reedy fibrant objects, and consider the diagram in \autoref{eq:small-exp}.
  \begin{figure}
    \[\vcenter{\xymatrix@R=1pc@C=1pc{
      Q \ar@{->>}[dr] \ar[dd] && \Pi_{\ftil} (Q)\ar@{->>}[dr] \\
      & P \ar[dd] \ar@{->>}[dr] \ar@{-->}[rr]^(0.3){\ftil} &&
      C_1 \times_{C_0} \Pi_{f_0} B_0 \ar[dr] \ar@{->>}'[d][dddd]\\
      B_1 \ar@{->>}[dr] && f_0^* \Pi_{f_0} B_0 \ar[dd] \ar[rr] && \Pi_{f_0} B_0 \ar@{->>}[dddd] \\
      & A_1 \times_{A_0} B_0 \ar@{->>}[dr] \ar@{->>}[dd] \\
      && B_0 \ar@{->>}[dd]_(0.3){g_0} \\
      & A_1 \ar@{->>}[dr] \ar@{->>}'[r][rr]^{f_1} && C_1 \ar[dr]\\
      && A_0 \ar@{->>}[rr]_{f_0} && C_0
    }}\]\caption{Construction of Reedy dependent products}\label{eq:small-exp}
  \end{figure}
  The objects $P$ and $Q$ are defined so as to make the squares
  \[\vcenter{\xymatrix{
      P\ar@{->>}[r]\ar[d] \pullbackcorner &
      f_0^* \Pi_{f_0} B_0\ar[d]\\
      A_1 \times_{A_0} B_0\ar@{->>}[r] &
      B_0
    }} \qquad\text{and}\qquad
  \vcenter{\xymatrix{
      Q\ar@{->>}[r]\ar[d] \pullbackcorner &
      P\ar[d]\\
      B_1\ar@{->>}[r] &
      A_1 \times_{A_0} B_0
    }}\]
  (which appear in \autoref{eq:small-exp}) pullback squares.
  These pullbacks exist in \sC because their bottom morphisms are fibrations: the first as it is a pullback of $A_1 \fib A_0$ (which is a fibration as $A$ is Reedy fibrant), and the second as $f$ is a Reedy fibration.

  By the pasting law for pullbacks, the left-hand face of \autoref{eq:small-exp} is a pullback.
  Since the front and right-hand faces are pullbacks by definition, so is the back face:
  \[\vcenter{\xymatrix{
      P\ar@{->>}[r]^-{\ftil}\ar[d] \pullbackcorner &
      C_1 \times_{C_0} \Pi_{f_0} B_0\ar[d]\\
      A_1 \ar@{->>}[r]_{f_1} &
      C_1.
    }}\]
  Thus the map \ftil, which is induced by the universal property of $C_1 \times_{C_0} \Pi_{f_0} B_0$, is also a fibration.
  Hence, the dependent product $\Pi_{\ftil} Q$ exists and is a fibration.

  Now we define $(\Pi_f B)_0 = \Pi_{f_0} B_0$ and $(\Pi_f B)_1 =\Pi_{\ftil}(Q)$; we have shown that $\Pi_f B \to C$ is a Reedy fibration.
  It is straightforward to verify that this is actually the dependent product of $g$ along $f$ in $\sC^\bbtwo$, giving axiom~\ref{item:cat4}.

  Finally, \autoref{thm:reedy-fact} gives axiom~\ref{item:cat7} for $(\sC^\bbtwo)\f$, while~\ref{item:cat8} follows from its truth in \sC and the fact that Reedy acyclic cofibrations are levelwise and Reedy fibrations are in particular levelwise.
\end{proof}

Now in order to interpret type theory in $(\sC^\bbtwo)\f$, we need it to be split.
We can, of course, apply the coherence theorems to it.
However, in \S\ref{sec:scones} it will be useful to know that a split structure on \sC \emph{directly} induces a split structure on $(\sC^\bbtwo)\f$ (or, actually, a generalization of it).
Moreover, an explicit description of this split structure will also help explain how the internal type theory of $(\sC^\bbtwo)\f$ can be interpreted in terms of the type theory of \sC, which will be useful in \S\ref{sec:univalence}.

We begin with the cloven structure.

\begin{defn}
  If \sC is cloven, then a \textbf{structured Reedy fibration} in $\sC^\bbtwo$ is a Reedy fibration $A\fib B$ together with fibration structures on $A_0 \fib B_0$ and on $A_1 \fib A_0 \times_{B_0} B_1$, where $A_0 \times_{B_0} B_1$ denotes the specified pullback of the structured fibration $A_0 \fib B_0$ along the map $B_1 \to B_0$.
\end{defn}

Note that if \sC is the cloven syntactic category, then a structured Reedy fibrant object $C$ consists exactly of a type in empty context and a type dependent on it:
\begin{equation}\label{eq:srfo}
  \begin{split}
    &\pr (C_0 \ty)\\
    (c_0: C_0) &\pr (C_1(c_0) \ty).
  \end{split}
\end{equation}
Similarly, a structured Reedy fibration $A\fib C$ consists of two more types, the first dependent only on $C_0$ and the second dependent on all three preceding ones:
\begin{equation}\label{eq:srfib}
  \begin{split}
    (c_0: C_0) &\pr (A_0(c_0) \ty)\\
    (c_0: C_0),\, (c_1: C_1(c_0)) ,\, (a_0: A_0(c_0)) &\pr (A_1(c_0,c_1,a_0) \ty).
  \end{split}
\end{equation}

Now we need to specify the pullback of a structured fibration $B\fib A$ along a map $f:C\to A$.
The cloven structure of \sC gives a specified pullback
\begin{equation}
  \vcenter{\xymatrix{
      f_0^*(B_0)\ar[r]\ar[d] &
      B_0\ar[d]\\
      C_0\ar[r]_{f_0} &
      A_0
      }}
\end{equation}
which we take to define $(f^*B)_0$ and the structured fibration on $(f^*B)_0\fib C_0$.
Similarly, we have a specified pullback
\begin{equation}
  \vcenter{\xymatrix{
      (f^*B)_1\ar[r]\ar[d] &
      B_1 \ar[d]\\
      C_1 \times_{C_0} (f^*B)_0\ar[r] &
      A_1\times_{A_0} B_0
      }}
\end{equation}
defining $(f^*B)_1$ and the structured fibration $(f^*B)_1\fib C_1 \times_{C_0} (f^*B)_0$, where $C_1 \times_{C_0} (f^*B)_0$ is of course the specified pullback.
It is easy to see that if \sC is split, then these pullbacks in $(\sC^\bbtwo)\f$ satisfy~\ref{item:split1} in the definition of splitness (\ref{def:split}).
In the internal type theory of \sC, $(f^*B)$ is given by
\begin{align*}
  (c_0: C_0) &\pr (B_0(f_0(c_0)) \ty)\\
  (c_0: C_0),\, (c_1: C_1(c_0)) ,\, (b_0: B_0(f_0(c_0))) &\pr (B_1(f_0(c_0),f_1(c_0,c_1),b_0) \ty).
\end{align*}
Note that this makes sense because $f_1:\prod_{c_0:C_0} (C_1(c_0) \to A_1(f_0(c_0)))$.

We take the specified unit fibration $\mathsf{u}\to 1$ in $(\sC^\bbtwo)\f$ to be
\begin{equation}
  \vcenter{\xymatrix{
      \mathsf{u}^*\mathsf{u} \ar@{->>}[dr] \ar@(r,ul)[drr] \ar@(d,ul)[ddr] \\
      &\mathsf{u}\ar@{=}[r]\ar[d] &
      \mathsf{u}\ar@{->>}[d]\\
      &1\ar@{=}[r] &
      1
      }}
\end{equation}
with the specified fibration structure on $\mathsf{u}^*\mathsf{u} \to \mathsf{u}$ arising by pullback.
In the type theory, this means that $\mathsf{unit}$ in $(\sC^\bbtwo)\f$ is
\begin{align*}
  &\pr (\mathsf{unit} \ty)\\
  (x:\mathsf{unit}) &\pr (\mathsf{unit} \ty).
\end{align*}

For dependent sums and products, we need to start with a composable pair of structure Reedy fibrations $B\fib A\fib C$.
In addition to~\eqref{eq:srfo} and~\eqref{eq:srfib}, this consists of:
\begin{align*}
  (c_0: C_0),\, (a_0: A_0(c_0))  &\pr (B_0(c_0,a_0) \ty)\\
  c_0,c_1,a_0, (a_1: A_1(c_0,c_1,a_0)) ,\, (b_0: B_0(c_0,a_0))
  &\pr (B_1(c_0,c_1,a_0,a_1,b_0) \ty)
\end{align*}
(omitting the types of $c_0,c_1,a_0$ for brevity).
Then the composite structured Reedy fibration $A\fib C$ should be represented by the dependent types
\begin{equation}
  (c_0:C_0) \pr \left(\ssum{a_0:A_0(c_0)} B_0(c_0,a_0)  \ty\right)\label{eq:rds0}
\end{equation}
and
\begin{multline}
  (c_0:C_0),\, (c_1: C_1(c_0)),\, \left(p_0:\ssum{a_0:A_0(c_0)} B_0(c_0,a_0)\right) \\
  \pr
  \left(\ssum{a_1:A_1(c_0,c_1,\fst(p_0))} B_1(c_0,c_1,\fst(p_0),a_1,\snd(p_0))\ty\right).\label{eq:rds1}
\end{multline}
We leave it to the reader to express this diagrammatically in terms of \ctf.
Similarly, the dependent product $\Pi_f B \fib C$ constructed as in \autoref{eq:small-exp} is represented by the dependent types
\begin{equation}\label{eq:rdp0}
  (c_0: C_0) \pr
  \left(\sprod{a_0: A_0(c_0)} B_0(c_0,a_0) \ty\right)
\end{equation}
and
\begin{multline}\label{eq:rdp1}
  (c_0: C_0) ,\, (c_1: C_1(c_0)) ,\, \left(f_0: \sprod{a_0: A_0(c_0)} B_0(c_0,a_0)\right) \\ \pr
  \left(\sprod{a_0: A_0(c_0)} \sprod{a_1: A_0(c_0,c_1,a_0)} B_1(c_0,c_1,a_0,a_1,f_0(a_0)) \ty\right).
\end{multline}
When \sC is split, it is easy to verify the strict preservation of these structures by pullback (\ref{item:split2} and~\ref{item:split3} in \autoref{def:split}).

For path objects, we need to do a little more work, since \autoref{thm:reedy-ttfc} used the homotopy-theoretic axioms (\autoref{def:ttfc}\ref{item:cat7} and~\ref{item:cat8}) rather than the type-theoretic ones (\autoref{thm:path-fact}\ref{item:cat7p} and~\ref{item:cat8p}).
Suppose $A\fib B$ is a structured Reedy fibration.
We begin by defining $(P_B A)_0 = P_{B_0} A_0 \fib A_0\times_{B_0} A_0$ to be the specified path object in \sC associated to the structured fibration $A_0 \fib B_0$.
Now let $A_{01} = A_0 \times_{B_0} B_1$, so that we have a structured fibration $A_1 \fib A_{01}$, and hence a specified path object $P_{A_{01}} A_1 \fib A_1\times_{A_{01}} A_1$.
We will obtain the structured fibration
\[(P_B A)_1 \fib (A_1\times_{B_1} A_1) \times_{(A_0\times_{B_0} A_0)} P_{B_0} A_0\]
as the specified pullback of $P_{A_{01}} A_1 \fib A_1\times_{A_{01}} A_1$ along a map 
\[ (A_1\times_{B_1} A_1) \times_{(A_0\times_{B_0} A_0)} P_{B_0} A_0 \too A_1\times_{A_{01}} A_1 .\]
Such a map is, of course, determined by two maps
\begin{equation}
  (A_1\times_{B_1} A_1) \times_{(A_0\times_{B_0} A_0)} P_{B_0} A_0 \toto A_1\label{eq:pathmaps}
\end{equation}
which agree in $A_{01}$.
We take one of these maps to be simply the projection onto the second factor $A_1$ appearing in the domain.
We cannot take the other to be projection onto the first factor, however, since these two projections do not agree in $A_{01}$.
Instead, we consider the following square:
\begin{equation}
  \vcenter{\xymatrix{
      (A_1 \times_{B_1} A_1) \times_{(A_0\times_{B_0} A_0)} A_0 \ar[r]^-\cong \ar[d] &
      A_1\times_{A_{01}} A_1 \ar[rr]^-{\pi_1} &&
      A_1\ar[d]\\
      (A_1\times_{B_1} A_1) \times_{(A_0\times_{B_0} A_0)} P_{B_0} A_0 \ar[r] &
      A_1\times_{B_1} A_1 \ar[r]_-{\pi_2} &
      A_1 \ar[r] &
      A_0.
    }}\label{eq:sqtoliftin}
\end{equation}
Here $\pi_1$ and $\pi_2$ denote the projections onto the first and second factors of $A_1\times_{A_0} A_1$, respectively.
The reader will easily verify that this square nevertheless commutes.
Since the right-hand map is a structured fibration, and the left-hand map is the specified pullback of $A_0 \to P_{B_0} A_0$ along a structured fibration, 
there is a specified lift, which we take as the second map in~\eqref{eq:pathmaps}.

This completes the definition of a structured Reedy fibration $P A \fib A\times_B A$.
Now we need the diagonal to factor through it by an acyclic cofibration.
Consider first the following diagram
\begin{equation}
  \vcenter{\xymatrix{
      P_{A_0} A_1 \ar[r]\ar[d] &
      (P_{B_1} A)_1\ar[r]\ar[d] &
      P_{A_0} A_1\ar[d]\\
      A_1\times_{A_{01}} A_1\ar[r]\ar[d] &
      (A_1\times_{B_1} A_1)\times_{(A_0\times_{B_0} A_0)} P_{B_0} A_0\ar[r]\ar[d] &
      A_1\times_{A_{01}} A_1\\
      A_0 \ar[r] &
      P_{B_0} A_0.
    }}\label{eq:pathobj-ac}
\end{equation}
The upper-right square is a pullback by definition, and the lower-left square is a pullback by inspection.
The composite across the middle is the identity morphism of $A_1\times_{A_{01}} A_1$, and thus the outer top rectangle is also a pullback.
Hence, by the pasting law for pullback squares, the upper-left square is also a pullback.
However, all the vertical maps are fibrations, 
and the lower map $A_0 \to P_{B_0} A_0$ is an acyclic cofibration; hence its pullback $P_{A_0} A_1 \to (P_{B_0} A)_1$ is also.
Composing this with the defining acyclic cofibration $A_1 \to P_{A_0} A_1$ gives our desired factorization.

In terms of the type theory of \sC, this path object is represented by
\begin{equation}
(b_0:B_0),\, (a_0: A_0(b_0)) ,\, (a_0' : A_0(b_0)) \pr (a_0 \id  a_0')\ty\label{eq:rit0}
\end{equation}
and
\begin{multline}\label{eq:rit1}
  (b_0:B_0),\, (b_1:B_1(b_0)) ,\, (a_0: A_0(b_0)) ,\, (a_0' : A_0(b_0)),\, (p: a_0\id a_0')\\
  (a_1 : A_1(b_0,b_1,a_0)),\, (a_1' : A_1(b_0,b_1,a_0')) \pr (p_* a_1 \id  a_1')\ty
\end{multline}
where, as in \S\ref{sec:homotopy-type-theory}, $p_*$ denotes transport in the fibration $A_1\to A_0$ along the path $p$.
The reflexivity path constructor $A \to P_B A$ is represented by the terms
\begin{equation}\label{eq:rr0}
  (b_0:B_0),\, (a_0: A_0(b_0)) \pr (\r_{a_0} : (a_0 \id  a_0))
\end{equation}
and
\begin{equation}\label{eq:rr1}
  (b_0:B_0),\, (b_1:B_1(b_0)) ,\, (a_0: A_0(b_0)),\, (a_1 : A_1(b_0,b_1,a_0))
  \pr (\r_{a_1} : (\r_{a_0})_* a_1 \id  a_1).
\end{equation}
Finally, we can also write down an explicit term for the eliminator of identity types in $\sC^\bbtwo$ in terms of that in \sC.
Categorically, this means that we suppose a structured Reedy fibration $C\fib P_B A$, with a commutative square
\[\vcenter{\xymatrix{
    A \ar[d] \ar[r]^d &
    C\ar@{->>}[d]\\
    P_B A\ar@{=}[r] &
    P_B A
  }}\]
and construct a lift.
In the type theory of $\sC$, these data consist of dependent types
\begin{multline}\label{eq:rplc0}
  (b_0:B_0),\, (a_0: A_0(b_0)) ,\, (a_0' : A_0(b_0)) ,\, (p_0:(a_0 \id  a_0'))
  \pr C_0(b_0,a_0,a_0',p_0) \ty
\end{multline}
and
\begin{multline}\label{eq:rplc1}
  (b_0:B_0),\, (b_1:B_1(b_0)) ,\, (a_0: A_0(b_0)) ,\, (a_0' : A_0(b_0)),\, (p_0: a_0\id a_0'), \\
  (a_1 : A_1(b_0,b_1,a_0)),\, (a_1' : A_1(b_0,b_1,a_0')),\, (p_1 : (p_0)_* a_1 \id  a_1'),\\
   (c_0 : C_0(b_0,a_0,a_0',p_0)) \pr  C_1 (b_0,b_1,a_0,a_0',p_0,a_1,a_1',p_1,c_0) \ty
\end{multline}
together with terms
\begin{equation}
  (b_0:B_0),\, (a_0:A_0(b_0)) \pr (d_0(b_0,a_0) : C_0(b_0,a_0,a_0,\r_{a_0}))\label{eq:rpld0}
\end{equation}
and
\begin{multline}\label{eq:rpld1}
  (b_0:B_0),\, (b_1:B_1(b_0)) ,\, (a_0: A_0(b_0)),\, (a_1 : A_1(b_0,b_1,a_0))\\
  \pr  (d_1(b_0,b_1,a_0,a_1) : C_1 (b_0,b_1,a_0,a_0,\r_{a_0},a_1,a_1,\r_{a_1},d_0(b_0,a_0))).
\end{multline}
The desired lift can then be given by the terms
\begin{multline}\label{eq:rpl0}
  (b_0:B_0),\, (a_0: A_0(b_0)) ,\, (a_0' : A_0(b_0)) ,\, (p_0:a_0 \id  a_0') \\
  \pr ( J_{d_0}(a_0,a_0',p_0) :C_0(b_0,a_0,a_0',p_0))
\end{multline}
and
\begin{multline}\label{eq:rpl1}
  (b_0:B_0),\, (b_1:B_1(b_0)) ,\, (a_0: A_0(b_0)) ,\, (a_0' : A_0(b_0)),\, (p_0: a_0\id a_0'), \\
  (a_1 : A_1(b_0,b_1,a_0)),\, (a_1' : A_1(b_0,b_1,a_0')),\, (p_1 : (p_0)_* a_1 \id  a_1')\\
  \pr  ( J_{J_{d_1}(a_1,a_1',p_1)}(a_0,a_0',p_0)  :C_1 (b_0,b_1,a_0,a_0',p_0,a_1,a_1',p_1,J_{d_0}(a_0,a_0',p_0))).
\end{multline}
It is straightforward to check that when \sC is split, all of these data are also preserved by pullback in \ctf.

\begin{thm}\label{thm:reedy-split}
  If \sC is a cloven or split type-theoretic fibration category, then so is $(\sC^\bbtwo)\f$.
  Moreover, the ``codomain'' functor $(\sC^\bbtwo)\f \to \sC$ is a strict functor.\qed
\end{thm}
\begin{proof}
  The preceding constructions and observations imply the first statement.
  The second follows by inspection.
\end{proof}

We also observe:

\begin{lem}\label{thm:reedy-shnno}
  If \sC has a (cloven or split) \shnno, so does $(\sC^\bbtwo)\f$, and it is preserved strictly by $\mathrm{cod}:(\sC^\bbtwo)\f \to \sC$.
\end{lem}
\begin{proof}
  The identity map $N\to N$ is a fibration, hence a Reedy fibrant object of $\sC^\bbtwo$, and the morphisms $o$ and $s$ in \sC induce corresponding morphisms in $\sC^\bbtwo$.
  Then up to isomorphism, a Reedy fibration over $(N\fib N)$ consists of a pair of fibrations $B_1\xfib{p} B_0\xfib{q} N$ in \sC.
  The morphisms $s'$ and $o'$ in \ctf consist of diagrams
  \begin{equation}
    \vcenter{\xymatrix{
        1\ar[r]^-{o'_0}\ar@{=}[d] &
        B_1\ar@{->>}[d]\\
        1\ar[r]^-{o'_1}\ar@{=}[d] &
        B_0\ar@{->>}[d]\\
        1\ar[r]^-o &
        N
      }}
    \qquad\text{and}\qquad
    \vcenter{\xymatrix{
        B_1\ar[r]^-{s'_1}\ar@{->>}[d] &
        B_1\ar@{->>}[d]\\
        B_0\ar[r]^-{s'_0}\ar@{->>}[d] &
        B_0\ar@{->>}[d]\\
        N\ar[r]^-s &
        N
      }}
  \end{equation}
  in \sC.
  We define $f_0:N\to B_0$ using the universal property of $N$ in \sC, applied to $o'_0$ and $s'_0$.
  Now pulling back the fibration $B_1 \fib B_0$ to $N$ along $f_0$, we obtain a fibration $(f_0)^*B_1 \fib N$.
  Then using the fact that $f_0 o = o'_0$ and $f_0 s = s'_0 f_0$, we have induced maps $o''_1:1\to (f_0)^*B_1$ with $q o''_1 = o$, and $s''_1:(f_0)^*B_1 \to (f_0)^*B_1$ with $q s''_1 = s q$.
  Thus, we can use again the universal property of $N$ in \sC to obtain a map $N \to (f_0)^*B_1$, and hence $f_1:N\to B_1$, with the desired properties.

  In the cloven case, instead of the identity map we take $N^*\mathsf{u} \fib N$, where $N^*\mathsf{u}$ is the specified pullback of the structured fibration $\mathsf{u}\fib 1$ along $N\to 1$.
  Of course, in the type theory of \sC this consists of
  \begin{align*}
    &\pr (\mathsf{nat} \ty)\\
    (x:\mathsf{nat}) &\pr (\mathsf{unit}\ty).
  \end{align*}
  Since $N^*\mathsf{u} \fib N$ is isomorphic to the identity map of $N$, we can derive the universal property in the same way as above, and by construction it will be preserved strictly by the codomain functor to \sC.
\end{proof}

We end this section with a few further facts about the type theory of $(\sC^\bbtwo)\f$, based on the following observation.
The corresponding fact is easy and standard when \sC is a model category and $\sC^\bbtwo$ has a whole Reedy model structure.

\begin{prop}\label{thm:reedy-afib}
  For a Reedy fibration $f: B\fib A$ between Reedy fibrant objects, the following are equivalent.
  \begin{enumerate}
  \item $f$ is an acyclic fibration in $(\sC^\bbtwo)\f$.\label{item:ra1}
  \item The fibrations $B_0 \fib A_0$ and $B_1 \fib A_1$ in \sC are acyclic fibrations.\label{item:ra2}
  \item The fibrations $B_0 \fib A_0$ and $B_1 \fib B_0\times_{A_0} A_1$ in \sC are acyclic fibrations.\label{item:ra3}
  \end{enumerate}
\end{prop}
\begin{proof}
  Since acyclic fibrations are stable under pullback in \sC, if $B_0 \fib A_0$ is acyclic then so is $B_0\times_{A_0} A_1\fib A_1$.
  Thus, we have~\ref{item:ra2}$\Leftrightarrow$\ref{item:ra3} by the 2-out-of-3 property.
  And since fibrations and acyclic cofibrations in $\sC^\bbtwo$ are in particular levelwise, so are homotopies and homotopy equivalences; this gives~\ref{item:ra1}$\Rightarrow$\ref{item:ra2}.

  Conversely, suppose $f:B\fib A$ is a Reedy fibration of Reedy fibrant objects satisfying~\ref{item:ra3}, write $B_{01}= B_0\times_{A_0} A_1$, and label various morphisms as shown below.
  \[\vcenter{\xymatrix@C=3pc{
      B_1 \ar@{->>}[dr]|{\fhat} \ar@(d,ul)@{->>}[ddr]_{f_1} \ar@(r,ul)@{->>}[rrd]^{\beta} \\
      &B_{01} \ar@{->>}[r]_-{t}\ar@{->>}[d]^{s} \pullbackcorner &
      B_0 \ar@{->>}[d]^{f_0}\\
      &A_1\ar@{->>}[r]_{\alpha} &
      A_0.
    }}\]
  By the Acyclic Fibration Lemma (\ref{thm:afib}) applied to $f_0$, we have a morphism $g_0: A_0 \to B_0$ with $f_0 g_0 = 1_{A_0}$, and a homotopy $G : g_0 f_0 \sim 1_{B_0}$ using some path object $P_{A_0} B_0$ for $f_0$ in $(\sC/A_0)\f$.
  Since pullback of fibrations preserves fibrations and acyclic cofibrations, $\al^* (P_{A_0} B_0)$ is a path object for $s$ in $(\sC/A_1)\f$.
  Now the universal property of pullback induces a map $h: A_1 \to B_{01}$ such that $t h = g_0 \alpha$ and $s h = 1_{A_1}$.
  Similarly, we have a homotopy $H : h s \sim 1$ using the path object $\al^* (P_{A_0} B_0)$ for $s$ in $(\sC/A_1)\f$, with the property that $v H = G t$, where $v$ is the pullback morphism:
  \[\vcenter{\xymatrix{
      \al^* (P_{A_0} B_0)\ar@{->>}[r]^-v\ar@{->>}[d] \pullbackcorner &
      P_{A_0} B_0\ar@{->>}[d]\\
      A_1 \ar@{->>}[r]_\al &
      A_0.
    }}\]

  Since \fhat is an acyclic fibration in $(\sC/A_1)\f$ by \autoref{thm:fibhe}, from Lemmas~\ref{thm:fibhtpy} and~\ref{thm:afib} we have a map $k: B_{01} \to B_1$ in $(\sC/A_1)\f$ such that $\fhat k = 1$, and a homotopy $K : k \fhat \sim 1_{B_1}$ in $(\sC/A_1)\f$ such that $\map_{\fhat} K = r \fhat$.
  Moreover, by the proof of \autoref{thm:afib}, we may suppose that $K$ is defined using a path object $P_{A_1} B_1$ for $B_1$ in $(\sC/A_1)\f$ such that
  \[ P_{A_1} B_1 \fib \al^* (P_{A_0} B_0) \times_{(B_{01} \times_{A_1} B_{01})} (B_1 \times_{A_1} B_1) \]
  is a fibration.
  In particular, $\map_{\fhat}$ is a fibration $P_{A_1} B_1 \fib \al^* (P_{A_0} B_0)$.

  Now the composite $k h$ satisfies
  \[ f_1 k h = s \fhat k h = s h = 1_{A_1} \]
  and also
  \[\beta k h = t \fhat k h = t h = g_0 \alpha.\]
  Thus, if we define $g_1 = k h$, then we have a morphism $g: A \to B$ in $\sC^\bbtwo$.
  Moreover, we have a concatenated homotopy
  \[ c(\map_k H\fhat, K) \;:\; g_1 f_1 = k h s \fhat \sim k \fhat \sim 1_{B_1}\]
  using the path object $P_{A_1} B_1$, such that
  \[ \map_{\fhat} (c(\map_k H \fhat, K)) \sim c(\map_{\fhat} \map_k H \fhat, \map_{\fhat} K) \sim c(H\fhat ,r\fhat) = H\fhat. \]
  Since, as noted above, $\map_{\fhat}$ is a fibration, we may transport $c(\map_k H\fhat, K)$ along this homotopy to obtain a homotopy $L : k h f_1 \sim 1_{B_1}$ in $(\sC/A_1)\f$ using the path object $P_{A_1} B_1$ with the property that $\map_{\fhat} L = H \fhat$, and hence
  \[v \, \map_{\fhat}\, L = v H \fhat = G t \fhat= G \beta.\]

  Define $(P_A B)_0 = P_{A_0} B_0$ and $(P_A B)_1 = P_{A_1} B_1$.
  Then the fibration \[
  v \, \map_{\fhat} : (P_A B)_1 \fib (P_A B)_0\]
  makes $P_A B$ into a Reedy fibrant object of $\sC^\bbtwo$.
  Now in the following diagram:
  \[\hspace{3cm}\vcenter{\xymatrix{
      \mathllap{\al^* (P_{A_0} B_0) \times_{(B_{01} \times_{A_1} B_{01})}} (B_1 \times_{A_1} B_1) \ar[r]\ar[d] &
      \al^*(P_{A_0} B_0)\ar[r]^v\ar[d] &
      P_{A_0} B_0\ar[d]\\
      B_1 \times_{A_1} B_1 \ar[r] &
      B_{01} \times_{A_1} B_{01}\ar[r]\ar[d] &
      B_0 \times_{A_0} B_0\ar[d]\\
      &A_1\ar[r]_\al &
      A_0
    }}\]
  the bottom square and right-hand rectangle are pullbacks, hence so is the upper-right square.
  Since the upper-left square is a pullback by definition, so is the upper rectangle.
  Thus, we have a Reedy fibration $P_A B \fib B \times_A B$.
  We clearly also have a Reedy acyclic cofibration $B \acof P_A B$ factoring the diagonal, so $P_A B$ is a path object for $B$ in $(\sC^\bbtwo/A)\f$.
  Finally, the homotopies $G$ and $L$ define a map $B \to P_A B$ which defines a homotopy in $(\sC^\bbtwo/A)\f$ from $g f$ to $1_B$.
  Therefore, $f$ is a homotopy equivalence in $(\sC^\bbtwo/A)\f$, and hence an acyclic fibration in $(\sC^\bbtwo)\f$.
\end{proof}

\begin{cor}\label{thm:pwhe}
  The homotopy equivalences in $(\sC^\bbtwo)\f$ are the levelwise homotopy equivalences in \sC.
\end{cor}
\begin{proof}
  Since fibrations and acyclic cofibrations in $(\sC^\bbtwo)\f$ are in particular levelwise, so are homotopy equivalences.
  Conversely, if $f$ is a levelwise homotopy equivalence, factor it as $f = p i$ for a Reedy fibration $p$ and a Reedy acyclic cofibration $i$.
  Since $(\sC^\bbtwo)\f$ is a type-theoretic fibration category, $i$ is a homotopy equivalence therein.
  And by the 2-out-of-3 property, $p$ is a levelwise homotopy equivalence.
  Thus $p$ satisfies \autoref{thm:reedy-afib}\ref{item:ra2}, hence is an acyclic fibration in $(\sC^\bbtwo)\f$, thus also a homotopy equivalence.
  Hence $f$ is also a homotopy equivalence in $(\sC^\bbtwo)\f$.
\end{proof}

\begin{cor}\label{thm:funext}
  If \sC satisfies function extensionality, then so does $(\sC^\bbtwo)\f$.
\end{cor}
\begin{proof}
  Let $f: B\fib A$ be a Reedy acyclic fibration in $(\sC^\bbtwo)\f$ and $g: A\fib C$ a Reedy fibration, and refer again to the construction of $\Pi_g(f)$ in \autoref{eq:small-exp} on page~\pageref{eq:small-exp}.
  Using \autoref{thm:reedy-afib}\ref{item:ra3} applied to $f$, the pullback-stability of acyclic fibrations, and the assumption on \sC, we see that $\Pi_{f_0} B_0 \fib C_0$ and $\Pi_{\ftil}(Q) \fib C_1\times_{C_0}\Pi_{f_0}B_0$ are acyclic fibrations.
  By \autoref{thm:reedy-afib}\ref{item:ra3} again, $\Pi_f B \fib C$ is an acyclic fibration in $(\sC^\bbtwo)\f$.
\end{proof}

In \S\ref{sec:scones} it will be important that like all the other structure of \ctf, the function extensionality axiom can be chosen to be \emph{strictly} preserved by $\mathrm{cod}:(\sC^\bbtwo)\f\to \sC$ when \sC is split.
This is the purpose of the following lemma.

\begin{lem}\label{thm:codeqv}
  Suppose $\sC$ satisfies function extensionality, that $f:A\to B$ is an equivalence in \ctf, and that we are given $e_0:1 \to \isequiv_\sC(f_0)$ in \sC.
  Then there exists $e:1\to \isequiv_{\ctf}(f)$ in \ctf whose 0-component is $e_0$.
\end{lem}
\begin{proof}
  Since all the type-theoretic constructions in \ctf restrict to those of \sC on the 0-components, the object $\isequiv_{\ctf}(f)$ is a fibration $\isequiv_1(f) \fib \isequiv_0(f)$, where $\isequiv_0(f) = \isequiv_\sC(f_0)$.
  Since $f$ is an equivalence, we have $g: 1 \to \isequiv_{\ctf}(f)$, consisting of $g_0:\isequiv_\sC(f)$ and $g_1:\isequiv_1(f)$.
  But $\isequiv_\sC(f)$ is an h-proposition, so $g_0 \sim e_0$.
  Thus, by transport (path-lifting), we can modify $g_1$ to a homotopic map $e_1$ which lies over $e_0$, yielding the desired map $e:1\to \isequiv_{\ctf}(f)$.
\end{proof}

Thus, given a term in \sC exhibiting the function extensionality axiom, we can choose such a term in \ctf which is preserved strictly by the codomain functor.

\section{Universes in the Sierpinski \io-topos}
\label{sec:universes}

We now move on to constructing universes in $(\sC^\bbtwo)\f$.
Thus, let $p: \Util\fib U$ be a universe in \sC as in \autoref{def:univ}, defining a notion of \emph{small fibration} in \sC.
We define a fibration $q: \Vtil\fib V$ in $(\sC^\bbtwo)\f$ as follows.
Set $V_0 = U$, $\Vtil_0 = \Util$, and $q_0 = p$.
Let $V_1 = U^{(1)}= (U\times U \to U)^{(\Util\to U)}$, with $V_1 \to V_0$ being the projection $U^{(1)}\to U$; since this is a fibration, $V$ is Reedy fibrant.
Finally, by definition $V_1$ comes with an evaluation map $V_1\times_U \Util \to U\times U$ over $U$, which is to say an arbitrary map $V_1\times_U \Util \to U$; define $\Vtil_1 \fib V_1 \times_{V_0} \Vtil_0$ to be the fibration named by this map.
Then by construction, $q$ is a Reedy fibration.

In the type theory of \sC, $V_0$ is the universe \type, while the fibration $V_1\to V_0$ represents the dependent type
\[ (A: \type) \pr (A\to \type) \ty. \]
The fibration $\Vtil_0\fib V_0$ represents, of course, the universal dependent type
\[(A : \type) \pr A \ty\]
in \sC, while $\Vtil_1 \fib V_1 \times_{V_0} \Vtil_0$ represents the dependent type
\[ (A_0 : \type),\; (A_1: A_0\to \type),\; (a_0: A_0) \pr A_1(a_0)\ty. \]

\begin{defn}
  A map $f: A\to B$ in $\sC^\bbtwo$ is called a \textbf{Reedy small-fibration} if both $f_0$ and the induced map $A_1 \to A_0 \times_{B_0} B_1$ are small fibrations in \sC.
\end{defn}

\begin{prop}\label{thm:reedy-small}
  A map $f: A\to B$ is a Reedy small-fibration if and only if it is small with respect to $V$, i.e.\ it is a pullback of $q$ along some map $B\to V$.
\end{prop}
\begin{proof}
  By construction, $q$ is a Reedy small-fibration, and this property is evidently preserved under pullback.
  Conversely, suppose $f: A\to B$ is a Reedy small-fibration.
  Since $f_0$ is a small fibration, it is named by some map $a_0: B_0 \to U = V_0$.
  Then the composite $B_1 \xto{\beta} B_0 \xto{a_0} U$ names the pullback $A_0 \times_{B_0} B_1$.
  Since $A_1 \to A_0 \times_{B_0} B_1$ is a small fibration, it has a name which supplies a lifting, say $a_1$, of $a_0 \beta$ to $U^{(1)} = V_1$.
  Then $a: B\to V$ is a name for $f$ with respect to $V$.
\end{proof}

\begin{rmk}
  If small fibrations in \sC\ are closed under composition, then a Reedy small-fibration $f: A\to B$ has the property that both $f_0$ and $f_1$ are small fibrations.
  Conversely, if the small fibrations in \sC\ are ``left-cancellable'' (i.e. if $g$ and $f$ are fibrations and $g$ and $g f$ are small, then $f$ is also small), then a Reedy fibration with this property is automatically a Reedy small-fibration.
  Left-cancellability holds whenever smallness is characterized by a downward-closed cardinality condition on the fibers, as is the case for the univalent universe in simplicial sets.
\end{rmk}

\begin{thm}\label{thm:small-univ}
  $V$ is a universe, in the sense of \autoref{def:univ}, for the Reedy small-fibrations in $(\sC^\bbtwo)\f$.
  If $U$ is a cloven or split universe, then so is $V$, and the codomain functor $\ctf \to \sC$ preserves this structure strictly.
\end{thm}
\begin{proof}
  For \autoref{def:univ}\ref{item:u1}, suppose given Reedy small-fibrations $A \xto{f} B \xto{g} C$.
  Then $(g f)_0 = g_0 f_0$ is a composite of small fibrations in \sC, hence small.
  And in \eqref{eq:rfib-compose} we saw that the induced map $A_1 \to A_0 \times_{C_0} C_1$ can be written as a composite of fibrations in \sC, each of which is small if $f$ and $g$ are.
  Hence, $g f$ is a Reedy small-fibration.

  For \autoref{def:univ}\ref{item:u2}, refer to \autoref{eq:small-exp} on page~\pageref{eq:small-exp}.
  If $f$ and $g$ are Reedy small-fibrations, then $f_0$ and $g_0$ are small fibrations, hence so is $\Pi_{f_0} B_0 \fib C_0$.
  Since $f_1$ is small, so is its pullback \ftil, so $\Pi_{\ftil}$ preserves small fibrations.
  However, since $g$ is a Reedy small-fibration, the map $B_1 \fib A_1\times_{A_0} B_0$ is a small fibration, and hence so is its pullback $Q\fib P$; thus the map $\Pi_{\ftil}(Q)\fib C_1 \times_{C_0} \Pi_{f_0}B_0$ is also a small fibration.
  Therefore, $\Pi_f B \fib C$ is a Reedy small-fibration.

  For \autoref{def:univ}\ref{item:u3}, suppose $f: A\to B$ is a map over $C$ between Reedy small-fibrations $A\fib C$ and $B\fib C$.
  We use the construction of Reedy factorizations in \autoref{thm:reedy-fact}.
  Since $A_0\fib C_0$ and $B_0 \fib C_0$ are small, by \autoref{def:univ}\ref{item:u3} for \sC there is a factorization $A_0 \acof D_0 \fib B_0$ whose second factor is small.
  Thus, its pullback $D_0 \times_{B_0} B_1 \fib B_1$ is small, and hence so is the composite $D_0 \times_{B_0} B_1 \fib B_1 \fib C_1$.
  Using \autoref{def:univ}\ref{item:u3} again in \sC, we have a factorization $A_1 \acof D_1 \fib D_0 \times_{B_0} B_1$ whose second factor is small.
  Therefore, $D \fib B$ is a Reedy small-fibration.

  This completes the proof that $V$ is a universe.
  If $U$ is cloven or split, then we can make $V$ cloven or split by using the constructions of type operations in \ctf described in the proof of \autoref{thm:reedy-split}, but interpreting them in terms of the specified operations on $U$ rather than the specified operations on structured fibrations in \sC.
  The codomain functor will preserve this structure strictly, for the same reason that it preserves the cloven structure of \ctf.

For instance, in terms of \sC, the objects making up the Reedy fibration $V^{(1)} \fib V$ are the following.
\begin{itemize}
\item $V_0 = U$ is the universe type \type.
\item $V_1 = U^{(1)}$ is the context $(A_0:\type),\, (A_1 : A_0 \to \type)$.
\item $(V^{(1)})_0 = U^{(1)}$ is the context $(A_0:\type),\, (B_0 : A_0 \to \type)$.
\item $(V^{(1)})_1$ is the context
  \begin{multline*}
    (A_0:\type),\, (A_1 : A_0 \to \type) ,\, (B_0 : A_0 \to \type) ,\\
    \left(B_1 : \sprod{a_0:A_0} (A_1(a_0) \to B_0(a_0) \to \type)\right).
  \end{multline*}
\end{itemize}
Now the expressions~\eqref{eq:rds0} and~\eqref{eq:rds1} in type theory, interpreted using the specified operation $U^{(1)}\to U$ implementing dependent sums in \sC, define a morphism $V^{(1)}\to V$ which implements dependent sums in $\sC^\bbtwo$.
Similarly,~\eqref{eq:rdp0} and~\eqref{eq:rdp1}, interpreted using the specified operation $U^{(1)}\to U$ implementing dependent products in \sC, define a morphism $V^{(1)}\to V$ which implements dependent products in $\sC^\bbtwo$.

Analogously, the objects of the Reedy fibration $\Vtil \times_V \Vtil \fib V$ are
\begin{itemize}
\item $V_0 = U$ is the universe type \type.
\item $V_1 = U^{(1)}$ is the context $(A_0:\type),\, (A_1 : A_0 \to \type)$.
\item $(\Vtil \times_V \Vtil)_0$ is the context
  $(A_0:\type),\, (a_0:A_0),\, (a_0':A_0)$.
\item $(\Vtil \times_V \Vtil)_1$ is the context
  \[ (A_0:\type),\, (a_0:A_0),\, (a_0':A_0),\, (A_1 : A_0 \to \type) ,\, (a_1:A_1(a_0)),\, (a_1':A_1(a_0')). \]
\end{itemize}
The expressions~\eqref{eq:rit0} and~\eqref{eq:rit1} then specify a morphism $\Vtil \times_V \Vtil \to V$ which constructs path types in $\sC^\bbtwo$.
\end{proof}

Next, we consider how universe embeddings lift to lift to $(\sC^\bbtwo)\f$.

\begin{rmk}
  Suppose $i: U\into U'$ is a monomorphism of universes in \sC such that $U$ is $U'$-small, every $U$-small fibration is $U'$-small, and $i$ adds no new names (in the sense of \autoref{thm:no-new-names}); thus $i$ can be made into a universe embedding.
  Let $V$ and $V'$ be the corresponding universes in $\sC^\bbtwo$; then it is easy to see that $V$ is $V'$-small, every Reedy $V$-small fibration is Reedy $V'$-small, and we have a monomorphism $j: V\into V'$ which also adds no new names.
  Hence $j: V\into V'$ can be made into a universe embedding as well.

  In the rest of this section, we show that the same is true for \emph{any} universe embedding in \sC, whether or not it adds new names.
  In particular, this shows that a countably infinite sequence of universe embeddings can also be lifted to $\sC^\bbtwo$.
  It also allows us to avoid modifying the universe structure, so that it will still be strictly preserved by the codomain functor; we will need this latter fact in \S\ref{sec:scones}.
\end{rmk}

\begin{prop}\label{thm:univ}
  If $i: U\into U'$ is an embedding of cloven universes in \sC, then there is an induced embedding $j: V\into V'$ of cloven universes in $\sC^\bbtwo$.
\end{prop}
\begin{proof}
  We define $j_0: V_0 \to (V')_0$ to be $i: U\to U'$, and $j_1: V_1 \to (V')_1$ to be the map $i^{(1)}: U^{(1)} \to (U')^{(1)}$ defined after~\eqref{eq:uembstr}.
  To start with, we need a pullback square
  \[\vcenter{\xymatrix{
      \Vtil\ar[r]\ar[d] &
      \Vtil'\ar[d]\\
      V\ar[r]_j &
      V'
    }}\]
  in $\sC^\bbtwo$, which will be a cube
  \begin{equation}
    \vcenter{\xymatrix@-1pc{
        (\mathrm{ev}_U)^* \Util \ar[rr] \ar[dr] \ar[dd] &&
        (\mathrm{ev}_{U'})^* \Util' \ar'[d][dd] \ar[dr]\\
        &\Util\ar[rr]^(.3){j_1 = i^{(1)}}\ar[dd] &&
        \Util'\ar[dd]\\
        U^{(1)} \ar'[r][rr] \ar[dr] && (U')^{(1)} \ar[dr] \\
        & U\ar[rr]_{j_0 = i} &&
        U'.
      }}\label{eq:vembcube}
  \end{equation}
  Here $\Vtil_1 = (\mathrm{ev}_U)^* \Util$ has the universal property that maps $X\to (\mathrm{ev}_U)^* \Util$ correspond naturally to triples
  \begin{equation}\label{eq:vtil1}
    \Big( X\xto{a} U ,\;
    a^*\Util \xto{b} U ,\;
    X \xto{s} b^*\Util \Big)
  \end{equation}
  where $s$ is a section of $b^*\Util \to a^*\Util \to X$.
  Of course, $(\Vtil')_1 = (\mathrm{ev}_{U'})^* \Util'$ is analogous, and the map $\Vtil_1 \to (\Vtil')_1$ is given by composing the components $a$ and $b$ with $i$.
  
  Now the front face of~\eqref{eq:vembcube} is a pullback since $i$ is a universe embedding in \sC, so it remains to show that the back face is also.
  However, the back vertical maps simply forget the sections $s$, so the back face being a pullback simply says that a map $X\to (\Vtil')_1$ corresponding to a triple
  \[ \Big( X\xto{a} U' ,\;
  a^*\Util' \xto{b} U' ,\;
  X \xto{s} b^*\Util' \Big)
  \]
  factors through $\Vtil_1$ just when $a$ and $b$ factor through $U$.
  This is true because $\Util$ is the pullback $i^* \Util'$.

  Next, we need a pullback square
  \[\vcenter{\xymatrix{
      V\ar[r]\ar[d] &
      \Vtil'\ar[d]\\
      1\ar[r]_v &
      V'
    }}\]
  in $\sC^\bbtwo$, which will be a cube
  \begin{equation}
    \vcenter{\xymatrix@-1pc{
        U^{(1)} \ar[rr] \ar[dr] \ar[dd] &&
        (\mathrm{ev}_{U'})^* \Util' \ar'[d][dd] \ar[dr]\\
        &U \ar[rr]\ar[dd] &&
        \Util'\ar[dd]\\
        1 \ar'[r][rr]^-{v_1} \ar[dr] && (U')^{(1)} \ar[dr] \\
        & 1\ar[rr]_{v_0} &&
        U'
      }}\label{eq:vsmallcube}
  \end{equation}
  in \sC.
  Of course, with $v_0 \eq u$ being the specified name for $U$ in $U'$, the front face of this cube is given.
  We define $v_1: 1 \to (U')^{(1)}$ to name the dependent $U'$-named type $U^{(1)}\to U$, where $U$ is named by $u$ and $U^{(1)}\to U$ is named by $i: U\to U'$.
  It is then easy to see that the back face is also a pullback.

  Now I claim that if we give $V$ and $V'$ their canonical universe structures induced from those of $U$ and $U'$, as above, then $j: V\into V'$ is a universe embedding.
  Consider, for instance, the case of dependent sums; we want the following cube to commute:
  \begin{equation}
    \vcenter{\xymatrix@-.5pc{
        U^{(1\times 1)} \ar[rr]^-{i^{(1\times 1)}} \ar[dr] \ar[dd]_{(\Sigma_V)_1} &&
        (U')^{(1\times 1)} \ar'[d][dd]^-{(\Sigma_{V'})_1} \ar[dr]\\
        &U^{(1)} \ar[rr]^(.3){i^{(1)}}\ar[dd]^(.7){\Sigma} &&
        (U')^{(1)}\ar[dd]^{\Sigma'}\\
        U^{(1)} \ar'[r]^-{i^{(1)}}[rr] \ar[dr] &&
        (U')^{(1)} \ar[dr] \\
        & U\ar[rr]_{i} &&
        U'.
      }}\label{eq:vsumcube}
  \end{equation}
  The front face commutes since $i$ is a universe embedding, so consider the back face.
  A map $X\to U^{(1\times 1)}$ corresponds to a quadruple
  \begin{equation}
    \Big( X\xto{a} U ,\;
    a^*\Util \xto{b} U ,\;
    a^*\Util \xto{c} U ,\;
    b^*\Util \times_{a^*\Util} c^*\Util \xto{d} U \Big).\label{eq:U1x1a}
  \end{equation}
  The map $i^{(1\times 1)}$ acts by composing $a$, $b$, $c$, and $d$ with $i: U\into U'$.
  Since we defined $(\Sigma_V)_1$ with two applications of $\Sigma$ applied to these morphisms, and $i$ commutes with $\Sigma$ and $\Sigma'$, it follows that the back square in~\eqref{eq:vsumcube} commutes as desired.
  The cases of dependent products and identity types are similar.
\end{proof}

Thus, however many internal universes there are in the type theory of \sC, we can find the same number in the type theory of $(\sC^\bbtwo)\f$, which are strictly preserved by $\mathrm{cod}:(\sC^\bbtwo)\f\to \sC$.

\section{Univalence in the Sierpinski \io-topos}
\label{sec:univalence}

We continue with the notations of the last two sections; our goal is now to prove the following theorem.

\begin{thm}\label{thm:univalence}
  Suppose that $U$ is a universe in \sC which satisfies the univalence axiom.
  Then the corresponding universe $V$ in $(\sC^\bbtwo)\f$ also satisfies the univalence axiom.
\end{thm}

As with the function extensionality axiom in \S\ref{sec:sierpinski}, it suffices to prove that the relevant map in \ctf is an equivalence.
By \autoref{thm:codeqv}, we can then choose a term in \ctf representing the univalence axiom which is strictly preserved by the codomain functor.

\begin{proof}
  Let $E\to V\times V$ be the universal space of equivalences in $(\sC^\bbtwo)\f$, corresponding to the dependent type
  \[ (A:\type),\; (B:\type) \pr \equiv(A,B)\ty \]
  defined at the end of \S\ref{sec:homotopy-type-theory}.
  We must show that the section $V\to E$ of the diagonal $V\to V\times V$, which assigns to each type its identity equivalence, is itself an equivalence.
  By \autoref{thm:pwhe}, it suffices to show that it is levelwise an equivalence in \sC.

  To start with, since all the structure at level 0 is exactly as in \sC, the univalence of $U$ directly implies that $V_0 \to E_0$ is an equivalence.
  Thus, it remains to consider $V_1\to E_1$.
  Now since the last step in the construction of \equiv is a dependent sum, we have a pair of Reedy fibrations
  \begin{equation}
    \vcenter{\xymatrix{
        E_1\ar[r]\ar[d] &
        E_0\ar[d]\\
        F_1\ar[r]\ar[d] &
        F_0\ar[d]\\
        V_1\times V_1\ar[r] &
        V_0\times V_0
      }}\label{eq:isequiv-tower}
  \end{equation}
  in which $F\to V\times V$ represents the dependent type
  \[ (A: \type),\, (B: \type) \pr (A \to B) \ty \]
  in the internal type theory of $\ctf$, while $E\to F$ similarly represents
  \begin{equation}
    (A: \type),\, (B: \type),\, (f: A \to B) \pr \isequiv(f) \ty.\label{eq:e1}
  \end{equation}
  By construction, this means that $F_0\to V_0\times V_0$ represents
  \[ (A_0: \type),\, (B_0: \type) \pr (A_0 \to B_0) \ty \]
  in \sC, whereas $F_1 \to (V_1\times V_1)\times_{V_0\times V_0} F_0$ represents
  \begin{multline*}
    (A_0: \type),\, (A_1: A_0\to \type),\, (B_0: \type),\, (B_1: B_0\to \type),\, (f_0: A_0 \to B_0)\\
    \pr \sprod{a_0: A_0} (A_1(a_0) \to B_1(f_0(a_0))) \ty.
  \end{multline*}
  Our goal is to describe $E_1$ similarly in terms of the internal type theory of \sC, so that we can apply univalence there.
  We proceed by evaluating~\eqref{eq:e1} in terms of \sC, considering separately the two factors
  \begin{align}
    (A: \type),\, (B: \type),\, (f: A \to B) &\pr \sum_{s: B\to A} \,\prod_{b: B} (f(s(b)) \id  b) \ty
    \label{eq:hiso1}\\
    (A: \type),\, (B: \type),\, (f: A \to B) &\pr \sum_{r: B\to A}\, \prod_{a: A} (r(f(a)) \id  a) \ty
    \label{eq:hiso2}
  \end{align}
  which are of course closely analogous.

  Firstly, by definition of path-spaces and pullback in $(\sC^\bbtwo)\f$, the dependent type
  \[ (A: \type),\, (B: \type),\, (f: A \to B) ,\, (s: B\to A) ,\,(b: B) \pr (f(s(b)) \id  b)  \ty \]
  is represented by the tower of Reedy fibrations shown in \autoref{fig:sectpath}.
  In this diagram, each morphism is a fibration and each square is a Reedy fibration.
  The ellipses in each context stand for all the variables appearing in contexts below and to the right of it.
  \begin{figure}
    \centering
    \[\vcenter{\xymatrix{
        \Big(\dots, p_1: (p_0)_* \big(f_1(s_0(b_0),s_1(b_0,b_1))\big) \id  b_1 \Big)\ar[r]\ar[d] &
        (\dots, p_0: f_0(s_0(b_0)) \id  b_0)\ar[d]\\
        (\dots, b_1: B_0(b_1))\ar[r]\ar[d] &
        (\dots, b_0: B_0)\ar[d]\\
        \Big(\dots, s_1: \sprod{b_0: B_0} B_1(b_0) \to A_1(s_0(b_0))\Big)\ar[r]\ar[d] &
        (\dots, s_0: B_0\to A_0)\ar[d]\\
        \Big(\dots, f_1: \sprod{a_0: A_0} A_1(a_0) \to B_1(f_0(a_0))\Big)\ar[r]\ar[d] &
        (\dots, f_0: A_0\to B_0)\ar[d]\\
        (\dots, B_1: B_0\to\type) \ar[d]\ar[r] &
        (\dots, B_0:\type)\ar[d]\\
        (\dots, A_1: A_0\to\type) \ar[r] &
        (A_0:\type)
        }}\]
    \caption{Path spaces for the universal section}
    \label{fig:sectpath}
  \end{figure}

  Now, applying dependent product to the top two morphisms, and using the construction from \autoref{thm:reedy-ttfc}, we find that the dependent type
  \[ (A: \type),\, (B: \type),\, (f: A \to B) ,\, (s: B\to A) \pr \prod_{b: B} (f(s(b)) \id  b)  \ty \]
  is represented by the tower in \autoref{fig:secthtpy}.
  (For brevity, we have omitted the types of some variables.)
  \begin{figure}
    \centering
    \[\small\vcenter{\xymatrix@C=1pc{
        \Big(\dots, q_1: \sprod{b_0,b_1} \big((q_0(b_0))_* \big(f_1(s_0(b_0),s_1(b_0,b_1))\big) \id  b_1\big)\Big)\ar[r]\ar[d] &
        (\dots, q_0: \sprod{b_0} (f_0(s_0(b_0)) \id  b_0))\ar[d]\\
        \Big(\dots, s_1: \sprod{b_0} B_1(b_0) \to A_1(s_0(b_0))\Big)\ar[r]\ar[d] &
        (\dots, s_0: B_0\to A_0)\ar[d]\\
        \Big(\dots, f_1: \sprod{a_0} A_1(a_0) \to B_1(f_0(a_0))\Big)\ar[r]\ar[d] &
        (\dots, f_0: A_0\to B_0)\ar[d]\\
        (\dots, B_1: B_0\to\type) \ar[d]\ar[r] &
        (\dots, B_0:\type)\ar[d]\\
        (\dots, A_1: A_0\to\type) \ar[r] &
        (A_0:\type)
        }}\]
    \caption{Section homotopies for the universal section}
    \label{fig:secthtpy}
  \end{figure}
  Therefore,~\eqref{eq:hiso1} is obtained by a dependent sum from the top squares in \autoref{fig:secthtpy}.
  And of course,~\eqref{eq:hiso2} is directly analogous.

  Now, recall that we are interested in the map $V\to E$, and specifically its 1-component $V_1\to E_1$.
  This map factors through the pullback $V_0 \times_{E_0} E_1$.
  Moreover, since $V_0 \times_{E_0} E_1 \to E_1$ is a pullback of the equivalence $V_0\to E_0$ along the fibration $E_1\to E_0$, 
  it is also an equivalence.
  Thus, by 2-out-of-3, $V_1 \to E_1$ is an equivalence if and only if $V_1 \to V_0 \times_{E_0} E_1$ is so.

  In terms of the variables appearing in Figure~\ref{fig:secthtpy}, the map $V_0\to E_0$ acting on $A_0: \type$ is defined by
  \begin{align*}
    B_0 &\eq A_0\\
    f_0 &\eq \idfunc_{A_0}\\
    s_0 &\eq \idfunc_{A_0}\\
    q_0 &\eq \lambda_{b_0: A_0}.\; \r_{b_0}
  \end{align*}
  and similarly for the corresponding data for $r$ as appearing in~\eqref{eq:hiso2}.
  Therefore, upon pullback along this map, the types of the data in $E_1$ become
  \begin{align*}
    A_0 &: \type\\
    B_0, B_1 &: A_0 \to \type\\
    f_1 &: \sprod{a_0} A_1(a_0) \to B_1(a_0)\\
    s_1 &: \sprod{a_0} B_1(a_0) \to A_1(a_0)\\
    q_1 &: \sprod{a_0, a_1} \big(f_1(a_0, s_1(a_0,a_1)) \id  a_1\big)
  \end{align*}
  and similarly for $r$.
  (We have used the fact that transporting along the identity path is the identity.)
  Hence, the fibration $V_0\times_{E_0} E_1 \to V_0\times_{F_0} F_1$ is represented by the dependent type
  \begin{multline}\label{eq:allhiso1}
    A_0,A_1,B_1,f_1\pr
    \sum_{s_1:\sprod{a_0} B_1(a_0) \to A_1(a_0)} \Bigg(\prod_{a_0, a_1} \big(f_1(a_0, s_1(a_0,a_1)) \id  a_1\big)\Bigg) \;\times\;\\
    \sum_{r_1:\sprod{a_0} B_1(a_0) \to A_1(a_0)}\Bigg( \prod_{a_0, a_1} \big(r_1(a_0, f_1(a_0,a_1)) \id  a_1\big)\Bigg)
  \end{multline}
  (all variables have the same types as above).
  However, in the presence of function extensionality, this type is naturally equivalent to
  \begin{multline}\label{eq:allhiso}
    A_0,A_1,B_1,f_1\pr \prod_{a_0} \Bigg(
    \sum _{s_1: B_1(a_0) \to A_1(a_0)} \prod_{a_1} \big(f_1(a_0, s_1(a_1)) \id  a_1\big) \;\times\;\\
    \sum_{r_1: B_1(a_0) \to A_1(a_0)} \prod_{a_1} \big(r_1(f_1(a_0,a_1)) \id  a_1\big)
    \Bigg).
  \end{multline}
  Given $((s_1,p),(r_1,q))$ inhabiting~\eqref{eq:allhiso1} we send it to
  \[ \lambda a_0 .\, \big((\lambda b_1.\, s_1(a_0,b_1),\, \lambda a_1.\, p(a_0,a_1)),\,
  (\lambda b_1.\, r_1(a_0,b_1),\, \lambda a_1.\, q(a_0,a_1))
  \big)
  \]
  inhabiting~\eqref{eq:allhiso}; while given $h$ inhabiting~\eqref{eq:allhiso} we send it to
  \begin{multline*}
    \Big(
    \big(\lambda a_0\, b_1.\, \fst(\fst(h(a_0)))(b_1),\,
    \lambda a_0\,a_1.\, \snd(\fst(h(a_0)))(a_1)
    \big)
    ,\;\\
    \big(\lambda a_0.\, b_1.\, \fst(\snd(h(a_0)))(b_1) ,\,
    \lambda a_0\,a_1.\, \snd(\snd(h(a_0)))(a_1)
    \big)
    \Big)
  \end{multline*}
  inhabiting~\eqref{eq:allhiso1}.
  With our definitional $\eta$-rules for dependent sums and products, these two functions are actually inverse judgmental isomorphisms (although lacking such $\eta$-rules, they would still be inverse equivalences by function extensionality).
  This can be proven purely category-theoretically as well, by showing that~\eqref{eq:allhiso1} and~\eqref{eq:allhiso} represent isomorphic functors and invoking the Yoneda lemma.
  (This sort of equivalence is traditionally called the ``type-theoretic axiom of choice.'')

  However,~\eqref{eq:allhiso} is nothing but
  \[ A_0,A_1,B_1,f_1\pr \prod_{a_0} \isequiv(f_1(a_0)). \]
  Thus, the induced fibration $V_0 \times_{E_0} E_1 \to V_1\times_{V_0} V_1$ is isomorphic to
  \begin{equation}
    A_0, A_1, B_1 \pr \sum_{f_1 : \sprod{a_0} A_1(a_0) \to B_1(a_0)} \prod_{a_0} \isequiv(f_1(a_0)).\label{eq:alleqv1}
  \end{equation}
  But by the same sort of argument, this is isomorphic to
  \begin{equation}
    A_0,A_1,B_1\pr \prod_{a_0} \sum_{f : A_1(a_0) \to B_1(a_0) } \isequiv(f)\label{eq:alleqv}
  \end{equation}
  which of course is nothing but
  \[ A_0,A_1,B_1\pr \prod_{a_0} \equiv(A_1(a_0),B_1(a_0)). \]
  Now we have a commutative square
  \[\vcenter{\xymatrix{
      V_1 \ar[r]\ar[d] &
      V_0 \times_{E_0} E_1\ar[d]\\
      P_{V_0} V_1\ar[r] &
      V_1\times_{V_0} V_1
    }}\]
  in $(\sC/V_0)\f$, in which the left-hand map is an acyclic cofibration and the right-hand map is a fibration.
  Therefore, we have an induced map $P_{V_0} V_1 \to V_0 \times_{E_0} E_1$ of fibrations over $V_1\times_{V_0} V_1$, which it suffices to show to be an equivalence.
  This map is represented by a section of the dependent type
  \[ A_0, A_1, B_1 \pr (A_1 \id  B_1) \to \prod_{a_0} \equiv(A_1(a_0),B_1(a_0))  \ty\]
  obtained from the eliminator for the path type $(A_1\id B_1)$.
  But this map factors, up to homotopy, as a composite
  \[ (A_1 \id  B_1) \xto{\happly} \prod_{a_0} (A_1(a_0) \id  B_1(a_0)) \xto{\Pi(\ptoe)} \prod_{a_0} \equiv(A_1(a_0),B_1(a_0)). \]
  (This follows immediately by an application of $J$ to the identity type $A_1\id B_1$: when applied to reflexivity, both reduce to $\lambda a_0\,a_1.\,a_1$.)
  But $\happly$ is an equivalence by strong function extensionality (\autoref{thm:strong-funext}).
  And $\ptoe$ is an equivalence by univalence in \sC, so by \autoref{thm:funext-forallequiv}, $\Pi(\ptoe)$ is also an equivalence.
  Therefore, our desired map is internally a fiberwise equivalence over $V_1\times_{V_0} V_1$, and hence (by \autoref{thm:fibhe}) also an equivalence on total spaces externally.
  Hence $V$ is univalent.
\end{proof}

This yields our first really new model of the univalence axiom.

\begin{cor}
  The Reedy model category $\sSet^\bbtwo$ supports a model of intensional type theory with dependent sums and products, identity types, and as many univalent universes as there are inaccessible cardinals.
\end{cor}

As before, since the homotopy theory of $\sSet^\bbtwo$ models the ``Sierpinski \io-topos'' $\infty\mathrm{Gpd}^\bbtwo$ , we can say informally that we have a model of type theory in this \io-topos.

\section{Diagrams on inverse categories}
\label{sec:invcat}

As we have observed, what makes \S\S\ref{sec:sierpinski}--\ref{sec:univalence} work is that a Reedy fibrant object $A_1 \fib A_0$ of $\sC^\bbtwo$ can be represented by a context in type theory:
\[ (a_0 : A_0),\; (a_1 : A_1(a_0)). \]
A corresponding fact is true for Reedy fibrant diagrams on some other categories.
For instance, spans of fibrations $A_1 \fib A_0 \twoheadleftarrow A_2$ correspond to contexts of the form
\[ (a_0 : A_0), \; (a_1 : A_1(a_0)),\; (a_2 : A_2(a_0)) \]
whereas cospans $A_0 \ot A_2 \to A_1$ such that $A_2 \fib A_0\times A_1$ is a fibration correspond to contexts of the form
\[ (a_0 : A_0),\; (a_1 : A_1),\; (a_2 : A_2(a_0,a_1)). \]
(This correspondence between diagrams and contexts has also been used elsewhere, e.g. by~\textcite{makkai:folds}.)
In this section we extend \S\S\ref{sec:sierpinski}--\ref{sec:univalence} to such cases.

In this section and the next, we will give up on carrying along cloven and split structure by hand, and simply appeal to a coherence theorem after the construction is complete.
It should be possible to do everything carefully enough to avoid this, but it would be more work and is not necessary for our current applications.

\begin{defn}
  An \textbf{inverse category} is a category such that the relation ``$x$ receives a nonidentity arrow from $y$'' on its objects is well-founded.
\end{defn}

In an inverse category, we write $\prec$ for the above well-founded relation.
As usual for any well-founded relation, we can define the \textbf{ordinal rank} of an object $x\in I$ inductively:
\[ \rho(x) \eq \sup_{y\prec x} (\rho(y)+1). \]
The \textbf{rank of $I$} is by definition $\rho(I) \eq \sup_{x\in I} (\rho(x)+1)$.
Thus, regarding ordinals as categories in the usual way, we have a functor $\rho: I \to \rho(I)\op$ which reflects identities.
The existence of an identity-reflecting functor to the opposite of an ordinal is an alternative definition of an inverse category.

The point of the definition is that we can construct diagrams on $I$ and maps between them by well-founded induction, as follows.
For an object $x\in I$, we write $x\sslash I$ for the full subcategory of the co-slice category $x/I$ which excludes only the identity $\idfunc_x$.
Note that $x\sslash I$ is also an inverse category with $\rho(x\sslash I) = \rho(x)$, and for any nonidentity $\al: x\to y$ we have
\begin{equation}
  \al\sslash (x\sslash I) \cong y\sslash I.\label{eq:reedyslice}
\end{equation}

If $A$ is a diagram in \sC defined on the full subcategory $\setof{y | y\prec x}\subset I$, we can precompose it with the forgetful functor $x\sslash I \to \setof{y | y\prec x}$.
We define the \textbf{matching object} $M_x A$ to be the limit of the resulting diagram:
\[ M_x A \eq \lim_{x\sslash I} A \]
if it exists.
In this case, to give an extension of $A$ to the full subcategory $\setof{y|y\preceq x} \subseteq I$ is precisely to give an object $A_x$ with a map $A_x\to M_x A$.
Similarly, given diagrams $A$ and $B$ defined on the full subcategory $\setof{y|y\preceq x}$, and a natural transformation $f: A|_{\setof{y|y\prec x}}\to B|_{\setof{y|y\prec x}}$ between their restrictions to $\setof{y | y\prec x}$, to give an extension of $f$ to $\setof{y|y\preceq x}$ is precisely to give a map
\[ A_x \to M_x A \times_{M_x B} B_x \]
if the pullback in the codomain exists.
Note that if $x$ has no $\prec$-predecessors, then $x\sslash I$ is empty and $M_x A$ is terminal.

Now suppose that \sC is a type-theoretic fibration category.

\begin{defn}
  A \textbf{Reedy fibration} in $\sC^I$ is a map $f: A\to B$ between $I$-diagrams such that $A$ and $B$ have all matching objects, each pullback $M_x A \times_{M_x B} B_x$ exists, and each map
  \[ A_x \to M_x A \times_{M_x B} B_x \]
  is a fibration in \sC.  A \textbf{Reedy acyclic cofibration} in $\sC^I$ is a levelwise acyclic cofibration.
\end{defn}

In particular, $A$ is Reedy fibrant iff it has all matching objects and each map $A_x \to M_x A$ is a fibration.
Note that if $A$ and $B$ are Reedy fibrant, then the pullback $M_x A \times_{M_x B} B_x$ automatically exists for any $f: A\to B$, as it is a pullback of the fibration $B_x \fib M_x B$.

If $I$ is finite, then Reedy fibrant $I$-diagrams can be regarded as contexts of a certain form in the type theory of \sC.
In the general case, we can regard them as a certain type of ``infinite context''.

Before going further, we need to guarantee that the limits involved in forming matching objects exist and are well-behaved.
For general $I$, this is an additional completeness property of \sC, so we give it a name.

\begin{defn}\label{def:Ilim}
  For $I$ an inverse category, we say that \sC has \textbf{Reedy $I$-limits} if
  \begin{enumerate}
  \item Any Reedy fibrant $A\in\sC^I$ has a limit, which is fibrant in \sC.\label{item:il1}
  \end{enumerate}
  and for Reedy fibrant $A$ and $B$ and any morphism $f: A\to B$, the following hold:
  \begin{enumerate}[resume]
  \item If $f$ is a Reedy fibration, then $\lim f: \lim A \to \lim B$ is a fibration in \sC.\label{item:il2}
  \item If $f$ is a levelwise equivalence, then $\lim f$ is an equivalence in \sC.\label{item:il3}
  \item If $f$ is a Reedy acyclic cofibration, then $\lim f$ is an acyclic cofibration in \sC.\label{item:il4}
  \end{enumerate}
\end{defn}

Unsurprisingly, in the model category case this is automatic.

\begin{lem}
  If \sC is a type-theoretic model category, then it has Reedy $I$-limits for any small inverse category $I$.
\end{lem}
\begin{proof}
  When \sC is a model category, $\sC^I$ has a whole Reedy model structure in which the cofibrations and weak equivalences are levelwise.
  (See, for instance,~\textcite[Ch.~5]{hovey:modelcats}.)
  Thus, $\lim: \sC^I \to \sC$ is a right Quillen functor, hence preserves fibrant objects, fibrations, and weak equivalences between fibrant objects, giving \autoref{def:Ilim}\ref{item:il1}--\ref{item:il3}.
  Finally,~\ref{item:il4} follows since cofibrations in a type-theoretic model category are assumed stable under limits.
\end{proof}

More interesting is that we can construct Reedy $I$-limits in any type-theoretic fibration category inductively.
The basic idea of this is certainly folklore, at least in special cases; the most general statement I know of is in~\textcite{rb:cofibrations}.
Roughly, the construction uses two or three special cases of Reedy $I$-limits to build all of them.

The first special case is finite products.
Of course, any discrete category is inverse, and (since all objects in \sC are fibrant) all diagrams on such a category are Reedy fibrant.
Moreover, the Reedy fibrations are just the levelwise ones.

\begin{lem}
  If $I$ is a finite discrete category, then any type-theoretic fibration category has Reedy $I$-limits.
\end{lem}
\begin{proof}
  Since \sC has a terminal object and pullbacks of fibrations, and all objects are fibrant, it has binary products and hence all finite ones.
  Now a product morphism $\prod_{1\le i\le n} f_i : \prod_{1\le i\le n} A_i \to \prod_{1\le i\le n} B_i$ is a finite composite of morphisms of the form
  \[ \idfunc \times f_j \times \idfunc : \left(\prod_{1\le i< j} A_i\right) \times A_j \times \left(\prod_{j<i\le n} B_i\right)
  \too \left(\prod_{1\le i< j} A_i\right) \times B_j \times \left(\prod_{j<i\le n} B_i \right) .\]
  Each of these is a pullback of $f_j$ along the fibration $\prod_{1\le i< j} A_i \times \prod_{j<i\le n} B_i \fib 1$.
  This preserves fibrations, equivalences, and acyclic cofibrations, and all three classes of maps are preserved by composition.
\end{proof}

If \sC has Reedy $I$-limits for all discrete $I$ with $|I|<\kappa$, we say that \sC has \textbf{Reedy \ka-products}.
Thus, any type-theoretic fibration category has Reedy \om-products.

The second special case is pullbacks of fibrations.
The following lemma is actually not quite a special case of Reedy $I$-limits for inverse $I$, but it is a special case of the corresponding statement for $I$ being a more general ``Reedy category''.


\begin{lem}\label{thm:acof-cogluing}
  Suppose a commutative cube in a type-theoretic fibration category:
  \[\xymatrix@-.8pc{
    A_4 \ar[rr] \ar[dd] \ar[dr]^{u_4} && A_3 \ar[dr]^{u_3} \ar@{->>}'[d][dd]\\
    & B_4 \ar[dd] \ar[rr] && B_3 \ar@{->>}[dd]\\
    A_2 \ar[dr]_{u_2} \ar'[r]^f[rr] && A_1 \ar[dr]^(.3){u_1} \\
    & B_2 \ar[rr]_g && B_1
  } \]
  in which the front and back faces are pullbacks and the maps $B_3\fib B_1$ and $A_3\fib A_1$ are fibrations.
  Then
  \begin{enumerate}
  \item If $u_2$ and the induced map $A_3 \to A_1 \times_{B_1} B_3$ are fibrations, so is $u_4$.\label{item:acg1}
  \item If $u_1$, $u_2$, and $u_3$ are equivalences, so is $u_4$.\label{item:acg2}
  \item If $u_1$, $u_2$, and $u_3$ are acyclic cofibrations, so is $u_4$.\label{item:acg3}
  \end{enumerate}
\end{lem}
\begin{proof}
  Conclusions~\ref{item:acg1} and~\ref{item:acg2} are the ``cogluing lemma'', which is true in any category of fibrant objects; see for instance~\textcite[{}1.4.1]{rb:cofibrations}.
  For~\ref{item:acg3}, since $u_1$ is an acyclic cofibration and $B_3 \fib B_1$ is a fibration, the pullback $u_1^* B_3 \to B_3$ is an acyclic cofibration.
  Therefore, since $u_3$ is also an acyclic cofibration, by \autoref{thm:acof-cancel}, so is the induced map $A_3 \to u_1^* B_3$.
  Now since this is a map between fibrations over $A_1$, by \autoref{def:ttfc}\ref{item:cat8}, its pullback along $f$ is again an acyclic cofibration.
  But $f^* A_3 \cong A_4$ and $g u_2 = u_1 f$ and $g^* B_3 \cong B_4$, so this pullback is isomorphic to the induced map $A_4 \to u_2^* B_4$.

  Now $u_2^*B_4 \to B_4$ is also an acyclic cofibration, being a pullback of $A_2\acof B_2$ along the fibration $B_4\fib B_2$.
  Hence the composite $A_4 \acof u_2^* B_4 \acof B_4$, which is $u_4$, is also an acyclic cofibration.
\end{proof}

The final special case, which is only needed when $I$ is infinite, is towers of fibrations.
If \la is an ordinal, then $\la\op$ is inverse; we say that \sC has \textbf{Reedy limits of \ka-towers} if it has Reedy $\la\op$-limits for all ordinals $\la<\ka$.

For an inverse category $I$, we write $\sigma(I)$ for the \textbf{breadth} of $I$, which is the supremum of the cardinalities of all ``levels'' $I_\la = \setof{x\in I | \rho(x) = \la}$.

\begin{lem}\label{thm:Ilimits}
  If $I$ is an inverse category and \sC is a type-theoretic fibration category which has
  \begin{itemize}
  \item Reedy limits of $\rho(I)$-towers, and
  \item Reedy $\sigma(I)$-products,
  \end{itemize}
  then \sC has Reedy $I$-limits.
  In particular, if $I$ is finite, then any type-theoretic fibration category has Reedy $I$-limits.
\end{lem}
\begin{proof}
  \autoref{def:Ilim}\ref{item:il1}--\ref{item:il3} follow from~\textcite[{}9.3.5]{rb:cofibrations} (and a precise observation of what sizes of products and towers are needed).
  We will summarize the construction, which will make it clear that~\ref{item:il4} also follows.

  We proceed by induction on $\rho(I)$.
  If $\rho(I) = \la+1$, let $J$ denote for the full subcategory of $I$ on objects of rank $<\la$; then $\rho(J)\le\la$.
  If $A$ is a Reedy fibrant $I$-diagram, then because of~\eqref{eq:reedyslice}, its restriction to each $x\sslash I$ is also Reedy fibrant, as is its restriction to $J$.
  (In particular, by the inductive hypothesis, $M_x A$ necessarily exists.)
  We can then construct $\lim^I A$ as the pullback
  \[\vcenter{\xymatrix{
      \lim^I A \ar[r]\ar[d] &
      \prod_{\rho(x)=\la} A_x\ar@{->>}[d]\\
      \lim^J A|_J\ar[r] &
      \prod_{\rho(x)=\la} M_x A.
    }}\]
  Since $A|_J$ is Reedy fibrant, by the inductive hypothesis $\lim^J A|_J$ exists.
  And because \sC has Reedy $\sigma (I)$-products, the products on the right exist and the right-hand map is a fibration; thus the pullback also exists.

  Now if $A\to B$ is a Reedy fibration between Reedy fibrant objects, then as products preserve fibrations (by assumption) and matching objects and $J$-limits take Reedy fibrations to fibrations (by the inductive hypothesis), the resulting cube satisfies the hypotheses of \autoref{thm:acof-cogluing}\ref{item:acg1}, so that $\lim^I A \to \lim^I B$ is a fibration.
  Similarly, if $A\to B$ is a levelwise equivalence or acyclic cofibration between Reedy fibrant objects, the resulting cube satisfies the hypotheses of \autoref{thm:acof-cogluing}\ref{item:acg2} or~\ref{item:acg3}.
  
  Finally, if $\rho(I)$ is a limit ordinal, we can express $\lim^I A$ as a limit over $\rho(I)\op$ of the limits over the full subcategories $\setof{x\in I | \rho(x) \le \la}$.
  By the inductive hypothesis, each of these is a Reedy limit, and so is the $\rho(I)\op$-limit by assumption.
\end{proof}

We now return to constructing a model of type theory in $\sC^I$.
For this, we require only that the limits used for \emph{matching objects} exist and be well-behaved.

\begin{defn}\label{def:adm}
  Suppose \sC is a type-theoretic fibration category.
  An inverse category $I$ is \textbf{admissible} for \sC if \sC has Reedy $(x\sslash I)$-limits for every object $x\in I$.
\end{defn}

\noindent
From the preceding lemmas, therefore, we can conclude:
\begin{itemize}
\item If \sC is a type-theoretic model category, then every small inverse category is admissible for \sC.
\item If each $(x\sslash I)$ is finite, then $I$ is admissible for any type-theoretic fibration category.
\end{itemize}
Note that there are many infinite $I$ for which each $(x\sslash I)$ is finite.
The obvious example is $\om\op$; another is the subcategory of face maps in $\Delta\op$.

Of course, by $(\sC^I)\f$ we mean the full subcategory of $\sC^I$ on the Reedy fibrant objects.

\begin{lem}
  Suppose $I$ is admissible for \sC.
  Then a morphism in $(\sC^I)\f$ is a Reedy acyclic cofibration if and only if it has the left lifting property with respect to Reedy fibrations, and
  every morphism in $(\sC^I)\f$ factors as a Reedy acyclic cofibration followed by a Reedy fibration.
\end{lem}
\begin{proof}
  This is easy and standard.
  Given a commutative square
  \begin{equation}
    \vcenter{\xymatrix@-.5pc{
        A\ar[r]\ar[d] &
        C\ar@{->>}[d]\\
        B\ar[r] &
        D
      }}\label{eq:Irl}
  \end{equation}
  in which $A\acof B$ is a Reedy acyclic cofibration and $C\fib D$ is a Reedy fibration, we inductively define a lift $B\to C$ by lifting in the following square in \sC:
  \[\vcenter{\xymatrix{
      A_x\ar[r]\ar[d] &
      C_x\ar@{->>}[d]\\
      B_x\ar[r] &
      M_x C \times_{M_x D} D_x
    }}\]
  in which the bottom map involves the previously defined components $B_y \to C_y$ for $y\prec x$.
  Thus, Reedy acyclic cofibrations have the left lifting property with respect to Reedy fibrations.
  Similarly, to factor $f: A\to B$ as $A\acof C \fib B$, we inductively factor the induced map
  \[ A_x \too M_x C \times_{M_x B} B_x. \]
  The retract argument then implies the characterization of Reedy acyclic cofibrations.
\end{proof}

Note that these inductive steps are exactly like the ``level $1$'' steps of the proof of \autoref{thm:reedy-fact}, but where we have replaced $(-)_1$ with $(-)_x$, and $(-)_0$ with $M_x (-)$.
Most of the proofs in the remainder of this section will similarly be essentially copies of proofs from \S\S\ref{sec:sierpinski}--\ref{sec:univalence}.
We will henceforth leave such details to the reader, merely remarking on where the admissibility of $I$ is used.

For instance, if $I$ is admissible for \sC and $A\fib B$ is a Reedy fibration, then each $M_x A \fib M_x B$ is a fibration and hence so is each $M_x A \times_{M_x B} B_x \fib B_x$.
Thus, by composition, Reedy fibrations are in particular levelwise fibrations.

\begin{thm}
  If \sC is a type-theoretic fibration category and $I$ is an inverse category that is admissible for \sC, then $(\sC^I)\f$ is also a type-theoretic fibration category.
\end{thm}
\begin{proof}
  This is a copy of \autoref{thm:reedy-ttfc}.
  One important wrinkle is that in \autoref{eq:small-exp}, we have to replace $\Pi_{f_0} B_0$ not by $\Pi_{M_x f} (M_x B)$, but by $M_x (\Pi_f B)$.
\end{proof}

The alternative construction of path objects in $(\sC^\bbtwo)\f$ described before \autoref{thm:reedy-split} also generalizes.
In~\eqref{eq:pathobj-ac}, the very bottom map must be replaced by the induced map $M_x A \to M_x(P_B A)$; we require \autoref{def:Ilim}\ref{item:il4} to ensure that this is again an acyclic cofibration.

Next we generalize \autoref{thm:reedy-afib} and its corollaries~\ref{thm:pwhe} and~\ref{thm:funext}.

\begin{prop}\label{thm:Ireedy-afib}
  Let $I$ be admissible for \sC, and let $f: A\fib B$ be a Reedy fibration in $\sC^I$ between Reedy fibrant objects.
  Then the following are equivalent.
  \begin{enumerate}
  \item $f$ is an acyclic fibration in $\sC^I$.\label{item:iraf1}
  \item Each fibration $A_x \fib B_x$ is an acyclic fibration.\label{item:iraf2}
  \item Each fibration $A_x \fib M_x A \times_{M_x B} B_x$ is an acyclic fibration.\label{item:iraf3}
  \end{enumerate}
\end{prop}
\begin{proof}
  Since matching objects of Reedy fibrant objects preserve levelwise equivalences,~\ref{item:iraf2}$\Leftrightarrow$\ref{item:iraf3} follows from 2-out-of-3 as in \autoref{thm:reedy-afib}, and~\ref{item:iraf1}$\Rightarrow$\ref{item:iraf2} is likewise immediate.
  To prove~\ref{item:iraf3}$\Rightarrow$\ref{item:iraf1}, we construct, by induction on $x\in I$, a section $g$ of $f$ and a path object $P_A B$ for $B$ in $(\sC^I/A)\f$ which supports a homotopy $g f \sim 1_B$.
  Since matching objects preserve fibrations and acyclic cofibrations (by \autoref{def:Ilim}\ref{item:il2} and~\ref{item:il4}), they also preserve path objects and hence homotopies; thus the proof of \autoref{thm:reedy-afib} gives exactly the induction step we need.
\end{proof}

\begin{cor}
  The homotopy equivalences in $(\sC^I)\f$ are the levelwise homotopy equivalences in \sC.\qed
\end{cor}

\begin{cor}
  If \sC satisfies function extensionality, so does $(\sC^I)\f$.\qed
\end{cor}


Now let $\Util \fib U$ be a universe in \sC, defining a notion of \emph{small fibration}.
We define a Reedy fibration $\Vtil\fib V$ in $(\sC^I)\f$ as follows.
For $x\in I$, by induction suppose $\Vtil\fib V$ is defined on $\setof{y | y\prec x}$.
Taking limits, we have a fibration $M_x \Vtil \fib M_x V$.
Define
\[V_x \eq (M_x V \times U \to M_x V)^{(M_x\Vtil \to M_x V)} \]
equipped with the evident fibration $V_x \to M_x V$.
By definition, we have an evaluation map $V_x \times_{M_x V} M_x \Vtil \to M_x V\times U$ over $M_x V$, hence a plain morphism $V_x \times_{M_x V} M_x \Vtil \to U$.
Let $\Vtil_x \fib V_x \times_{M_x V} M_x \Vtil$ be the small fibration named by this map.
Then by construction, $V$ is Reedy fibrant and $\Vtil\fib V$ is a Reedy fibration.

\begin{defn}
  A morphism $f: A\to B$ in $(\sC^I)\f$ is a \textbf{Reedy small-fibration} if each map $A_x \to M_x A \times_{M_x B} B_x$ is a small fibration in \sC.
\end{defn}

\begin{prop}
  $f: A\to B$ is a Reedy small-fibration if and only if it is small with respect to the universe $V$ defined above.
\end{prop}
\begin{proof}
  A copy of \autoref{thm:reedy-small}.
\end{proof}

We now need the following additional assumption.

\begin{defn}
  We say that $I$ is \textbf{admissible} for the universe $\Util\to U$ if it is admissible for \sC, and moreover Reedy $(x\sslash I)$-limits take Reedy small-fibrations to small fibrations in \sC, for any $x\in I$.
\end{defn}

If $U$-small fibrations are defined by a cardinality condition on the fibers, then $I$ is admissible for $U$ as long as this cardinality class is closed under $(x\sslash I)$-limits for each $x$.
This is the case for the univalent universes in groupoids and simplicial sets, if they are defined using an inaccessible \ka such that $|I|<\ka$.

\begin{lem}\label{thm:rfib-lw}
  If $I$ is admissible for $\Util\fib U$, then a Reedy small-fibration is in particular a levelwise small-fibration.
\end{lem}
\begin{proof}
  Let $A\fib B$ be a Reedy small-fibration.
  By assumption, each induced fibration $M_x A \fib M_x B$ is small, hence so is its pullback to $B_x$.
  But $A_x \to B_x$ is the composite
  \[ A_x \fib M_x A \times_{M_x B} B_x \fib B_x \]
  and is therefore also small.
\end{proof}

\begin{thm}
  If $I$ is admissible for a universe $\Util\fib U$, then $\Vtil\fib V$ is a universe, in the sense of \autoref{def:univ}, for the Reedy small-fibrations in $(\sC^I)\f$.
\end{thm}
\begin{proof}
  A copy of \autoref{thm:small-univ}.
  We do frequently have to use \autoref{thm:rfib-lw}.
\end{proof}

\begin{thm}
  If $i: U\into U'$ is a universe embedding in \sC and $I$ is admissible for $U$ and $U'$, then there is an induced universe embedding $j: V\into V'$ in $(\sC^I)\f$.
\end{thm}
\begin{proof}
  A copy of \autoref{thm:univ}.
  Now $\Vtil_x$ has the universal property that maps $X\to \Vtil_x$ correspond naturally to triples 
  \begin{equation}\label{eq:vtilx}
    \Big( X\xto{a} M_x V ,\;
    a^* (M_x\Util) \xto{b} U ,\;
    X \xto{s} b^*\Util \Big)
  \end{equation}
  and the rest of the proof goes through as before.
\end{proof}

Finally, we have:

\begin{thm}
  If $I$ is admissible for a univalent universe $U$ in \sC, then the induced universe $V$ in $(\sC^I)\f$ is also univalent.
\end{thm}
\begin{proof}
  A copy of \autoref{thm:univalence}.
  Of course, the right-hand towers in Figures~\ref{fig:sectpath} and~\ref{fig:secthtpy} are replaced by matching objects, and similarly everywhere else.
  We use \autoref{def:Ilim}\ref{item:il3} to conclude, by induction, that the induced map $M_x V \to M_x E$ is an equivalence.
\end{proof}


This yields a larger class of new models of the univalence axiom.

\begin{cor}
  For any small inverse category $I$, the Reedy model category $\sSet^I$ supports a model of intensional type theory with dependent sums and products, identity types, and with as many univalent universes as there are inaccessible cardinals larger than $|I|$.\qed
\end{cor}

As before, we may say that this model lives in the \io-topos $\infty \mathrm{Gpd}^I$.

\begin{rmk}
  The Reedy model structure on $\sC^I$ exists more generally than when $I$ is an inverse category: we only need $I$ to be a \emph{Reedy category} or some generalization thereof (see e.g.~\textcite{reedy,bm:extn-reedy,cisinski:presheaves}).
  But in general, the Reedy cofibrations are not levelwise (though the weak equivalences are).
  On the other hand, for suitable \sC (including simplicial sets) and \emph{any} $I$, the category $\sC^I$ has an \emph{injective model structure} in which the weak equivalences and cofibrations are levelwise.
  But in general, the injective fibrations seem to admit no simple description.

  It just so happens that when $I$ is inverse, the Reedy and injective model structures coincide.
  This coincidence actually only requires $I$ to be \emph{elegant} in the sense of \textcite{br:reedy}.
  In \textcite{shulman:elreedy} I will show that when $\sC=\sSet$, the results of this paper can be generalized to all elegant $I$ (using different methods).
\end{rmk}

\begin{rmk}
  One application of (pre)sheaf models for type theory is to exhibit the non-provability of various logical statements.
  In homotopy type theory, it is natural to treat the h-propositions as the logical propositions.
  Categorically, this corresponds to using the \io-categorical monomorphisms as the ``predicates'', and the subterminal (a.k.a.\ $(-1)$-truncated) objects as the ``propositions''.

  In particular, the ``propositional logic'' of the \io-topos $\sSet^I$ is the same as that of the 1-topos $\mathrm{Set}^I$, namely the Heyting algebra of cosieves in $I$.
  It is shown in~\textcite{b:fgfheyting} (in other language) that Heyting algebras of cosieves on inverse categories suffice to violate any propositional statement that is not an intuitionistic tautology.
  Therefore, the univalence axiom does not imply any such statement.
  It seems that even this was not previously known.
\end{rmk}

\section{Oplax limits}
\label{sec:oll}

Finally, we will show that the methods of the previous sections extend to \emph{oplax limits} over inverse categories.
This includes gluing constructions, scones, and other types of ``logical relations'', as well as the ``combinatorial realizability'' of \textcite{hw:crmtt}, and thus allows us to derive homotopical canonicity and parametricity results.

First we need to define the functors along which we can glue.

\begin{defn}\label{def:fibfr}
 A functor between type-theoretic fibration categories is a \textbf{strong fibration functor} if it preserves terminal objects, fibrations, acyclic cofibrations, and pullbacks of fibrations.
\end{defn}

Note that these are more general than the functors considered in \autoref{thm:syn-init}, which must preserve all specified structure strictly.

\begin{lem}\label{thm:ffeqv}
  A strong fibration functor preserves equivalences.
\end{lem}
\begin{proof}
  Since it preserves fibrations and acyclic cofibrations, it preserves path objects and therefore preserves homotopies, hence also homotopy equivalences.
\end{proof}

Let \nTTFC denote the category of type-theoretic fibration categories and strong fibration functors.

\begin{defn}
  Suppose $I$ is a category and $\sC :I\op \to\nCAT$ is a functor, written $x\mapsto \sC_x$ and $\alpha\mapsto\alpha^*$.
  The \textbf{oplax limit} of \sC is the following category, which we denote $\oplim I \sC$.
  \begin{itemize}
  \item Its objects consist of:
    \begin{enumerate}
    \item For each $x\in I$, an object $A_x \in \sC_x$; and
    \item For each $\alpha:x\to y$ in $I$, a morphism $A_\alpha:A_x \to \alpha^*(A_y)$ in $\sC_x$; such that
    \item For each $x$ we have $A_{1_x} = 1_{A_x}$; and
    \item For each $\alpha$, $\beta$ we have $\alpha^*(A_\beta) \circ A_\alpha = A_{\beta\alpha}$.
    \end{enumerate}
  \item Its morphisms $f:A\to B$ consist of:
    \begin{enumerate}
    \item For each $x\in I$, a morphism $f_x: A_x \to B_x$ in $\sC_x$; such that
    \item For each $\alpha:x\to y$, we have $B_\alpha \circ f_x = \alpha^*(f_y)\circ A_\alpha$.
    \end{enumerate}
  \end{itemize}
\end{defn}

The oplax limit has the universal property that for any category \sA, functors $\sA \to \oplim I \sC$ are in natural bijection with oplax natural transformations from the constant $I\op$-diagram at $\sA$ to the diagram $\sC$.
It can also be described as the category of sections of the Grothendieck construction of $\sC$.
Note that the oplax limit of the constant functor at a category \sC is just the diagram category $\sC^I$.

Now suppose that $I$ is an inverse category, $x\in I$, and $A\in\oplim{\setof{y|y\prec x}}{\sC}$.
Then we have a diagram in $\sC_x$ defined on $(x\sslash I)$ which takes $\alpha:x\to y$ to $\alpha^*(A_y)$.
We define the \textbf{matching object} $M_x A$ to be the limit of this diagram:
\begin{equation}
  M_x A \eq \lim_{(\alpha:x\to y) \in (x\sslash I)} \alpha^*(A_y)
\end{equation}
if it exists.
Then to give an extension of $A$ to $x$, in the evident sense, is precisely to give an object $A_x\in\sC_x$ with a map $A_x \to M_x A$, and similarly for morphisms as in the constant case.

\begin{defn}
  If $I$ is inverse and $\sC:I\op \to\nTTFC$, then a \textbf{Reedy fibration} is a morphism $f:A\to B$ in $\oplim I \sC$ such that $A$ and $B$ have all matching objects, each pullback $M_x A \times_{M_x B} B_x$ exists, and each map
  \[ A_x \to M_x A \times_{M_x B} B_x \]
  is a fibration in $\sC_x$.
  A \textbf{Reedy acyclic cofibration} is a levelwise acyclic cofibration.
\end{defn}

In particular, $A$ is Reedy fibrant iff it has all matching objects and each map $A_x \to M_x A$ is a fibration in $\sC_x$.
Generalized Reedy homotopy theories similar to this one were considered in \textcite{johnson:modreedy}.

\begin{defn}
  Suppose $I$ is inverse and $\sC:I\op \to\nTTFC$.
  We say \sC is \textbf{admissible} if $\sC_x$ has Reedy $(x\sslash I)$-limits for every $x\in I$, and these are preserved by $\alpha^*$ for any $\alpha:y\to x$.
\end{defn}

Of course, the constant functor at \sC is admissible just when \sC is admissible for $I$ in the sense of \autoref{def:adm}.
The lemmas in \S\ref{sec:invcat} also show:
\begin{itemize}
\item If each $\sC_x$ is a type-theoretic model category and each functor $\sC_x$ is continuous, then \sC is admissible.
\item If each $(x\sslash I)$ is finite, then \sC is admissible.
\end{itemize}
We write ${\oplim I \sC}\f$ for the full subcategory of $\oplim I \sC$ on the Reedy fibrant objects.

\begin{thm}
  Suppose $\sC:I\op\to\nTTFC$ is admissible.  Then:
  \begin{enumerate}
  \item ${\oplim I \sC}\f$ is a type-theoretic fibration category.
  \item The homotopy equivalences in ${\oplim I \sC}\f$ are the levelwise ones.
  \item If each $\sC_x$ satisfies function extensionality, so does ${\oplim I \sC}\f$.
  \end{enumerate}
\end{thm}
\begin{proof}
  Just like the corresponding facts in \S\ref{sec:invcat}, using the fact that by definition and by \autoref{thm:ffeqv}, strong fibration functors preserve all the same structure as Reedy limits.
\end{proof}

Now suppose for each $x\in I$ we have a universe $\Util_x \fib U_x$ in $\sC_x$.
We define a Reedy fibration $\Vtil\fib V$ in ${\oplim I \sC}\f$ just as in \S\ref{sec:invcat}, replacing $\Util\fib U$ by $\Util_x \fib U_x$ at each appropriate place.
In particular, the fibration $V_x \fib M_x V$ is the local exponential
\[(M_x V \times U_x \to M_x V)^{(M_x\Vtil \to M_x V)}. \]

\begin{defn}
  A morphism $f: A\to B$ in ${\oplim I \sC}\f$ is a \textbf{Reedy small-fibration} if each map $A_x \to M_x A \times_{M_x B} B_x$ is a small fibration in $\sC_x$.
\end{defn}

\begin{prop}
  $f: A\to B$ is a Reedy small-fibration if and only if it is small with respect to the universe $V$ defined above.
\end{prop}

\begin{defn}
  We say that $\sC:I\op\to\nTTFC$ is \textbf{admissible} for the chosen universes $\Util_x\to U_x$ if it is admissible, and moreover Reedy $(x\sslash I)$-limits take Reedy small-fibrations to small fibrations in $\sC_x$, for any $x\in I$.
\end{defn}

\begin{thm}
  If $\sC:I\op\to\nTTFC$ is admissible for universes $\Util_x\fib U_x$, then $\Vtil\fib V$ is a universe for the Reedy small-fibrations in ${\oplim I \sC}\f$.
\end{thm}

\begin{thm}
  If $i: U_x\into U'_x$ is a universe embedding in $\sC_x$ for each $x$, and $\sC:I\op\to\nTTFC$ is admissible for both families of universes, then there is an induced universe embedding $j: V\into V'$ in ${\oplim I \sC}\f$.
\end{thm}

Finally, using \autoref{thm:ffeqv} again, we have:

\begin{thm}
  If $\sC:I\op\to\nTTFC$ is admissible for a family of univalent universes $U_x$, then the induced universe $V$ in ${\oplim I \sC}\f$ is also univalent.
\end{thm}

\section{Gluing, scones, and canonicity}
\label{sec:scones}

In this section, we apply a particular case of an oplax limit to prove a \emph{homotopy canonicity} result.
Traditional canonicity for type theory means that every term of natural number type $N$ is equal, judgmentally, to a \emph{numeral} --- one of the form $s^n o$ for some external natural number $n\in\mathbb{N}$.
This fails when we add axioms such as univalence and function extensionality.
Homotopy canonicity refers to a statement that nevertheless any term of type $N$ is \emph{propositionally} equal to a numeral, i.e.\ we have $x \id s^no$ for some $n\in\mathbb{N}$.
Voevodsky has conjectured that type theory with univalence satisfies homotopy canonicity in this sense; here we prove the conjecture for one univalent universe in the presence of an additional 1-truncation axiom.

The particular oplax limit we use is a \textbf{gluing construction}, which is the case when the indexing category is $I=\bbtwo$.
In this case, we have a single strong fibration functor $\Gamma:\sC\to\sD$, and the oplax limit $\oplim \bbtwo \sC$ is the comma category $(\sD\dn\Gm)$.
A Reedy fibrant object of $(\sD\dn\Gm)$ is a fibration $A_1 \fib \Gm(A_0)$.

The most common $\Gamma$ to use is a ``global sections'' functor, which in the simplest case is set-valued.
We regard \nSet as a type-theoretic fibration category where every function is a fibration; thus, the acyclic cofibrations and the equivalences are just the isomorphisms.
The ordinary global sections functor $\sC(1,-):\sC\to\nSet$ preserves limits and fibrations, but not acyclic cofibrations and hence not equivalences.
Hence, for homotopical canonicity, we must use some sort of quotient or homotopical quotient of this functor.

To start with, we define $\Gamma_0:\sC\to\nSet$ to be the quotient of $\sC(1,-)$ by the homotopy relation:
\[ \Gamma_0(A) \eq \sC(1,A)/\sim. \]
Note that this is independent of whatever choices of path objects we make in \sC.
It easily preserves terminal objects, fibrations, acyclic cofibrations, and equivalences, but does not in general preserve pullbacks.
However, it does preserve pullbacks when restricted to objects that are \emph{0-truncated} in the following sense.

\begin{defn}
  The notion of an object $A$ of a type-theoretic fibration category \sC being \textbf{$n$-truncated} is defined by induction as follows:
  \begin{itemize}
  \item $A$ is \textbf{$(-1)$-truncated} if it is an h-proposition, i.e.\ any two morphisms with codomain $A$ are homotopic.
  \item $A$ is \textbf{$(n+1)$-truncated} if $P A$ is an $n$-truncated object of $(\sC/A\times A)\f$.
  \end{itemize}
\end{defn}

If we have function extensionality, it is equivalent to start the induction with the $(-2)$-truncated objects being the contractible ones; but we will not need this.
We can define truncatedness internally to type theory as well:
\begin{align*}
  \istrunc {(-1)} (A) &\eq \isprop(x\id y)\\
  \istrunc{(n+1)}(A) &\eq \prod_{x:A} \; \prod_{y:A} \istrunc n (x\id y).
\end{align*}
By Lemmas~\ref{thm:isprop-isprop} and~\ref{thm:prop-forall} and induction on $n$, the type $\istrunc n(A)$ is always an h-proposition.

\begin{defn}
  A type-theoretic fibration category \sC is \textbf{$n$-truncated} if all objects of \sC are $n$-truncated.
\end{defn}

Thus, for instance, \sC is $(-1)$-truncated if any two morphisms are homotopic, so that up to homotopy \sC is essentially just a partial order.
Similarly, \sC is 0-truncated if any two parallel homotopies are homotopic, so that up to homotopy \sC is essentially just an ordinary category.

\begin{lem}\label{thm:global-sections0}
  If \sC is 0-truncated, then $\Gamma_0:\sC\to\nSet$ is a strong fibration functor.
\end{lem}
\begin{proof}
  Clearly $\Gamma_0$ preserves terminal objects, fibrations, and homotopy equivalences, hence also acyclic cofibrations.
  To show that it preserves pullbacks of fibrations, suppose $p:B\fib A$ is a fibration and $f:C\to A$; we must show that the canonical function
  \begin{equation}
    \Gamma_0(C\times_A B) \to \Gamma_0 C \times_{\Gamma_0 A} \Gamma_0 B\label{eq:g0pb}
  \end{equation}
  is a bijection.

  Firstly, an element of $\Gamma_0 C \times_{\Gamma_0 A} \Gamma_0 B$ is a pair $([c],[b])$ where $c:1\to C$ and $b:1\to B$ and there exists a (non-specified) homotopy $p b \sim f c$.
  By transport (path-lifting), there is then a $b':1\to B$ with $b'\sim b$ and $p b' = f c$.
  Then we have $[(c,b')]\in \Gamma_0(C\times_A B)$, and its image in $\Gamma_0 C \times_{\Gamma_0 A} \Gamma_0 B$ is equal to our original element $([c],[b])$; thus~\eqref{eq:g0pb} is surjective.

  For injectivity, suppose $[(c,b)] \in \Gamma_0(C\times_A B)$ and $[(c',b')] \in \Gamma_0(C\times_A B)$ whose images in $\Gamma_0 C \times_{\Gamma_0 A} \Gamma_0 B$ are equal.
  Thus, we have $h:b\sim b'$ and $k:c\sim c'$.
  As described in \S\ref{sec:hothy-fibcat}, we may choose path objects for $A$, $B$, and $C$ along with a morphism $\ap_f :P C \to P A$ and a fibration $\ap_p : P B \fib P A$, and we may assume that $h$ and $k$ are homotopies with respect to these path objects.
  Then the homotopies $\map_p h$ and $\map_f k$ have equal endpoints, but are not necessarily equal.
  However, since \sC is 0-truncated, they are homotopic, as maps $1\to P A$ over $A\times A$.
  Thus, by 2-dimensional path-lifting, there is a homotopy $h':b\sim b'$ such that $\map_p h' = \map_f k$.

  Now $h'$ and $k$ induce a map $1\to P C \times_{P A} P B$.
  Moreover, by \autoref{thm:acof-cogluing}, the pullback $P C \times_{P A} P B$ is a path object for $C\times_A B$.
  Thus, $(h',k)$ gives a homotopy $(c,b)\sim (c',b')$, so that $[(c,b)] = [(c',b')]$ in $\Gamma_0(C\times_A B)$; hence~\eqref{eq:g0pb} is injective.
\end{proof}

We will refer to the gluing construction $(\nSet\dn\Gm_0)\f$ as the \textbf{Sierpinski 0-cone} or \textbf{0-scone} of \sC: its objects are objects $A\in\sC$ equipped with a $\Gamma_0(A)$-indexed family of sets.
Now recall that the category of sets contains a single univalent universe, namely the subobject classifier $\Omega = \{\top,\bot\}$, whose ``small fibrations'' are the monomorphisms.
Thus we have:

\begin{cor}
  If \sC is 0-truncated and has a univalent universe, then $(\nSet\dn\Gm_0)\f$ has one univalent universe, whose small objects are small objects of \sC equipped with a homotopy-invariant subset of their global sections.\qed
\end{cor}

Unfortunately, 0-truncated univalent universes tend to be quite small.
Specifically, if there are any small types which admit a nontrivial automorphism, then a univalent universe contains nonidentity self-paths and hence is not 0-truncated.
However, we can at least consider a universe all of whose types are h-propositions, since an h-proposition has no nonidentity automorphisms.
In a moment we will prove canonicity for such a type theory, but first we need a lemma about the natural numbers object.

\begin{lem}\label{thm:scone-nno}
  If \sC is 0-truncated and has a \shnno, so does $(\nSet\dn\Gm_0)\f$.
\end{lem}
\begin{proof}
  If $N$ is the \shnno of \sC, consider the function $\mathbb{N} \to \Gamma_0(N)$ sending each external natural number $n\in\mathbb{N}$ to the homotopy class of the composite $s^n o:1\to N$ (if \sC is the syntactic category of a type theory, then this is the numeral $\underline{n}$).
  There are obvious morphisms $o$ and $s$ induced on this object of $(\nSet\dn\Gm_0)\f$.
  A fibration over it consists of a fibration $B \fib N$ in \sC together with, for each $n\in\mathbb{N}$, a family of sets $B(n,b)$ indexed by pairs consisting of a natural number $n\in\mathbb{N}$ and a homotopy class of morphisms $b:1\to B$ lifting $s^n o:1\to N$.

  To give $o'$ in $(\nSet\dn\Gm_0)\f$ means to give $o':1 \to B$ lifting $o$, together with an element $o''\in B(0,o')$.
  Similarly, to give $s'$ in $(\nSet\dn\Gm_0)\f$ means to give $s':B\to B$ lying over $s:N\to N$, together with for each $n\in\mathbb{N}$ and $b:1\to B$ lifting $s^n o:1\to N$, a function $s'':B(n,b) \to B(n+1,s'b)$.

  Now the universal property of $N$ in \sC induces a section $f:N\to B$ as usual.
  Moreover, we can define elements $f_n \in B(n,f s^n o)$ by induction on $n$, taking $f_0 = o''$ and $f_{n+1} = s''(f_n)$.
  Together these give the desired section.
\end{proof}

\begin{thm}\label{thm:canonicity0}
  Consider dependent type theory augmented by:
  \begin{itemize}
  \item an axiom asserting that every type is 0-truncated;
  \item the function extensionality axiom;
  \item a \shnno; and
  \item one universe together with the univalence axiom for it.
  \end{itemize}
  Then every term of type $N$ is homotopic to a numeral.
\end{thm}

Note that we do \emph{not} assert that the \shnno belongs to the universe.
Indeed, this would be inconsistent with 0-truncation as remarked above, since a \shnno always has nonidentity automorphisms.
We will return to this question below.

\begin{proof}
  Let \sC be the syntactic category of our type theory, which is of course canonically split.
  We have observed in \autoref{eg:set-split} that \nSet is canonically split (but we could also just apply a coherence theorem to it).
  We did not check in \S\ref{sec:oll} that the general oplax limit construction preserves splitness, but it is easy to see that in this particular case, the explicit verifications in \S\S\ref{sec:sierpinski}--\ref{sec:univalence} for the Sierpinski topos apply just as well to a gluing construction between two split categories.
  Thus $(\nSet\dn\Gm_0)\f$ is also split and the forgetful functor $(\nSet\dn\Gm_0)\f \to \sC$ is strict.
  Inspecting the proof of \autoref{thm:scone-nno} reveals that the forgetful functor preserves the \shnno and its universal property as well.
  Finally, by the analogous argument to \autoref{thm:codeqv} and the observation above that $\istrunc n (A)$ is always an h-proposition, the 0-truncation axiom of $(\nSet\dn\Gm_0)\f$ can also be chosen to be preserved strictly by the forgetful functor.

  In sum, $(\nSet\dn\Gm_0)\f$ is an object of the category of which \sC is the initial object (\autoref{thm:syn-init}), and the forgetful functor $(\nSet\dn\Gm_0)\f \to \sC$ is a morphism in that category.
  Therefore, the forgetful functor must have a strict \emph{section}, which assigns to each object of $\sC$ (that is, each type or context) a homotopy-invariant subset of its global sections.
  (This technique is due to Peter Freyd.)

  In particular the section must take the \shnno $N$ of \sC to the one constructed in \autoref{thm:scone-nno}, $\mathbb{N} \to \Gamma_0(N)$.
  Similarly, it must take each term $x:1\to N$ to a morphism from the terminal object to $\mathbb{N} \to \Gamma_0(N)$.
  But since the terminal object of $(\nSet\dn\Gm_0)$ is $1\to \Gamma_0(1)$, this means that we have a function $1\to \mathbb{N}$ which lifts $\Gamma_0(x):\Gamma_0(1) \to \Gamma_0(N)$.
  Therefore, $x\in\Gamma_0(N)$ must be in the image of $\mathbb{N}$, i.e.\ $x$ must have the homotopy class of a numeral $s^n o$.
\end{proof}

This theorem is, of course, closely related to the corresponding fact about the intuitionisic higher-order logic of elementary toposes.
The latter was Freyd's original application of this method.

To obtain a glued model with a univalent universe that contains the natural numbers, we need to glue with at least a groupoid model instead of a set model.
Thus, given a type-theoretic fibration category \sC, we define a new functor $\Gamma_1:\sC\to\nGpd$ by taking $\Gamma_1(A)$ to be the groupoid whose objects are morphisms $1\to A$ in \sC, and whose morphisms are homotopy classes of homotopies.
Note that this is invariant, up to canonical isomorphism, under the choice of path objects in \sC.

\begin{lem}\label{thm:global-sections1}
  If \sC is 1-truncated, then $\Gamma_1:\sC\to\nGpd$ is a strong fibration functor.
\end{lem}
\begin{proof}
  It obviously preserves terminal objects.
  The transport (path-lifting) property in \sC implies that it preserves fibrations.
  It also evidently preserves homotopies, hence equivalences.
  Since the acyclic cofibrations in $\nGpd$ are the injective-on-objects equivalences, and acyclic cofibrations in \sC are monic, it follows that $\Gm_1$ also preserves acyclic cofibrations.

  Thus, it remains to show that $\Gm_1$ preserves pullbacks of fibrations.
  As before, consider a fibration $p:B\fib A$ and a map $f:C\to A$, and the canonical functor
  \begin{equation}
    \Gm_1(C\times_A B) \to \Gm_1 C \times_{\Gm_1 A} \Gm_1 B.\label{eq:g1pb}
  \end{equation}
  First note that when we compose $\Gm_1$ with the set-of-objects functor $\nGpd\to\nSet$, we obtain the representable functor $\sC(1,-)$, which clearly preserves pullbacks.
  Thus,~\eqref{eq:g1pb} is bijective on objects.
  So for it to be an isomorphism, it remains to show that it is fully faithful.
  
  Suppose given $(c,b):1 \to C \times_A B$, i.e.\ morphisms $c:1\to C$ and $b:1\to B$ with $p b = c$, and likewise $(c',b')$.
  An isomorphism $(c,b) \cong (c',b')$ in $\Gm_1 C \times_{\Gm_1 A} \Gm_1 B$ is a pair $([\gamma],[\beta])$ where $\gamma:c\sim c'$ and $\beta:b\sim b'$ are homotopies and there exists a (non-specified) homotopy $\map_p \beta \sim \map_f \gamma$.
  As in \autoref{thm:global-sections0}, we choose path objects such that $\ap_p : PB \fib PA$ is a fibration, and we assume $\gamma$ and $\beta$ are specified using these path objects.
  Then by 2-dimensional path-lifting, there is a homotopy $\beta':b\sim b'$ with $\beta'\sim \beta$ and $\map_p \beta' = \map_f \gamma$.
  Thus we can use $\beta$ and $\gamma$ to build a homotopy $(c,b) \sim (c',b')$ using the path object $PC\times_{PA}PB$ for $C\times_A B$, as in \autoref{thm:global-sections0}.
  This then gives an isomorphism in $\Gm_1(C\times_A B)$ which maps onto $([\gamma],[\beta])$.
  Thus,~\eqref{eq:g1pb} is full (note that this is essentially the same as the proof in \autoref{thm:global-sections0} that~\eqref{eq:g0pb} is injective.)

  Now suppose given two isomorphisms $(c,b)\cong (c',b')$ in $\Gm_1(C\times_A B)$, which we may again take to be homotopies $\mu,\mu':(c,b)\sim (c',b')$ defined using the path object $PC\times_{PA}PB$.
  Thus, they are determined by homotopies $\beta,\beta':b\sim b'$ and $\gamma,\gamma':c\sim c'$ such that $\ap_f \gamma = \ap_p \beta$ and $\ap_f \gamma' = \ap_p \beta'$.
  Suppose furthermore that $\mu$ and $\mu'$ are identified in $\Gm_1 C \times_{\Gm_1 A} \Gm_1 B$, which is to say that $h:\beta \sim \beta'$ and $k:\gamma\sim\gamma'$.

  By playing the same trick again, we may take these homotopies to be defined using path objects $P_{B\times B} (PB)$  and $P_{C\times C} (PC)$  such that $\ap_{\ap_p} : P_{B\times B} (PB)\fib P_{A\times A} (PA)$ is a fibration.
  Then $\ap_{\ap_p} h$ and $\ap_{\ap_f} k$ may not be equal, but since $A$ is 1-truncated, they are homotopic.
  Thus, by 3-dimensional path lifting, there is a homotopy $h':\beta\sim\beta'$ with $\ap_{\ap_p} h' = \ap_{\ap_f} k$.
  Now $h'$ and $k$ induce a homotopy $\mu\sim \mu'$ using the pullback iterated path object, as before, showing that~\eqref{eq:g1pb} is faithful.
\end{proof}

Of course, we call $(\nGpd\dn \Gm_1)\f$ the \textbf{1-scone} of \sC.
It is essentially the same as the glued model of \textcite{hw:crmtt}, although univalence was not considered there.
Since \nGpd contains two nested univalent universes $\Omega \into U$ (where $U$ is defined using some inaccessible cardinal $\kappa$), we have:

\begin{cor}
  If \sC is 1-truncated and has two nested univalent universes, then $(\nGpd\dn \Gm_1)\f$ has two nested univalent universes.
  \begin{itemize}
  \item The objects in the first universe of $(\nGpd\dn \Gm_1)\f$ are objects in the first universe of \sC equipped with a homotopy-invariant subset of their set of global sections.
  \item The objects in the second universe of $(\nGpd\dn \Gm_1)\f$ are objects in the second universe of \sC equipped with a $\kappa$-small discrete fibration over their groupoid of global sections.
\end{itemize}
\end{cor}

The types in the second universe of $(\nGpd\dn \Gm_1)\f$ are basically the same as the model of \textcite{hw:crmtt}.
The whole model has the advantage that its natural numbers object lies in the second universe.
In order to show this, let us say that a \shnno is \textbf{sound} if the numerals $s^n o$ and $s^m o$ are not homotopic for distinct $n,m\in\mathbb{N}$.

For example, the \shnno of the syntactic category is sound.
This is because the unique strict functor from the syntactic category to \nSet (which, unlike $\Gm_0$, preserves the \shnno by construction) would take any homotopy $s^n o \sim s^m o $ to an equality $n=m$ in \nSet.
By contrast, the terminal category has an \shnno that is not sound.

\begin{lem}\label{thm:shnno1}
  If \sC is 1-truncated and has a sound \shnno lying in the second universe, so does $(\nGpd\dn\Gm_1)\f$.
\end{lem}
\begin{proof}
  In the presence of a universe, a standard type-theoretic argument shows that any \shnno has decidable equality, and therefore by Hedberg's theorem~\parencite{hedberg:decuip} is 0-truncated.
  Therefore, if $N$ is a \shnno and we have $t:1\to N$ and $n\in\mathbb{N}$, then there is at most one homotopy class of homotopies $s^n o \sim t$.
  Moreover, if $N$ is sound, then there can be at most one $n$ such that $s^n o \sim t$.

  We define the \shnno of $(\nGpd\dn\Gm_1)\f$ to be the monic fibration $N_1 \fib \Gm_1(N)$ which is the inclusion of the full subgroupoid of $\Gm_1(N)$ determined by those objects $t:1\to N$ which are homotopic to $s^n o$ for some $n\in\mathbb{N}$.
  This fibration is discrete (indeed, its fibers are subterminal sets), hence it lies in the second universe of $(\nGpd\dn\Gm_1)\f$.
  The morphisms $o$ and $s$ obviously restrict from $\Gm_1(N)$ to $N_1$.
  Note that by soundness and 0-truncatedness of $N$, the groupoid $N_1$ is \emph{equivalent} to the discrete groupoid $\mathbb{N}$.

  A fibration over this object in $(\nGpd\dn\Gm_1)\f$ consists of a fibration $B\fib N$ in \sC, together with a fibration of groupoids $B_1 \fib N_1 \times_{\Gm_1(N)} \Gm_1(B)$.
  This pullback $N_1 \times_{\Gm_1(N)} \Gm_1(B)$ is just the full subgroupoid of $\Gm_1(B)$ determined by those $b:1\to B$ whose image in $\Gm_1(N)$ lies in $N_1$.
  We write $B_1(b)$ for the fiber of this fibration over such a $b$.

  To give the morphism $o'$ in $(\nGpd\dn\Gm_1)\f$ consists of giving $o':1\to B$ in \sC over $o$, together with an object $o''\in B_1(o')$.
  And to give the morphism $s'$ in $(\nGpd\dn\Gm_1)\f$ consists of $s':B\to B$ in \sC over $s$, together with functors $s'':B_1(b) \to B_1 (s'b)$ which vary pseudonaturally in $b$.

  Given all these data, the induction principle in \sC yields a section $f:N\to B$ with $f o = o'$ and $f s = s' f$.
  Now we must give a compatible section $f_1 : N_1 \to B_1$ commuting with $o''$ and $s''$.
  We define $f_1(t)$ by induction on the external natural number $n$ such that $t\sim s^n o$; by soundness of $N$ there is a unique $n$ for each $t\in N_1$.

  When $n=0$, we have the particular element $o\in N_1$, and we of course define $f_1(o)=o''$.
  For all other $t$ such that $o\sim t$, there is an essentially unique homotopy $\alpha:o\sim t$, inducing a homotopy $\ap_f \alpha : o' \sim f t$, hence an isomorphism $o' \cong f t$ in $\Gm_1(B)$.
  We define $f_1(t)$ by transporting $o''$ along this isomorphism in the fibration $B_1$; these transports are specified since our fibration is structured (i.e.\ arises from a pseudofunctor into \nGpd).
  Functoriality on the $0$-component of $N_1$ is immediate.

  Now suppose $f_1$ has been defined on the $n$-component of $N_1$.
  For each $t$ such that $s^{n+1} o \sim t$, there may or may not be a $t'$ such that $t = s t'$.
  If there is, then automatically $s^n o \sim t'$, since $s$ can be proven injective in type theory.
  Moreover, there is at most one such $t'$, because the morphism $s:N\to N$ is a monomorphism for any \shnno.
  (Indeed, it is a split monomorphism, because we can define the predecessor function by recursion.)

  Now if $t=st'$, we define $f_1(t) = s''(f_1(t'))$.
  Otherwise, we define $f_1(t)$ by transporting $(s'')^{n+1} o''$ along $\ap_f \alpha$, where $\alpha$ is the essentially unique homotopy $s^{n+1} o \sim t$.
  Functoriality on the $t$ coming from a $t'$ is immediate, and since $s^{n+1} o$ is such and we have $f_1(s^{n+1}o) = (s'')^{n+1} o''$ by construction, the lifted paths give a canonical way to extend functoriality to all of the $(n+1)$-component of $N_1$.

  Finally, the requisite equations all hold by construction.
  Thus, we have constructed a \shnno in $(\nGpd\dn\Gm_1)\f$.
\end{proof}

Note that the above proof uses the law of excluded middle in the metatheory.
I do not know whether this can be avoided.

\begin{thm}\label{thm:canonicity1}
  Consider dependent type theory augmented by:
  \begin{itemize}
  \item an axiom asserting that every type is 1-truncated;
  \item the function extensionality axiom;
  \item two nested universe both satisfying the univalence axiom, and
  \item a \shnno which belongs to the second universe.
  \end{itemize}
  Then every term of type $N$ is homotopic to a numeral.
\end{thm}
\begin{proof}
  Let \sC be the syntactic category.
  As in the proof of \autoref{thm:canonicity0}, we obtain a strict section $\sC \to (\nGpd\dn\Gm_1)\f$, which assigns to each type (or context) $A$ a fibration of groupoids over $\Gm_1(A)$, and takes the \shnno $N$ to the monic fibration constructed in \autoref{thm:shnno1}.
  Thus, every closed term $x:1\to N$ lifts to a morphism from $1$ into this fibration, which implies that $x$ must lie in $N_1$, i.e.\ $x$ is homotopic to a numeral.
\end{proof}

If we had a global sections functor valued in some notion of $\infty$-groupoid, then then we could hope to extend these canonicity results to arbitrarily many univalent universes without truncation hypotheses.
However, the only notion of $\infty$-groupoids in which we currently have a model of type theory is simplicial sets, and it seems a tricky problem to produce a simplicial-set-valued global sections functor which is strictly functorial.

\begin{rmk}
  There are many other traditional applications of gluing constructions, such as the existence and disjunction properties, and parametricity theorems.
  Indeed, by inspecting the construction of dependent products and universes in the scone, one sees that the unique section of the scone of the syntactic category is precisely the unary ``logical relation'' or ``reducibility'' associated to the type theory, which is used in the traditional proofs of ``free'' parametricity theorems~\parencite{wadler:free-theorems}.
  Thus, many such theorems can also be extended to type theory with univalence.
  Moreover, parametricity theorems which arise from other sorts of logical relations can be similarly obtained from other oplax limits; e.g.\ binary logical relations arise from oplax limits over the category $(\cdot \ot \cdot \to \cdot)$.
  Oplax limits over inverse diagrams of ordinal rank $>2$ can thus be regarded as a ``higher'' sort of logical relations.
  I do not know whether they imply ``higher'' notions of canonicity and parametricity.
\end{rmk}

\printbibliography

\end{document}